\documentclass[11pt]{article}

\usepackage{macros}
\usepackage[numbers]{natbib}
\usepackage{color}
\definecolor{darkblue}{rgb}{0,0,.5}
\usepackage[colorlinks=true,linkcolor=darkblue,%
  citecolor=darkblue,urlcolor=darkblue]{hyperref}
\usepackage[top=2.54cm,left=2.54cm,right=2.54cm,bottom=2.54cm]{geometry}
\usepackage{algorithm}
\usepackage[noend]{algorithmic}
\usepackage{overpic}

\newcommand{\Pfamily}{\mathcal{P}}
\newcommand{\xdomain}{\mathcal{X}}

\newcommand{\id}{{\rm id}}

\newcommand{\para}{\textup{par}}
\newcommand{\nonpara}{\textup{non-par}}

\begin{document}

\begin{center}
  {\Large The Right Complexity Measure in Locally Private Estimation: \\
    \vspace{.1cm}
    It is not the Fisher Information} \\
  \vspace{.3cm}
  \large{John C.\ Duchi ~~~~ Feng Ruan \\
    Stanford University} \\
  \vspace{.2cm}
  \large{September 2020}
\end{center}

\begin{abstract}

We identify fundamental tradeoffs between statistical utility and privacy
under local models of privacy in which data is kept private even from the
statistician, providing instance-specific bounds for private estimation and
learning problems by developing the \emph{local minimax risk}. In contrast
to approaches based on worst-case (minimax) error, which are
conservative, this allows us to evaluate the difficulty of individual
problem instances and delineate the possibilities for adaptation
in private estimation and inference.  Our main results
show that the local modulus of continuity of the estimand with respect to
the variation distance---as opposed to the Hellinger distance central to
classical statistics---characterizes rates of convergence under locally
private estimation for many notions of privacy, including differential
privacy and its relaxations. As consequences of these results, we identify
an alternative to the Fisher information for private estimation, giving a
more nuanced understanding of the challenges of adaptivity and optimality.

\end{abstract}


\newcommand{\losssq}{L_{\textup{sq}}}
\newcommand{\losssqv}{L_{\textup{sq},v}}

\section{Introduction}

The increasing collection of data at large scale---medical records, location
information from cell phones, internet browsing history---points to the
importance of a deeper understanding of the tradeoffs inherent between
privacy and the utility of using the data collected.  Classical mechanisms
for preserving privacy, such as permutation, small noise addition, releasing
only mean information, or basic anonymization are insufficient, and notable
privacy compromises with genomic data~\cite{HomerSzReDuTeMuPeStNeCr08} and
movie rating information~\cite{NarayananSh08} have caused the NIH to
temporarily stop releasing genetic information and Netflix to cancel a
proposed competition for predicting movie ratings. Balancing the tension
between utility and the risk of disclosure of sensitive information
is thus essential.

In response to these challenges, researchers in the statistics, databases,
and computer science communities have studied \emph{differential
  privacy}~\cite{Warner65, EvfimievskiGeSr03, DworkMcNiSm06, DworkSm09,
  HardtTa10, DuchiJoWa18, DuchiRo19} as a formalization of disclosure risk
limitation.  This literature discusses two notions of privacy: \emph{local}
privacy, in which data is privatized before it is even shared with a data
collector, and \emph{central} privacy, where a centralized curator maintains
the sample and guarantees that any information it releases is appropriately
private. The local model is stronger and entails some necessary loss of
statistical efficiency, yet its strong privacy protections encourage its
adoption. Whether for ease of regulatory compliance, for example with
European Union privacy rules~\cite{EU18}; for transparency and belief in the
importance of privacy; or to avoid risks proximate to holding sensitive
data, like hacking or subpoena risk; major technology companies have adopted
local differential privacy protections in their data collection and machine
learning tools. Apple provides local differential privacy in many of its
iPhone systems~\cite{ApplePrivacy17}, and Google has built systems supplying
central and local differential
privacy~\cite{ErlingssonPiKo14,AbadiChGoMcMiTaZh16}.  The broad impact of
privacy protections in billions of devices suggest we should carefully
understand the fundamental limitations and possibilities of learning with
local notions of privacy.

To address this challenge, we borrow from \citet{CaiLo15} to study the
\emph{local minimax complexity} of estimation and learning under local
privacy.  Worst-case notions of complexity may be too stringent for
statistical practice, and we wish to understand how difficult the
\emph{actual} problem we have is and whether we can adapt to this problem
difficulty, so that our procedures more efficiently solve easy problems---as
opposed to being tuned to worst-case scenarios. Our adoption of local
minimax complexity is thus driven by three desiderata, which \citet{CaiLo15}
identify: we seek fundamental limits on estimation and learning that (i) are
instance specific, applying to the particular problem at hand, (ii) are
(uniformly) attainable, in that there exist procedures to achieve the
instance-specific difficulty, and (iii) have super-efficiency limitations,
so that if a procedure achieves \emph{better} behavior than the lower bounds
suggest is possible, there should be problem instances in which the
procedure must have substantially worse behavior.  We provide characterize
the local minimax complexity of locally private estimation of
one-dimensional quantities, showing that this benchmark (nearly) always
satisfies desiderata (i) and (iii). Via a series of examples---some
specific, others general---we show that there are procedures whose risk is
of the order of the local minimax risk for all underlying (unknown)
populations.  As an essential part of this program is that the complexity is
(ii) attainable---which, to our knowledge, remains open even in the
non-private case---we view this paper as an initial foray into understanding
problem-specific optimality in locally private estimation.



\subsection{Contributions, outline, and related work}

Our development of instance-specific complexity notions
under privacy constraints allows us to quantify the
statistical price of privacy. Identifying the tension here is of course of
substantial interest, and \citet{DuchiJoWa18,DuchiJoWa13_focs} develop a
set of statistical and information-theoretic tools for understanding the
\emph{minimax} risk in locally differentially private settings, providing
the point of departure for our work.  To understand their and our coming
approach, we formalize our setting.

We have i.i.d.\ data $X_1, \ldots, X_n$ drawn according to a distribution
$P$ on a space $\mc{X}$. Instead of observing the original sample $\{X_i\}$,
however, the statistician or learner sees only \emph{privatized} data
$\{Z_i\}$, where $Z_i$ is drawn from a Markov kernel $\channel(\cdot \mid
X_i)$ conditional on $X_i$ (following information-theory, we
call $\channel$ the privacy \emph{channel}~\cite{CoverTh06}). We
allow the channel to be \emph{sequentially
  interactive}~\cite{DuchiJoWa18}, meaning that $Z_i$
may depend on the previous (private) observations $Z_1, \ldots, Z_{i
  - 1}$, i.e.\
\begin{equation}
  \label{eqn:sequential-interactive}
  Z_i \mid X_i = x, Z_1, \ldots, Z_{i-1}
  \sim \channel(\cdot \mid x, Z_{1:i-1}).
\end{equation}
This notion of interactivity is important for procedures, such as stochastic
gradient methods~\cite{DuchiJoWa18} or the one-step-corrected estimators we
develop in the sequel, which modify the mechanism after some number of
observations to more accurately perform inference.

The statistical problems we consider are, abstractly, as follows.  Let
$\mc{P}$ be a family of distributions, and let $\theta: \mc{P} \to \Theta
\subset \R^d$
be a parameter we wish to estimate and belonging to $\Theta$,
where $\theta(P)$ denotes the target parameter. Let
$\loss : \R^d \to \R_+$ be a symmetric quasiconvex loss, where we assume that
$\loss(\zeros) = 0$.  A typical example
is the mean $\theta(P) = \E_P[X]$ with squared error
$\loss(\theta - \theta(P)) = (\theta - \E_P[X])^2$.
Let $\channeldistset$ be a collection of private
channels, for example, $\diffp$-differentially private channels (which
we define in the sequel). The \emph{private minimax risk}~\cite{DuchiJoWa18}
is
\begin{equation}
  \label{eqn:private-minimax}
  \minimax_n(\loss, \mc{P}, \channeldistset)
  \defeq \inf_{\what{\theta}, \channel \in \channeldistset}
  \sup_{P \in \mc{P}}
  \E_{\channel \circ P}\left[\loss(\what{\theta}(Z_1, \ldots, Z_n)
    - \theta(P))\right]
\end{equation}
where $\channel \circ P$ denotes
the marginal $X_i \sim P$ and $Z_i$ drawn
conditionally~\eqref{eqn:sequential-interactive}.  \citet{DuchiJoWa18}
provide upper and lower bounds on this quantity when $\channeldistset$ is
the collection of $\diffp$-locally differentially private channels,
developing strong data processing inequalities to quantify the costs of
privacy.

The worst-case nature of the formulation~\eqref{eqn:private-minimax} gives
lower bounds that may be too pessimistic for practice, and it prohibits a
characterization of problem-specific difficulty.
Accordingly, we adopt a
\emph{local minimax approach}, which builds out of the classical statistical
literature on hardest one-dimensional alternatives that begins with
Stein~\cite{Stein56a, Birge83, DonohoLi87,DonohoLi91a,DonohoLi91b, CaiLo15,
  ChatterjeeDuLaZh16}.  To that end, we define the
\emph{local minimax risk} at the distribution $P_0$ for the set of channels
$\channeldistset$ as
\begin{equation}
  \label{eqn:private-local-minimax}
  \localminimax_n(P_0, \loss, \mc{P}, \channeldistset)
  \defeq \sup_{P_1 \in \mc{P}}
  \inf_{\what{\theta},
    \channel \in \channeldistset}
  \max_{P \in \{P_0, P_1\}}
  \E_{Q \circ P}\left[\loss(\what{\theta}(Z_1, \ldots, Z_n) - \theta(P))\right]. 
\end{equation}
The quantity~\eqref{eqn:private-local-minimax} measures the
difficulty of the loss minimization problem for a \emph{particular
  distribution} $P_0$ under the privacy constraints $\channeldistset$
characterizes, and at this distinguished distribution, we look for the
hardest alternative distribution $P_1 \in \mc{P}$. As we shall see,
the definition~\eqref{eqn:private-local-minimax} indeed becomes local,
if $P_1$ is far from $P_0$, then it is easy to develop an estimator
$\what{\theta}$ distinguishing $P_0$ and $P_1$, so that (for large
$n$) the supremum is essentially constrained to a neighborhood of $P_0$.

To situate our contributions, let us first consider the non-private
local minimax
complexity, when $\channeldistset =
\{\id\}$ (the identity mapping).
Throughout, we will use the shorthand
\begin{equation*}
  \localminimax_n(P_0, \loss, \mc{P}) \defeq
  \localminimax_n(P_0, \loss, \mc{P}, \{\textup{id}\})
\end{equation*}
for the non-private local minimax risk.
We wish to estimate a linear
function $v^T \theta$ of the parameter $\theta$
with (projected) square loss
$\losssqv(t) = (v^T t)^2$. In the
classical setting of a parametric family $\mc{P} =
\{P_\theta\}_{\theta \in \Theta}$ with Fisher information matrix
$\fisher[\theta]$, then (as we describe more formally in
Section~\ref{sec:classical-local-minimax}) the Fisher information bound for
the parameter $\theta_0$ is
\begin{equation}
  \label{eqn:classic-local-minimax}
  \localminimax_n(P_{\theta_0}, \losssqv, \mc{P}, \{\id\})
  \asymp \frac{1}{n} \E\left[\left(v^T Z\right)^2\right]
  ~~ \mbox{for} ~~
  Z \sim \normal\big(0, \fisher[\theta_0]^{-1}\big),
\end{equation}
where $\asymp$ denotes equality to within numerical constants.  More
generally, if we wish to estimate a functional $\theta(P) \in \R$ of $P$,
\citet{DonohoLi87,DonohoLi91a,DonohoLi91b} show how the \emph{modulus of
  continuity} takes the place of the classical information bound. Again
considering the squared error $\losssq(t) = t^2$,
define the Hellinger modulus of continuity of $\theta(\cdot)$ at $P_0 \in
\mc{P}$ by
\begin{equation}
  \label{eqn:hellinger-modcont}
  \modcont[{\rm hel}](\delta; P_0, \mc{P})
  \defeq \sup_{P_1 \in \mc{P}}
  \left\{|\theta(P_0) - \theta(P_1)| ~ \mbox{s.t.}~ P_1 \in \mc{P},
  \dhel(P_0, P_1) \le \delta \right\}
\end{equation}
where $\dhel^2(P_0, P_1) = \half \int (\sqrt{dP_0} - \sqrt{dP_1})^2$.  In
the local minimax case, characterizations via a local modulus are available
in some problems~\cite{CaiLo15, ChatterjeeDuLaZh16},
where $\localminimax_n(P_0, \losssq, \mc{P}) \asymp
\helmod^2(n^{-1/2}; P_0, \mc{P})$, while
under mild regularity conditions, the \emph{global modulus} $\sup_{P \in
  \mc{P}} \helmod(\delta; P, \mc{P})$ governs non-private global
minimax risk:
(often) one has $\minimax_n(\losssq, \mc{P}) \asymp \sup_{P_0 \in \mc{P}}
\helmod(n^{-1/2}; P_0, \mc{P})$~\cite{Birge83, DonohoLi87, DonohoLi91a,
  DonohoLi91b}.


In contrast, the work of \citet{DuchiJoWa18,DuchiJoWa13_focs} suggests that
for $\diffp$-locally differentially private estimation, we should replace
the Hellinger distance by \emph{variation distance}.  In the case of
higher-dimensional problems, there are additional dimension-dependent
penalties in estimation that local differential privacy makes unavoidable,
at least in a minimax sense~\cite{DuchiJoWa18}. In work independent of and
contemporaneous to our own, \citet{RohdeSt18} build off
of~\cite{DuchiJoWa18} to show that (non-local) minimax rates of convergence
under $\diffp$-local differential privacy are frequently governed by a
global modulus of continuity, except that the
variation distance $\tvnorm{P_0 - P_1} = \sup_A |P_0(A) - P_1(A)|$ replaces
the Hellinger distance $\dhel$. They also exhibit a
mechanism that is minimax optimal for ``nearly'' linear functionals based on
randomized response~\cite[Sec.~4]{Warner65,RohdeSt18}. Thus, locally
differentially private procedures give rise to a different geometry than
classical statistical problems.

We are now in a position for a high-level description of our results, which
apply in a variety of locally private estimation settings consisting of
weakenings of $\diffp$-differential privacy, whose definitions we formalize
in Section~\ref{sec:definition-of-privacy}. We provide a precise
characterization of the local minimax
complexity~\eqref{eqn:private-local-minimax} in these settings.  If we
define the local modulus of continuity at $P_0$ by
\begin{equation*}
  \tvmod(\delta; P_0, \mc{P})
  \defeq \sup_{P \in \mc{P}} \left\{|\theta(P_0) - \theta(P)|
  ~\mbox{s.t.}~ \tvnorm{P - P_0} \le \delta \right\},
\end{equation*}
then a consequence
of Theorem~\ref{theorem:modulus-of-continuity} is that for
the squared loss and
$\diffp$-locally private channels $\channeldistset_\diffp$,
\begin{equation*}
  \localminimax_n(P_0, \losssq, \mc{P}, \channeldistset_\diffp)
  \asymp \tvmod^2\left((n\diffp^2)^{-1/2}; P_0, \mc{P}\right).
\end{equation*}
We provide this characterization in more detail and for general losses in
Section~\ref{sec:local-minimax}.  Moreover, we show a super-efficiency
result that any procedure that achieves risk better than the local minimax
complexity at a distribution $P_0$ \emph{must} suffer higher risk at another
distribution $P_1$, so that this characterization does indeed satisfy our
desiderata of an instance-specific complexity measure.

The departure of these risk bounds from the typical Hellinger
modulus~\eqref{eqn:hellinger-modcont} has consequences for locally private
estimation and adaptivity of estimators, which we address for parametric
problems and examples in Section~\ref{sec:examples} and for general
estimation in Section~\ref{sec:general-influence-functions}.  Instead of the
Fisher information, an alternative we term the \emph{$\bigLone$-information}
characterizes the complexity of locally private estimation.  A challenging
consequence of these results is that, for some parametric models (including
Bernoulli estimation and binomial logistic regression), the local
complexity~\eqref{eqn:private-local-minimax} is \emph{independent} of the
underlying parameter: nominally easy problems (in the Fisher information
sense) are not so easy under local privacy constraints.  Our proofs rely on novel
Markov contraction inequalities for divergence measures, which strengthen
classical strong data processing
inequalities~\cite{CohenKeZb98,DelMoralLeMi03, DuchiJoWa18}.

Developing procedures achieving the local minimax
risk~\eqref{eqn:private-local-minimax} is challenging, but we show that
locally uniform convergence is asymptotically possible in a number of cases
in Sections~\ref{sec:examples} and~\ref{sec:general-influence-functions},
including well- and mis-specified exponential family models, using
stochastic gradient methods or one-step corrected estimators. An important
point of our results (Sec.~\ref{sec:adaptation-of-models}) is that the local
private minimax risk---sometimes in distinction from the non-private
case---depends strongly on the assumed family $\mc{P}$, making the
development of private adaptive estimators challenging.  We use a protein
expression-prediction problem in Section~\ref{sec:experiments} to compare
our locally optimal procedures with minimax optimal
procedures~\cite{DuchiJoWa18}; the experimental results suggests that the
locally optimal procedures outperform global minimax procedures, though
costs of privacy still exist.


\paragraph{Notation:}
We use a precise big-O notation throughout the paper, where for functions
$f, g : \mc{X} \to \R_+$, $g(x) = O(f(x))$ means that there exists a
numerical (universal) constant $C < \infty$ such that $g(x) \le C f(x)$; we
use $g(x) \lesssim f(x)$ to mean the same. We write $O_t(\cdot)$ when
the constant $C$ may depend on an auxiliary parameter $t$. We write
$g(x) \asymp f(x)$ if both $g(x) \lesssim f(x)$ and $f(x) \lesssim g(x)$.
If $g(x) = o(f(x))$ as $x \to x_0$, we mean that $\limsup_{x \to x_0} g(x) /
f(x) = 0$.  We let $\lpP$ be the collection of $g : \mc{X} \to \R^d$ with
$\int \norm{g(x)}^p dP(x) < \infty$, where $p = \infty$ is the set of
essentially bounded $g$; the dimension $d$ is tacit.  For a sequence of
distributions $P_n$, we write convergence in distribution $X_n \cd_{P_n} X$
to mean that for any bounded continuous $f$, $\E_{P_n}[f(X_n)] \to
\E[f(X)]$.


\section{Preliminaries}
\label{sec:preliminaries}

We begin in
Section~\ref{sec:definition-of-privacy} with definitions and some brief
discussion of the definitions of privacy we consider. To help
situate our approach, we discuss local
minimax complexity without privacy in
Section~\ref{sec:classical-local-minimax}.  There are several
plausible notions of attainment of the local minimax risk---all related to
desideratum (ii) in the introduction that the risk be achievable---so we
conclude in Section~\ref{sec:attainment-sortof} by giving several related
results, including an asymptotic and locally uniform convergence guarantee
that will be what we typically demonstrate for our procedures.  In spite of
the (sometimes) asymptotic focus, which builds out of Le Cam's quadratic
mean differentiability theory and various notions of efficiency in
semiparametric models~\cite{LeCam86, LeCamYa00, Pollard97, VanDerVaart98,
  BickelKlRiWe98}, we will typically achieve optimality only to within
numerical constants---getting sharp constants appears challenging when we
allow arbitrary privatization schemes and sequential
interactivity~\eqref{eqn:sequential-interactive}.


\subsection{Definitions of Local Privacy}
\label{sec:definition-of-privacy}

With the
notion~\eqref{eqn:sequential-interactive} of sequentially interactive
channels, where the $i$th private observation is drawn conditionally on the
past as $Z_i \mid X_i = x, Z_1, \ldots, Z_{i-1} \sim \channel(\cdot \mid x,
Z_{1:i-1})$, we consider several privacy definitions.
First is \emph{local differential privacy}, which
Warner~\cite{Warner65} proposes (implicitly) in his 1965 work on survey
sampling, and which \citet{EvfimievskiGeSr03} and \citet{DworkMcNiSm06} make
explicit.

\begin{definition}
  \label{definition:local-dp}
  The channel $\channel$ is \emph{$\diffp$-locally differentially private}
  if for all $i \in \N$, $x, x' \in \mc{X}$, and
  $z_{1:i-1} \in \mc{Z}^{i-1}$,
  \begin{equation*}
    \sup_{A \in \sigma(\mc{Z})}
    \frac{\channel(A \mid x, z_{1:i-1})}{\channel(A \mid x', z_{1:i-1})}
    \le e^\diffp.
  \end{equation*}
  The channel $\channel$ is \emph{non-interactive} if
  for all $z_{1:i-1} \in \mc{Z}^{i-1}$ and $A \in \sigma(\mc{Z})$,
  \begin{equation*}
    \channel(A \mid x, z_{1:i-1}) = \channel(A \mid x)
  \end{equation*}
\end{definition}

\citet{DuchiJoWa18} consider this notion of privacy, developing its
consequences for minimax optimal estimation. An equivalent
view~\cite{WassermanZh10} is that an adversary knowing the data is either
$x$ or $x'$ cannot accurately test, even conditional on the output $Z$,
whether the generating data was $x$ or $x$' (the sum of
Type I and II errors is at least $\frac{1}{1 + e^\diffp}$). To mitigate the
consequent difficulties for estimation and learning with differentially
private procedures, researchers have proposed weakenings of
Definition~\ref{definition:local-dp}, which we also consider.\footnote{We
  ignore $(\diffp, \delta)$-approximate differential privacy,
  as for locally private estimation, it is essentially equivalent
  to $\diffp$-differential privacy~\cite[e.g.][Appendix~D.1]{DuchiRo19}.}
These
repose on $\alpha$-R\'{e}nyi-divergences, defined for $\alpha \ge 1$ by
\begin{equation*}
  \rendiv{P}{Q}
  \defeq \frac{1}{\alpha - 1}
  \log \int \left(\frac{dP}{dQ}\right)^\alpha dQ.
\end{equation*}
For $\alpha = 1$ one takes the limit $\alpha \downarrow
1$, yielding $\rendiv{P}{Q} = \dkl{P}{Q}$, and
for $\alpha = \infty$ one has $\rendiv{P}{Q} = \esssup \log \frac{dP}{dQ}$.
\citet{Mironov17} then proposes the following definition:
\begin{definition}
  \label{definition:local-renyi}
  The channel $\channel$ is \emph{$(\alpha, \diffp)$-R\'{e}nyi locally
    differentially private (RDP)} if for all $x, x' \in \mc{X}$, and $z_{1:i-1}
  \in \mc{Z}$, we have
  \begin{equation*}
    \rendiv{\channel(\cdot \mid x, z_{1:i-1})}{
      \channel(\cdot \mid x', z_{1:i-1})}
    \le \diffp.
  \end{equation*}
\end{definition}
\noindent
This definition simplifies \emph{concentrated differential
  privacy}~\cite{DworkRo16, BunSt16} by requiring that it hold only for a
single fixed $\alpha$, and it has allowed
effective private methods for large scale machine
learning~\cite{AbadiChGoMcMiTaZh16}.


\newcommand{\prior}{\pi}

The choice $\alpha = 2$ in Definition~\ref{definition:local-renyi}
is salient and important in our
analysis. Consider a prior on points $x, x'$, represented
by $\prior(x) \in [0,1]$ and $\prior(x') = 1 - \prior(x)$, and
the posterior $\prior(x \mid Z)$ and $\prior(x' \mid Z)$ after observing the
private quantity $Z \sim \channel(\cdot \mid x)$. Then $(2,
\diffp)$-R\'{e}nyi privacy is equivalent~\cite[Sec.~VII]{Mironov17} to the
the prior and posterior odds of $x$ against $x'$ being
close in expectation:
\begin{equation*}
  \E\left[\frac{\prior(x \mid Z) / \prior(x' \mid Z)}{
      \prior(x) / \prior(x')} \mid x \right] \le e^\diffp
\end{equation*}
for all two-point priors $\prior$, where the expectation is taken
over $Z \mid x$. (For $\diffp$-differential
privacy, the inequality holds for \emph{all} $Z$ without expectation).
As R\'{e}nyi divergences are monotonic in $\alpha$
(cf.~\cite[Thm.~3]{ErvenHa14}), any
$(\alpha, \diffp)$-R\'{e}nyi private channel
is $(\alpha', \diffp)$-R\'{e}nyi private for $\alpha' \le \alpha$.
Thus, any lower bound we prove on estimation for
$(\alpha = 2, \diffp$)-local RDP implies an identical lower bound
for $\alpha' \ge 2$.

The definitions provide varying levels of privacy. It is immediate that if a
channel is $\diffp$-differentially private, then it is $(\alpha,
\diffp)$-R\'{e}nyi locally private for any $\alpha$. More sophisticated
bounds are possible. Most importantly,
$\diffp$-differential privacy (Definition~\ref{definition:local-dp}) implies
$(\alpha, 2 \alpha \diffp^2)$-R\'{e}nyi differential privacy
(Definition~\ref{definition:local-renyi}) for all $\alpha \ge 1$.  For
$\alpha = 2$, we can tighten this to $(2, \min\{\frac{3}{2} \diffp^2, 2
\diffp\})$-RDP. We therefore write our lower bounds to apply for
$(2, \diffp^2)$-R\'{e}nyi differentially private channels; this implies
lower bounds for all $(\alpha, \diffp^2)$-RDP channels, and
(as differential privacy is stronger than R\'{e}nyi privacy) implies
lower bounds for any $\diffp$-locally differentially private channels.

\subsection{A primer on local minimax complexity}
\label{sec:classical-local-minimax}

We briefly review local minimax complexity to give intuition for and
motivate our approach. The starting point is \citet{Stein56a}, who considers
estimating a nonparametric functional $\theta(P)$, proposing that the
``information'' about $\theta$ at $P_0$ should be the least Fisher
information over all one-dimensional subfamilies of distributions that
include $P_0$, leading to the local minimax
risk~\eqref{eqn:private-local-minimax} with $\channeldistset =
\{\textup{id}\}$.  Specializing to the squared error $\losssq$, in the
non-private case, one then defines
\begin{equation}
  \label{eqn:classical-local-minimax}
  \localminimax_n(P_0, \losssq, \mc{P})
  \defeq \sup_{P_1 \in \mc{P}} \inf_{\what{\theta}}
  \max_{P \in \{P_0, P_1\}} \E_P\left[(\what{\theta} - \theta(P))^2\right].
\end{equation}
Then the Hellinger modulus~\eqref{eqn:hellinger-modcont}
typically characterizes the local minimax
risk~\eqref{eqn:classical-local-minimax} to numerical
constants~\cite{DonohoLi91a, CaiLo15}, as the next
proposition shows (we include a proof for
completeness in Appendix~\ref{sec:proof-classical-local-minimax}).
\begin{proposition}
  \label{proposition:classical-local-minimax}
  For each $n \in \N$ and any $P_0 \in \mc{P}$,
  \begin{equation*}
    \frac{\sqrt{2} - 1}{8 \sqrt{2}}
    \helmod^2(n^{-1/2} / 2; P_0, \mc{P})
    \, \le \,
    \localminimax_n(P_0, \losssq, \mc{P})
    \, \le \, \sup_{r \ge 0} \left\{\helmod^2(r; P_0, \mc{P})
    \exp(-n r^2)\right\}.
  \end{equation*}
\end{proposition}
\noindent
Whenever the modulus of continuity behaves nicely,
the upper bound
shows that the lower is tight to within constant factors.
For example, under a polynomial growth assumption that
there exist $B, \beta < \infty$ such that $\helmod(c \delta; P_0, \mc{P})
\le B c^\beta \helmod(\delta; P_0, \mc{P})$ for all $c > 1$, then
\begin{equation}
  \label{eqn:local-minimax-via-hellinger}
  \localminimax_n(P_0, \mc{P})
  \le (B \beta^{\beta/2} e^{-\beta/2}) \cdot
  \helmod^2\left(n^{-1/2} / 2; P_0, \mc{P}\right)
\end{equation}
(cf.\ Appendix~\ref{sec:proof-classical-local-minimax}).
The global modulus of continuity of the parameter $\theta(P)$ with
respect to Hellinger distance also characterizes global minimax error for
estimation of linear functionals on convex spaces of
distributions~\cite{Birge83, DonohoLi87, DonohoLi91a}
and gives lower bounds generically.

These calculations are abstract, so it is instructive to specialize
to more familiar families, where we recover the information
bound~\eqref{eqn:classic-local-minimax}.  Consider a parametric family of
distributions $\mc{P} \defeq \{P_\theta\}_{\theta \in \Theta}$, $\Theta
\subset \R^d$, with dominating measure $\mu$. We
assume $\mc{P}$ is \emph{quadratic mean differentiable} (QMD) at
$\theta$~\cite[Ch.~7.1]{VanDerVaart98}, meaning there exists a score
$\score_\theta : \mc{X} \to \R^d$ such that
\begin{equation}
  \label{eqn:qmd-def}
  \int \Big(\sqrt{p_{\theta + h}} - \sqrt{p_{\theta}} - \half h^T \score_\theta
  \sqrt{p_{\theta}}\Big)^2 d\mu = o(\norm{h}^2)
\end{equation}
as $h \to 0$. Most classical families of distributions
(e.g.\ exponential families) are QMD with the familiar score
$\score_\theta(x) = \nabla_\theta \log p_\theta(x)$
(cf.~\cite{LehmannRo05,VanDerVaart98}). The Fisher
information $\fisher[\theta] = \int \score_\theta \score_\theta^T p_\theta d\mu
\in \R^{d \times d}$ then exists, and we have the asymptotic expansion
\begin{equation}
  \dhel^2(P_{\theta + h}, P_\theta)
  = \frac{1}{8} h^T \fisher[\theta] h + o(\norm{h}^2).
  \label{eqn:qmd-gives-hellinger}
\end{equation}
When the parameter $\theta$ is identifiable, the
local minimax risk~\eqref{eqn:classical-local-minimax} coincides
with the standard Fisher information bounds to within
numerical constants. Indeed, consider the following identifiability
\begin{assumption}
  \label{assumption:identifiability}
  For $\delta > 0$, there exists $\gamma > 0$ such that
  $\norm{\theta - \theta_0} > \delta$ implies
  $\dhel^2(P_\theta, P_{\theta_0}) > \gamma$.
\end{assumption}
\noindent
We can then make the
approximation~\eqref{eqn:classic-local-minimax} for
estimating $v^T\theta_0$ rigorous
(see Appendix~\ref{sec:proof-qmd-bounded-local-minimax}):
\begin{claim}
  \label{claim:qmd-bounded-local-minimax}
  Let $\mc{P} = \{P_\theta\}_{\theta \in \Theta}$ be quadratic mean
  differentiable at $\theta_0$ with positive definite Fisher information
  $\fisher[\theta_0]$, assume that $\Theta$ is bounded, and that $\theta_0$
  is identifiable (\ref{assumption:identifiability}).  Then for large
  $n \in \N$,
  \begin{equation*}
    \frac{1}{21}
    \cdot \frac{1}{n} v^T \fisher[\theta_0]^{-1} v
    \le \localminimax_n(P_{\theta_0}, \losssqv, \mc{P})
    \le \frac{9}{e} \cdot \frac{1}{n} v^T \fisher[\theta_0]^{-1} v.
  \end{equation*}
\end{claim}
\noindent
We cannot expect to achieve the correct numerical constants with the
two-point lower bounds in the local minimax risk~\cite{CasellaSt81}, but
Claim~\ref{claim:qmd-bounded-local-minimax} recovers the correct
scaling in problem parameters.

\subsection{Measuring attainment of the local minimax risk}
\label{sec:attainment-sortof}

As we note in the introduction, we would like a procedure that uniformly
achieves the local minimax benchmark, that is, for a given loss $L$,
returning to the more general notation~\eqref{eqn:private-local-minimax}, we
would like
\begin{equation*}
  \sup_{Q \circ P_0 \in \mc{P}}
  \frac{\E_{P_0}[\loss(\what{\theta}_n - \theta(P_0))]}{\localminimax_n(P_0,
    L, \mc{P}, \channeldistset)}
  \lesssim 1.
\end{equation*}
Achieving this generally is challenging (and for many families $\mc{P}$,
impossible~\cite{BickelKlRiWe98}); indeed, a major contribution of
\citet{CaiLo15} is to show that it \emph{is} possible to achieve this
uniform benchmark for the squared error and various functionals in
convexity-constrained nonparametric regression.

As a consequence, we often consider a weakening to achieve the local minimax
risk (to within numerical constants). We describe this precisely at the end
of this section, first reviewing some of the necessary parametric and
semi-parametric theory~\cite{VanDerVaart98, BickelKlRiWe98}. In parametric
cases, $\mc{P} = \{P_\theta\}_{\theta \in \Theta}$, we consider
sequences of scaled losses, taking the form
$\loss_n(\what{\theta}_n - \theta(P_0)) = \loss(\sqrt{n}(\what{\theta}_n -
\theta(P_0)))$.
An estimator $\what{\theta}_n$ is \emph{asymptotically local minimax rate
  optimal} if
\begin{equation}
  \label{eqn:asymptotic-local-benchmark}
  \sup_c \limsup_{n \to \infty}
  \sup_{\norm{h} \le c / \sqrt{n}}
  \frac{\E_{P_{\theta_0 + h}}[\loss(\sqrt{n}(\what{\theta}_n -
      (\theta_0 + h)))]}{
    \localminimax_n(P_{\theta_0}, \loss_n, \mc{P}, \{\textup{id}\})}
  \lesssim 1
\end{equation}
for all $\theta_0 \in \interior \Theta$.  Le Cam's local asymptotic
normality theory of course allows much more, even achieving correct
constants~\cite{LeCam86, LeCamYa00, Pollard97, VanDerVaart98}.  We
emphasize that while many of our ideas build out of semiparametric
efficiency, typically we only achieve optimality to within numerical
constants.


We will generally demonstrate procedures that achieve the (private) local
minimax risk in some locally uniform sense, and with this in mind, we review
a few necessary concepts in semi-parametric estimation on regularity,
sub-models, and tangent spaces~\cite[cf.][Chapters~8.5 \&
  25.3]{VanDerVaart98} that will be important for developing our
asymptotics.  Let $\mc{P}$ be a collection of distributions, and for some
$P_0 \in \mc{P}$ let $\Pmodel \defeq \{P_h\}_{h \in \R^d} \subset \mc{P}$ be
a sub-model within $\mc{P}$ indexed by $h \in \R^d$, where we assume that
$\Pmodel$ is quadratic mean differentiable (QMD)~\eqref{eqn:qmd-def} at
$P_0$ for a score function $g : \mc{X} \to \R^d$ (usually this score will
simply be $g(x) = \left.\nabla_h \log dP_h(x) \right|_{h =
  0}$)~\cite[Ch.~25.3]{VanDerVaart98}, that is,
\begin{equation}
  \label{eqn:l2-diff}
  \int \left|dP_h^{1/2} - dP_0^{1/2} - \half h^T g dP_0^{1/2}\right|^2 =
  o(\norm{h}^2)
\end{equation}
as $h \to 0$. Considering different QMD sub-models
$h \mapsto P_h$ around $P_0$ yields the tangent set $\tangentset$, which is
a collection of mean-zero score functions $g : \mc{X} \to \R^d$ with $g \in
\ltwoPz$.  Then a parameter $\theta : \mc{P} \to \R^k$ is
\emph{differentiable relative to $\tangentset$} if there exists a mean-zero
\emph{influence function} $\influencefunc : \mc{X} \to \R^k$,
where for each submodel $\Pmodel = \{P_h\}_{h \in \R^d}$ and associated
score $g : \mc{X} \to \R^d$,\footnote{Recalling
  \cite[Ch.~25.3]{VanDerVaart98} and
  the Riesz representation theorem, the existence of this influence function
  is equivalent to the exists of a continuous linear map $\varphi : \ltwoPz \to
  \R^k$ such that $\theta(P_h) - \theta(P_0) = \varphi(h^T g) + o(\norm{h})$.}
\begin{equation}
  \theta(P_h) = \theta(P_0)
  + \int \influencefunc(x) \<g(x), h\> dP_0(x) + o(\norm{h}).
  \label{eqn:differentiable-functional}
\end{equation}

We turn now away from properties of the parameter $\theta$
to properties of estimators that will be useful.
An estimator $\what{\theta}_n$ is \emph{regular for $\theta$}
at $P_0$ if for all $h$ and sequences $h_n \to h \in \R^d$,
\begin{equation*}
  \sqrt{n}(\what{\theta}_n - \theta(P_{h_n/\sqrt{n}}))
  \mathop{\cdlong}_{P^n_{h_n/\sqrt{n}}} Z
\end{equation*}
for a random variable $Z$ (which is identical for each $h$); such
estimators are classically central~\cite{VanDerVaart98}.
In our constructions, the (private) estimators $\what{\theta}_n$ depend both
on the variables $X_i$ and, as we construct $Z_i \sim \channel(\cdot \mid
X_i, Z_{1:i-1})$, we can assume w.l.o.g.\ that there is an independent
sequence of auxiliary random variables $\auxvar_i \simiid \auxdist$ such
that $\what{\theta}_n = \what{\theta}_n(X_{1:n}, \auxvar_{1:n})$.  Then
under the sampling distribution $P_0 \times \auxdist$, we shall often
establish the asymptotic linearity of $\what{\theta}_n$
at $P_0 \times \auxdist$, meaning
\begin{equation}
  \label{eqn:asymptotic-linearity}
  \sqrt{n}(\what{\theta}_n - \theta(P_0))
  = \frac{1}{\sqrt{n}} \sum_{i = 1}^n \influencefunc(X_i)
  + \frac{1}{\sqrt{n}} \sum_{i = 1}^n \phi_\aux(\auxvar_i)
  + o_{P_0}(1),
\end{equation}
where $\E[\influencefunc(X)] =
\E[\phi_\aux(\auxvar)] = 0$, and $\cov(\influencefunc) = \Sigma_0$ and
$\cov(\phi_\aux) = \Sigma_\aux$.  Such expansions, with $\phi_\aux \equiv
0$, frequently occur in classical parametric, semi-parametric, and
nonparametric
statistics~\cite[cf.][Chs.~8 \& 25]{VanDerVaart98}.  For example, in
parametric cases with $\mc{P} = \{P_\theta\}_{\theta \in \Theta}$, standard
score $\score_\theta = \nabla_\theta \log p_\theta$, and Fisher information
$\fisher[\theta] = \E_\theta[\score_\theta \score_\theta^T]$, if
$\what{\theta}_n$ is the MLE (without privacy), then $\phi_\aux \equiv 0$
and $\influencefunc(x) =
-\fisher[\theta_0]^{-1}\score_{\theta_0}(x)$.
We have the following regularity
result, which essentially appears as~\cite[Lemmas 8.14 \&
  25.23]{VanDerVaart98}, though we include a proof in
Appendix~\ref{sec:proof-regular-differentiability}.
\begin{lemma}
  \label{lemma:regular-differentiability}
  Let $\Pmodel = \{P_h\}_{h \in \R^d} \subset \mc{P}$ be a
  QMD~\eqref{eqn:qmd-def} sub-model at $P_0$ with score $g$, and
  assume that $\theta : \mc{P} \to \R^k$ is
  differentiable~\eqref{eqn:differentiable-functional} relative to
  $\tangentset$ at $P_0$. Let $\what{\theta}_n$ be asymptotically
  linear~\eqref{eqn:asymptotic-linearity} at $P_0 \times \auxdist$.  Then
  for any sequence $h_n \to h \in \R^d$,
  \begin{equation*}
    \sqrt{n}\left(\what{\theta}_n - \theta(P_{h_n/\sqrt{n}})\right)
    \mathop{\cdlong}_{P_{h_n/\sqrt{n}} \times \auxdist}
    \normal\left(0,
    \Sigma_0 + \Sigma_\aux\right).
  \end{equation*}
  Additionally, for any bounded continuous $\loss : \R^k \to \R_+$
  and any $c < \infty$,
  \begin{equation*}
    \lim_{n \to \infty} \sup_{\norm{h} \le c}
    \E_{P_{h/\sqrt{n}}}\left[\loss(\sqrt{n}(\what{\theta}_n - \theta(P_{h/\sqrt{n}})))
      \right]
    = \E[\loss(Z)]
    ~~~ \mbox{where} ~~~
    Z \sim \normal(0, \Sigma_0 + \Sigma_\aux).
  \end{equation*}
\end{lemma}

We use Lemma~\ref{lemma:regular-differentiability} to describe the local
uniform convergence we seek.
Define the rescaled losses $\loss_n(t) = \loss(\sqrt{n} \cdot t)$.
We say an estimator
$\what{\theta}_n$ and channel $Q \in \channeldistset$ are
\emph{local minimax rate optimal}
if for all $P_0 \in \mc{P}$ with
QMD submodel $\Pmodel = \{P_h\} \subset \mc{P}$ passing through $P_0$ with
score function $g$,
\begin{equation}
  \label{eqn:asymptotic-local-uniform}
  \sup_{c < \infty} \limsup_{n \to \infty}
  \sup_{\norm{h} \le c / \sqrt{n}}
  \frac{\E_{Q \circ P_h}[\loss(\sqrt{n}(\what{\theta}_n(Z_1, \ldots, Z_n)
      - \theta(P_h)))]}{
    \localminimax_n(P_0, \loss_n, \mc{P}, \channeldistset)}
  \le C,
\end{equation}
where the constant $C$ is a numerical constant independent of
$\loss$ and $P_0$.
Our general recipe is now
apparent: demonstrate an asymptotically
linear~\eqref{eqn:asymptotic-local-uniform} locally private estimator
$\what{\theta}_n$
with covariance $\Sigma_0 + \Sigma_\aux$.
Then for \emph{any} collection of losses $\{\loss\}$ for which
we can lower bound $\localminimax_n(P_0, \loss_n,
\Pfamily, \channeldistset) \gtrsim \E[\loss(Z)]$ when $Z \sim \normal(0,
\Sigma_0 + \Sigma_\aux)$, we obtain the
convergence~\eqref{eqn:asymptotic-local-uniform}.

\newcommand{\grad}{\nabla}
\newcommand{\res}{R}
\newcommand{\modltwolin}{\modltwo^{\rm lin}}
\newcommand{\lambdamin}{\lambda_{\min}}
\newcommand{\lambdamax}{\lambda_{\max}}
\newcommand{\Creverse}{\gamma} 
\newcommand{\Cgrow}{\beta} 

\section{Local minimax complexity and private estimation}
\label{sec:local-minimax}

We turn to our main goal of establishing localized minimax complexities for
locally private estimation.
We focus first on the squared error for simplicity
in Section~\ref{sec:squared-error}, giving consequences of our
results.
Instead of the Hellinger modulus~\eqref{eqn:hellinger-modcont},
we show upper and lower bounds on the
local minimax minimax complexity for private estimation
using a local total variation modulus.
We then give several example calculations, and provide a super-efficiency
result. In Sections~\ref{sec:full-private-lower}
and~\ref{sec:super-efficiency}, we generalize to show how a total variation
modulus characterizes local minimax complexity for nearly arbitrary
losses, making our initial results on squared error corollaries.



\subsection{Local minimax squared error and the variation distance modulus}
\label{sec:squared-error}



We begin with a somewhat simplified setting,
where we wish to estimate a parameter $\theta(P) \in \R$ of
a distribution $P \in \mc{P}$, a collection of possible distributions, and
we measure performance of an estimand $\theta$ via the
squared error $\losssq(\theta, P) = (\theta - \theta(P))^2$.
For a family of distributions $\mc{P}$,
the \emph{modulus of continuity} with respect to the variation
distance at distribution $P_0$ is
\begin{equation}
  \label{eqn:private-moc}
  \tvmod(\delta; P_0, \mc{P})
  \defeq \sup_{P \in \mc{P}} \left\{ |\theta(P) - \theta(P_0)|
  ~\mbox{s.t.}~ \tvnorm{P - P_0} \le \delta \right\}.
\end{equation}
As we shall see, this modulus of continuity fairly precisely characterizes
the difficulty of locally private estimation of functionals.  The key
is that the modulus is with respect to
\emph{variation distance}.  This is in contrast to the classical results we
review in the introduction and Section~\ref{sec:classical-local-minimax} on
optimal estimation, where the more familiar modulus of continuity with
respect to Hellinger distance characterizes problem difficulty.  As we
illustrate, the difference between the Hellinger~\eqref{eqn:hellinger-modcont}
and
variation~\eqref{eqn:private-moc} moduli
leads to different behavior
for private and non-private estimation problems.

With this, we come to a corollary of our
Theorem~\ref{theorem:modulus-of-continuity}, to come in
Section~\ref{sec:full-private-lower}:
\begin{corollary}
  \label{corollary:modulus-squared-error}
  Let $\channeldistset_\diffp$ be the collection
  of $(2, \diffp^2)$-locally R\'{e}nyi private channels
  (Definition~\ref{definition:local-renyi}).
  Then
  \begin{equation*}
    \localminimax_n(P_0, \Lsq, \mc{P}, \channeldistset_\diffp)
    \ge \frac{1}{16} \tvmod^2\left(\frac{1}{2 \sqrt{2 n \diffp^2}};
    P_0, \mc{P}\right).
  \end{equation*}
\end{corollary}
\noindent
An identical bound (to within numerical constants) holds for
$\diffp$-locally differentially private channels, as (recall
Section~\ref{sec:definition-of-privacy}) any $\diffp$-differentially private
channel is $(2, O(1) \diffp^2)$-R\'{e}nyi private.  In (nearly) simultaneous
independent work to the original version of this paper on the
\texttt{arXiv}, \citet{RohdeSt18} provide a
global~\eqref{eqn:private-minimax} minimax lower bound via a global modulus
of continuity with respect to variation distance,
extending~\cite{DonohoLi87,DonohoLi91a,DonohoLi91b} to the private case.
The main difference is our focus: while they, similar
to~\cite{DuchiJoWa18}, demonstrate that private minimax rates  depart
from
non-private ones, our focus is on instance-specific bounds. Consequently,
\citeauthor{RohdeSt18} study linear functionals $\theta(P)$,
designing estimators to achieve the global minimax risk, while we allow
nonlinear functionals and develop estimators that must achieve the refined
local minimax complexity, with the hope that we may calculate practically
useful quantities akin to classical information
bounds~\cite{VanDerVaart98,LeCamYa00}.  (As an aside, we also provide lower
bounds for weaker forms of privacy.)

We can provide a converse to Corollary~\ref{corollary:modulus-squared-error}
that (nearly) characterizes the local minimax error by
the modulus of continuity. Indeed,
Proposition~\ref{proposition:achievable} to come implies
that for $\diffp \le \frac{3}{2}$, we have
\begin{corollary}
  \label{corollary:modulus-upper}
  Let $\channeldistset_\diffp$ be all non-interactive
  $\diffp$-differentially private channels
  (Def.~\ref{definition:local-dp}). Then
  \begin{equation*}
    \localminimax_n(P_0, \Lsq, \mc{P}, \channeldistset_\diffp)
    \le 2 \sup_{\tau \ge 0}
    \tvmod^2\left(\frac{5 \sqrt{2} \tau}{\sqrt{n \diffp^2}}; P_0, \mc{P}\right)
    e^{-\tau^2}.
  \end{equation*}
\end{corollary}
\noindent
Exactly as in inequality~\eqref{eqn:local-minimax-via-hellinger}, whenever
the modulus $\tvmod$ grows at most polynomially---so that
there exist $B, \beta < \infty$ such that
$\tvmod(c \delta; P_0, \mc{P}) \le B c^\beta \tvmod(\delta; P_0, \mc{P})$
for $c > 1$, we have
\begin{equation*}
  \localminimax_n(P_0, \Lsq, \mc{P}, \channeldistset)
  \le C_{B,\beta}
  \tvmod^2\left(\frac{C_1}{\sqrt{n \diffp^2}}; P_0, \mc{P}\right)
\end{equation*}
where $C_1$ is a numerical constant and $C_{B,\beta}$ depends on
$B, \beta$ only. We note that we have thus far characterized the
local minimax benchmark but have provided
no estimator uniformly
achieving it.

\subsubsection{Example moduli of continuity}

It is instructive to give examples of the local modulus and
connect them to estimation rates. We give
three mean estimation examples---a fully nonparametric
setting, a collection of distributions $\mc{P}$ with bounded
variance, and a Bernoulli estimation problem---where we see
that the variation modulus~\eqref{eqn:private-moc}
is essentially
independent of the distribution $P_0$, in distinction with the Hellinger
modulus~\eqref{eqn:hellinger-modcont}.
After these, additional examples will highlight that this is
not always the case.

\begin{example}[Bounded mean estimation]
  \label{example:mean-estimation}
  Let $\mc{X} \subset \R$ be a bounded set and
  $\mc{P} \defeq \{P : \supp P \subset \mc{X}\}$ be the collection
  of distributions supported on $\mc{X}$.
  Using the shorthand $\theta_0 = \theta(P_0) = \E_{P_0}[X]$,
  we claim the following upper and lower 
  bounds:
  \begin{equation}
    \label{eqn:nonparametric-mean-modulus}
    \delta \cdot \sup_{x\in \xdomain} |x - \theta_0| \leq 
    \tvmod(\delta; P_0, \mc{P})
    \leq 2 \delta \cdot \sup_{x\in \xdomain} |x - \theta_0|,
  \end{equation}
  so that the local modulus is nearly independent of $P_0$.  To see the
  lower bound~\eqref{eqn:nonparametric-mean-modulus}, for any $x\in
  \xdomain$, define $P_x = (1-\delta)P_0 + \delta \cdot \ones_x$, where
  $\ones_x$ denotes a point mass at $x$. Then $\tvnorm{P_x - P_0} \leq
  \delta$, so $\tvmod(\delta) \geq \sup_{x \in
    \mc{X}} |\theta_0 - \theta(P_x)| = \delta \cdot \sup_{x\in \xdomain} |x
  - \theta_0|$. The upper bound~\eqref{eqn:nonparametric-mean-modulus} is
  straightforward:
  \begin{align*}
    |\theta(P) - \theta_0|
    & = \bigg|\int (x - \theta_0) (dP(x) - dP_0(x)) \bigg|
    \le 2 \sup_{x \in \xdomain} |x - \theta_0| \tvnorm{P - P_0}
  \end{align*}
  for all $P \in \mc{P}$, by the triangle inequality, which is our desired
  result.

  On the other hand, the Hellinger modulus~\eqref{eqn:hellinger-modcont}
  (asymptotically) smaller. Let $\mc{P}$ be any collection of distributions
  with uniformly bounded fourth moment. We claim (see
  Appendix~\ref{sec:proof-hellinger-mean-modulus} for proof) that there
  exist numerical constants $0 < c_0 \le c_1 < \infty$ such that for all
  small enough $\delta > 0$,
  \begin{equation}
    \label{eqn:hellinger-mean-modulus}
    c_0 \sqrt{\var_{P_0}(X)} \cdot \delta
    \le \helmod(\delta; P_0, \mc{P}) \le c_1 \sqrt{\var_{P_0}(X)} \cdot \delta
    ~~ \mbox{and} ~~
    \lim_{\delta \downarrow 0}
    \frac{\helmod(\delta; P_0, \mc{P})}{\sqrt{8 \var_{P_0}(X)} \delta}
    = 1.
  \end{equation}
  The variance
  $\var_{P_0}(X)$ of the distribution $P_0$
  thus determines the
  local Hellinger modulus~\eqref{eqn:hellinger-modcont}.
\end{example}

\begin{example}[Means with
    bounded variance]
  \label{example:bounded-variance-means}
  We specialize
  Example~\ref{example:mean-estimation} by considering distributions
  on $\mc{X}$ with a variance bound $\sigma^2$, defining
  $\mc{P} \defeq \{P : \supp P \subset \mc{X}, \var_P(X) \le \sigma^2\}$.
  We consider the case that $\var_{P_0}(X) < \sigma^2$;
  we claim that the bounds~\eqref{eqn:nonparametric-mean-modulus}
  again hold for small $\delta > 0$. The upper bound is
  immediate.  The lower bound follows by noting that if $P_x = (1 - \delta)
  P_0 + \delta \cdot \ones_x$, then
  $\var_{P_x}(X) =
  \var_{P_x}(X - \theta_0)
  = (1 - \delta) \var_{P_0}(X) + \delta(1 - \delta)
  (x - \theta_0)^2$,
  so that for small enough $\delta$ we have $\var_{P_x}(X) \le \sigma^2$ and
  the identical lower bound~\eqref{eqn:nonparametric-mean-modulus} holds.
\end{example}

\begin{example}[Modulus of continuity for Bernoulli parameters]
  \label{example:modulus-bernoulli}
  We further restrict to binary random variables, so that the problem
  is parametric. Let $\bernoulli(\theta)$ be the Bernoulli distribution
  with mean $\theta$, and $\mc{P} = \{\bernoulli(\theta)\}_{\theta \in [0,1]}$.
  We have
  $\tvnorms{P_{\theta_0} - P_\theta} = |\theta - \theta_0|$ and
  for $\delta \le \half$,
  $\tvmod(\delta; P_{\theta_0}, \mc{P}) = \delta$.
  On the other hand, 
  Eq.~\eqref{eqn:hellinger-mean-modulus} shows that
  $\helmod^2(\delta; P_{\theta_0}, \mc{P})
  =
  8 \frac{\delta^2}{\theta_0 (1 - \theta_0)}(1 + o(1))$. The Hellinger modulus
  is local to $\theta_0$, while the local variation modulus is global.
\end{example}

Summarizing
examples~\ref{example:mean-estimation}--\ref{example:modulus-bernoulli}, in
each case the local TV-modulus~\eqref{eqn:private-moc} of distributions
supported on $\mc{X}$ \emph{must} scale as the diameter of
$\mc{X}$---essentially identical to a global modulus of continuity over the
full set $\mc{P} = \{P : \supp P \subset \mc{X}\}$---while the
Hellinger modulus~\eqref{eqn:hellinger-modcont} scales linearly in
$\sqrt{\var_{P_0}(X)}$. This lack of locality
in the local modulus for variation distance has consequences for
estimation, which we can detail by applying
Corollary~\ref{corollary:modulus-squared-error}:
\begin{corollary}[Locally private mean estimation]
  \label{corollary:mean-estimation}
  Let $\mc{X}$ be bounded, let $\mc{P}$ be any of the collections of
  distributions in
  Examples~\ref{example:mean-estimation}--\ref{example:modulus-bernoulli},
  and let $\channeldistset_\diffp$ be the collection of $(2, \diffp^2)$-R\'{e}nyi
  locally private channels. There exists a numerical constant $c > 0$ such
  that for any $P_0 \in \mc{P}$ (where in the case
  of Example~\ref{example:bounded-variance-means} we require $\var_{P_0}(X)
  < \sigma^2$), for all large enough $n$
  \begin{equation*}
    c \frac{\diam(\mc{X})^2}{n \diffp^2}
    + c \frac{\var_0(X)}{n}
    \le \minimax_n(P_0, \Lsq, \mc{P}, \channeldistset)
    \le \frac{\diam(\mc{X})^2}{2n \diffp^2} + \frac{\var(X)}{n}.
  \end{equation*}
\end{corollary}
\noindent
Standard mechanisms~\cite{DuchiJoWa18} achieve the upper
bound in Corollary~\ref{corollary:mean-estimation}: letting $Z_i = X_i +
\frac{\diam(\mc{X})}{\diffp} W_i$ for $W_i \simiid \laplace(1)$ gives an
$\diffp$-differentially private view of $X_i$; define the estimator
$\what{\theta}_n = \frac{1}{n} \sum_{i = 1}^n Z_i$.
This highlights the difference with the
non-private case, where the matching upper and lower bounds
are $\var_{P_0(X)} / n$, while in the private case
the diameter of $\mc{X}$ is central.

Yet the local total-variation (private) modulus can
depend strongly on the distribution $P_0$ and
set $\mc{P}$ of potential alternatives, a point to which
we will return later. Two simple examples illustrate.

\begin{example}[Modulus of continuity for a normal mean]
  \label{example:normal-modulus}
  Let $\mc{P} = \{\normal(\theta, \sigma^2)\}_{\theta \in \R}$ for
  a known variance $\sigma^2$. Letting $\phi$ and $\Phi$ be the standard
  normal p.d.f.\ and c.d.f., respectively, for any
  pair $\theta_0, \theta_1 \in \R$ with $\Delta = |\theta_0 - \theta_1|$,
  we then have
  $\tvnorms{\normal(\theta_0, \sigma^2) - \normal(\theta_1, \sigma^2)}
  = \Phi(\Delta / 2 \sigma) - \Phi(-\Delta / 2 \sigma)$. Solving for
  the modulus gives that for any $P_0 \in \mc{P}$,
  \begin{equation*}
    \tvmod(\delta; P_0, \mc{P})
    = \frac{\sigma \delta}{\phi(0)}(1 + O_\sigma(\delta))
  \end{equation*}
  as $\delta \to 0$. It is possible but tedious to extend this
  to cases with an unknown variance, so that
  $\mc{P} = \{\normal(\theta, \sigma^2)\}_{\theta \in \R, \sigma^2
    < \infty}$, in which case
  we obtain
  $\tvmod(\delta; P_0, \mc{P}) \asymp \sqrt{\var_{P_0}(X)} \delta$ as
  $\delta \to 0$.
\end{example}

Other parametric families also have stronger dependence on the
local distribution $P$.

\begin{example}[Exponential distributions]
  \label{example:exponential-modulus}
  Let $p_\theta(x) = \frac{1}{\theta} \exp(-\frac{x}{\theta}) \indic{x \ge
    0}$ be the density of an exponential distribution with scale $\theta$,
  and $\mc{P}$ be the collection of such distributions.
  Let $\tau, \theta > 0$, and set $x\subopt = \frac{\theta \tau}{\theta - \tau}
  \log \frac{\theta}{\tau}$.
  The variation distance between two exponential distributions is then
  $\tvnorm{P_\theta - P_\tau} = |e^{-x\subopt / \theta} - e^{-x\subopt / \tau}|$.
  For $\theta = \tau + \delta$ (or $\tau = \theta - \delta$), we thus obtain
  that
  \begin{equation*}
    \tvnorm{P_\theta - P_{\theta - \delta}}
    = \exp\left(-\frac{\theta}{\delta} \log \frac{1}{1 - \delta / \theta}
    \right) \left|\frac{1}{1 - \delta/\theta} - 1 \right|
    = e^{-1} \left[\frac{|\delta|}{\theta} + O_\theta(\delta)\right],
  \end{equation*}
  and $\tvnorm{P_\theta - P_{\theta - \delta}}$ is monotonic in $|\delta|$.
  Eliding details, we thus find that
  \begin{equation*}
    \tvmod(\delta; P_\theta, \mc{P})
    = \theta \delta \cdot (e + O_\theta(\delta)),
  \end{equation*}
  which evidently is local to $\theta$.
\end{example}

\noindent

\subsubsection{Super-efficiency for squared error}
\label{sec:super-efficiency-squared-error}

To demonstrate that the local modulus of continuity is the
``correct'' lower bound on estimation, we consider the third of the
desiderata for a strong lower bound that we idenfity in the introduction: a
super-efficiency result~\cite{BrownLo96,CaiLo15,Tsybakov98} showing that any
estimator substantially outperforming the local minimax benchmark at a given
distribution $P_0$ necessarily suffers higher expected error for some other
distribution $P_1$.  As a corollary of
Proposition~\ref{proposition:super-efficiency} to come, we establish the
following result.
\begin{corollary}
  \label{corollary:squared-super-efficiency}
  Let $\channel$ be a sequentially
  interactive $(2, \diffp^2)$-R\'{e}nyi-private
  channel (Def.~\ref{definition:local-renyi}).
  If for some $\eta \in [0, 1]$ the
  estimator $\what{\theta}$ satisfies
  \begin{equation*}
    \E_{\channel \circ P_0} [(\what{\theta}(Z_{1:n}) - \theta_0)^2]
    \le \eta \tvmod^2\left(\frac{1}{\sqrt{4 n \diffp^2}}; P_0, \mc{P} \right),
  \end{equation*}
  then for all $t \in [0, 1]$
  there exists a distribution $P_1 \in \mc{P}$ such that
  \begin{equation*}
    \E_{\channel \circ P_1}[(\what{\theta}(Z_{1:n}) - \theta(P_1))^2]
    \ge \frac{1}{8}
    \hinge{1 - \eta^\frac{ (1 - t)}{2}}^2
    \tvmod^2 \Bigg(\frac{1}{4} \sqrt{\frac{t  \log \frac{1}{\eta}}{n \diffp^2}};
    P_1, \mc{P}\Bigg).
  \end{equation*}
\end{corollary}

Unpacking the corollary by ignoring constants (e.g., set $t = \half$),
we see (roughly) the following result: if an estimator achieves expected
squared error less (by a factor $\eta < 1$) than the squared modulus of
continuity at $P_0$, it must have squared error scaling
with the modulus for a radius $\sqrt{\log \frac{1}{\eta}}$-times larger.
For example, considering the sample mean
examples~\ref{example:mean-estimation}--\ref{example:modulus-bernoulli},
we see that in any of the settings, there exists a numerical constant
$c > 0$ such that if $\what{\theta}_n$ is locally private and
\begin{equation*}
  \E_{P_0}\left[(\what{\theta}_n - \theta(P_0))^2\right]
  \le \eta \frac{\diam(\mc{X})^2}{n \diffp^2}
\end{equation*}
for some $0 < \eta < 1$,
then there exists $P_1 \in \mc{P}$ such that for all large enough $n$,
\begin{equation*}
  \E_{P_1}\left[(\what{\theta}_n - \theta(P_1))^2\right]
  \ge c  \frac{\diam(\mc{X})^2}{n \diffp^2} \cdot
  \log \frac{1}{\eta}.
\end{equation*}

\subsection{Local private minimax risk for general losses}
\label{sec:full-private-lower}

We return to prove our local minimax upper and lower bounds for general
losses, along the way proving the claimed corollaries.
Recall that we use any symmetric quasiconvex loss $\loss : \R^d \to \R_+$
satisfying $\loss(\zeros) = 0$.
Then
for a family of distributions $\mc{P}$,
the \emph{modulus of continuity}
associated with the loss $L$ at the distribution $P_0$ is
\begin{equation}
  \label{eqn:private-moc-loss}
  \lossmodcont(\delta; P_0, \mc{P})
  \defeq \sup_{P \in \mc{P}} \left\{ \loss\left(\frac{\theta(P_0) - \theta(P)}{2}
  \right)
  ~\mbox{s.t.}~ \tvnorm{P - P_0} \le \delta \right\},
\end{equation}
where the normalization by $\half$ is convenient for our proofs.
We then have our first main theorem, which lower bounds the local minimax
risk using the modulus~\eqref{eqn:private-moc-loss} in analogy to
Proposition~\ref{proposition:classical-local-minimax}.  We defer the proof
to Section~\ref{sec:proof-modulus-of-continuity}, where we also present a
number of new strong data-processing inequalities to prove it.
\begin{theorem}
  \label{theorem:modulus-of-continuity}
  Let $\channeldistset$ be the collection of $(2, \diffp^2)$-locally
  R\'{e}nyi differentially private channels
  (Definition~\ref{definition:local-renyi}).
  Let $c_{\textup{conv}} = 1$ if $\loss$ is convex and $2$ otherwise.
  Then for any distribution
  $P_0$, we have
  \begin{equation*}
    \localminimax_n(P_0, L, \mc{P}, \channeldistset)
    \ge
    \frac{1}{4 c_{\textup{conv}}}
    \lossmodcont\left(\frac{1}{2 \sqrt{2 n \diffp^2}};
    P_0, \mc{P}\right).
  \end{equation*}
\end{theorem}
\noindent
Corollary~\ref{corollary:modulus-squared-error} is then immediate: for the
squared error, $\loss(\half(\theta(P_0) - \theta(P_1))) =
\frac{1}{4}(\theta(P_0) - \theta(P_1))^2$.  Note also, as in the
discussion after Corollary~\ref{corollary:modulus-squared-error}, that this
implies the lower bound $\lossmodcont(O(1) / \sqrt{n \diffp^2}; P_0,
\mc{P})$ on any $\diffp$-locally differentially private procedure.

An upper bound in the theorem is a somewhat more delicate argument, and for
now we do not provide procedures achieving the lower bound.  Instead, under
reasonable conditions on the loss, we can show the (partial) converse that
the modulus $\lossmodcont$ describes the local minimax complexity.


\begin{condition}[Growth inequality]
  \label{cond:reverse-triangle}
  There exists $\Creverse < \infty$ such that for all $t \in \R^d$,
  \begin{equation*}
    \loss(t)
    \le \Creverse \loss(t / 2).
  \end{equation*}
\end{condition}
\noindent
For example, for the squared error we have
$\losssq(t / 2) = t^2 / 4 = \losssq(t) / 4$, giving $\Creverse = 4$.
In
Appendix~\ref{sec:proof-achievable}, we prove the following
partial converse to Theorem~\ref{theorem:modulus-of-continuity}.
\begin{proposition}
  \label{proposition:achievable}
  Let Condition~\ref{cond:reverse-triangle} 
  on the loss $L$ hold.
  Let $\diffp \ge 0$ and
  $\delta_\diffp = \frac{e^\diffp}{e^\diffp + 1} - \half$,
  and let $\channeldistset$ be the collection
  of non-interactive $\diffp$-differentially private
  channels (Definition~\ref{definition:local-dp}).
  Then
  \begin{equation*}
    \localminimax_n(P_0, L, \mc{P}, \channeldistset)
    \le 2 \Creverse \sup_{\tau \ge 0}
    \left\{\lossmodcont\left(\frac{\sqrt{2} \tau}{\delta_\diffp \sqrt{n}};
    P_0, \mc{P}\right) e^{-\tau^2} \right\}.
  \end{equation*}
\end{proposition}

The proposition as written is a bit unwieldy, so we unpack it slightly.
We have $\delta_\diffp \ge \min\{\frac{\diffp}{5}, 1/3\}$, so
for each $P_1 \in \mc{P}$ there exists a non-interactive
$\diffp$-DP channel $\channel$ and estimator
$\what{\theta}$ such that
\begin{equation*}
  \max_{P \in \{P_0, P_1\}} \E_{P,Q}
  \left[L(\what{\theta}(Z_{1:n}), P)\right]
  \le 2 \Creverse \cdot
  \sup_{\tau \ge 0} \lossmodcont\left(\frac{3 \sqrt{2} \tau}{\sqrt{n
      \min\{9 \diffp^2 / 25, 1\}}},
  P_0, \mc{P}\right) e^{-\tau^2}.
\end{equation*}
Typically, this supremum is achieved at $\tau = O(1)$,
so that Proposition~\ref{proposition:achievable} shows that the
modulus~\eqref{eqn:private-moc-loss} at radius $\frac{O(1)}{\sqrt{n \diffp^2}}$
characterizes the local minimax risk to constants for $\diffp \lesssim 1$.
Appropriate assumptions,
including the following condition on the modulus of continuity,
allow more precision.
\begin{condition}[Polynomial growth]
  \label{cond:polynomial-growth}
  For each $P_0$, there exist $\alpha, \Cgrow < \infty$ such that
  for all $c \ge 1$
  \begin{equation*}
    \lossmodcont(c \delta; P_0, \mc{P})
    \le (\Cgrow c)^\alpha \lossmodcont(\delta; P_0, \mc{P}).
  \end{equation*}
\end{condition}
\noindent
Condition~\ref{cond:polynomial-growth} is similar to the typical
H\"older-type continuity properties assumed on the modulus
of continuity for estimation problems~\cite{DonohoLi87, DonohoLi91a}.
It holds, for example, for nonparametric mean estimation
problems (recall Example~\ref{example:mean-estimation}),
and we make this more concrete after the following corollary.

\begin{corollary}
  \label{corollary:achievable}
  In addition to the conditions of Proposition~\ref{proposition:achievable},
  let Condition~\ref{cond:polynomial-growth} hold.
  Then
  \begin{equation*}
    \localminimax(P_0, \loss, \mc{P}, \channeldistset)
    \le \gamma \beta^\alpha e^{\frac{\alpha}{2} [\log \frac{\alpha}{2} - 1]}
    \lossmodcont \left(\frac{\sqrt{2}}{\delta_\diffp \sqrt{n}}; P_0,
    \mc{P}\right).
  \end{equation*}
\end{corollary}
\begin{proof}
  We apply Proposition~\ref{proposition:achievable}.  For $\tau \le 1$, it
  already gives the result; otherwise, we use the growth
  condition~\ref{cond:polynomial-growth} to
  obtain
  $\E_0[\loss(\what{\theta} - \theta(P_0))]
  + \E_1[\loss(\what{\theta} - \theta(P_1))]
  \le 2 \Creverse \lossmodcont(\frac{\sqrt{2}}{
    \delta_\diffp \sqrt{n}}; P_0, \mc{P})
  \Cgrow^\alpha \sup_{\tau \ge 1}
  \tau^\alpha e^{-\tau^2}$.
  Noting that $\sup_{\tau \ge 0}
  \tau^\alpha e^{-\tau^2} = (\alpha/2)^{\alpha/2} e^{-\alpha/2}$ gives
  the result.
\end{proof}

We generally expect Condition~\ref{cond:polynomial-growth} to hold, so that
the modulus describes the risk.  Indeed, for any loss $\loss : \R^d \to \R_+$
satisfying Conditition~\ref{cond:reverse-triangle}, we immediately obtain
condition~\eqref{cond:polynomial-growth} whenever $\tvmod(\cdot)$ satisfies
the condition, which it does for each of
Examples~\ref{example:mean-estimation}--\ref{example:modulus-bernoulli}
and (locally) \ref{example:normal-modulus}--\ref{example:exponential-modulus}.

\subsection{Super-efficiency}
\label{sec:super-efficiency}

We provide our general super-efficiency result via a constrained risk
inequality~\cite{BrownLo96, DuchiRu18}.  Our result applies in the typical
setting in which the loss is $\loss(t) = \Phi(\ltwo{\theta -
  \theta(P)})$ for some increasing function $\Phi: \R_+ \to \R_+$, and we
use the shorthand $\risk(\what{\theta}, \theta, P) \defeq
\E_P[\Phi(\ltwos{\what{\theta}(Z) - \theta})]$ for the risk (expected loss)
of the estimator $\what{\theta}$ under the distribution $P$.  We build off
the approach of \citet[Thm.~1]{BrownLo96}, who show that if $\what{\theta}$
has squared error for a parameter $\theta$ under a distribution $P_0$, then
its risk under a distribution $P_1$ close to $P_0$ may be large (see
also~\cite[Thm.~6]{Tsybakov98}).
The next proposition, whose proof we provide
in Section~\ref{sec:proof-super-efficiency}, extends this to show
that improvement over our modulus of continuity lower bound at a
point $P_0$ implies worse performance elsewhere.
\begin{proposition}
  \label{proposition:super-efficiency}
  Let $\channel$ be a sequentially
  interactive ($2, \diffp^2$)-R\'{e}nyi private channel 
  (Def.~\ref{definition:local-renyi}) with associated marginal
  distributions $\marginprob_a^n(\cdot) = \int \channel(\cdot \mid x_{1:n})
  dP_a^n(x_{1:n})$.  Let Condition~\ref{cond:reverse-triangle}
  hold with parameter $\Creverse$. If for some $\eta \in [0, 1]$ the
  estimator $\what{\theta}$ satisfies
  \begin{equation*}
    \risk(\what{\theta}, \theta_0, \marginprob_0^n)
    \le \eta \lossmodcont\left(\frac{1}{\sqrt{4 n \diffp^2}};
    P_0, \mc{P} \right),
  \end{equation*}
  then for all $t \in [0, 1]$
  there exists a distribution $P_1 \in \mc{P}$ such that
  \begin{equation*}
    \risk(\what{\theta}, \theta(P_1), \marginprob_1^n)
    \ge \frac{1}{2\Creverse}
    \hinge{1 - \eta^\frac{ (1 - t)}{2}}^2
    \lossmodcont\Bigg(\frac{1}{4} \sqrt{\frac{t  \log \frac{1}{\eta}}{n \diffp^2}};
    P_1, \mc{P}\Bigg).
  \end{equation*}
\end{proposition}

The proposition depends on a number of constants, but roughly, it shows (for
small enough $\eta$, where we simplify by taking $t = 1/2$) that if an
estimator $\what{\theta}$ is super-efficient at $P_0$, in that
$\risk(\what{\theta}, \theta_0, \marginprob_0^n) \le \eta \cdot \lossmodcont(1 /
\sqrt{4 n \diffp^2}; P_0)$, then there exists $c > 0$ such that
for some $P_1$ we have $\risk(\what{\theta}, \theta_1 \marginprob_1^n) \ge c
\cdot \lossmodcont(\sqrt{\log (1/\eta)} / \sqrt{32 n \diffp^2}; P_1)$.  In this
sense, our bounds are sharp: any estimator
achieving much better risk than the local modulus at a
distribution $P_0$ must pay elsewhere.




\newcommand{\Var}{{\rm Var}}
\newcommand{\Cov}{{\rm Cov}}
\newcommand{\predloss}{L_{\rm pred}}
\newcommand{\symmetricloss}{\varphi}

\section{The private information}
\label{sec:examples}

The ansatz of finding a locally most difficult problem via the local
variation modulus of continuity~\eqref{eqn:private-moc} gives an approach to
lower bounds that leads to non-standard behavior for a number of classical
and not-so-classical problems in locally private estimation.
In this section, we investigate
examples in several one-dimensional parametric problems, showing how local
privacy leads to a different geometry of local complexities than classical
cases.  Our first step is to define the
\emph{$\bigLone$ information},
a private analogue of the Fisher
Information that
governs the complexity of estimation under local
privacy.  We illustrate the private $\bigLone$ information for several
examples, including of the mean of Bernoulli random variable, the scale of
an exponential random variable, and in linear and logistic models
(Sec.~\ref{sec:concrete-examples}), showing the consequences of (locally)
private estimation in one dimension. Our last two sections develop locally
private algorithms for achieving the local minimax risk. The first of these
(Sec.~\ref{sec:stochastic-gradient-stuff}) describes private stochastic
gradient algorithms and their (locally uniform) asymptotics, while the last
(Sec.~\ref{sec:one-param-expfams}) develops a new locally private algorithm
based on Fisher scoring to achieve the $\bigLone$ information in
one-dimensional exponential families.

\subsection{Private analogues of the Fisher Information}
\label{sec:fisher-information}

Our first set of results builds off of
Theorem~\ref{theorem:modulus-of-continuity} by performing asymptotic
approximations to the variation distance for regular parametric
families of distributions. One major consequence of our results is that, under 
the notions of locally private estimation we consider, the classical Fisher 
information is \emph{not} the right notion of complexity in estimation, though
an analogy is possible. Again we emphasize that
we hope to characterize complexity only to numerical
constant factors, seeking the problem-dependent terms that analogize the
classical information.

We begin by considering parametric families that allow analogues of Le Cam's
quadratic mean differentiability (QMD)~\cite[Ch.~7]{VanDerVaart98}. Consider
a 1-dimensional parametric collection $\mc{P} = \{P_\theta\}_{\theta \in
  \Theta}$ with dominating measure $\mu$ and densities $p_\theta =
dP_\theta / d\mu$.  Analogizing the QMD definition~\eqref{eqn:qmd-def} from
the Hellinger to the variation distance, we say
$\mc{P}$ \emph{is $\bigLone$-differentiable at $\theta_0$} with score
$\score_{\theta_0} : \mc{X} \to \R$ if
\begin{equation}
  \label{eqn:l1-differentiable}
  \int |p_{\theta_0 + h} - p_{\theta_0}
  - h \score_{\theta_0} p_{\theta_0}| d\mu
  = o(|h|).
\end{equation}
For QMD families,
$\bigLone$-differentiability is automatic
(see Appendix~\ref{sec:proof-nice-families-l1-differentiable} for a proof).
\begin{lemma}
  \label{lemma:nice-families-l1-differentiable}
  Let the family $\mc{P} \defeq \{P_\theta\}_{\theta \in \Theta}$ be
  QMD~\eqref{eqn:qmd-def} at the point $\theta_0$. Then
  $\mc{P}$ is $\bigLone$-differentiable at $\theta_0$
  with identical score $\score_\theta$ to the QMD case.
\end{lemma}

Recalling (as in Sec.~\ref{sec:classical-local-minimax}) that
for QMD families~\eqref{eqn:qmd-def}, the Fisher
information is
$\fisher[\theta] = \E_{P_\theta}[(\score_\theta)^2]$, and
$\dhel^2(P_{\theta + h}, P_\theta)
= \frac{1}{8} \fisher[\theta]h^2 + o(h^2)$,
by analogy, we define the \emph{$\bigLone$-information} as
\begin{equation}
  \label{eqn:l1-information}
  \information[\theta_0] 
  \defeq \E_{P_\theta}[|\score_\theta|] = \int \left|\score_{\theta_0}(x)\right|
  dP_{\theta_0}(x).
\end{equation}
We can then locally approximate the total variation distance by the
$\bigLone$-information:
\begin{equation*}
  \tvnorm{P_{\theta + h} - P_\theta} = \half \information[\theta]|h| + o(|h|).
\end{equation*}

We consider a somewhat general setting in which
we wish to estimate the value $\functional(\theta)$ of a functional
$\functional : \Theta \to \R$, where $\functional$
is $\mc{C}^1$ near $\theta_0$. We measure our error by
$\loss(\functional(\theta) - \functional(\theta_0))$,
and give a short proof of
the next proposition via Theorem~\ref{theorem:modulus-of-continuity}
in Appendix~\ref{sec:proof-information-private-lb}.

\begin{proposition}
  \label{proposition:information-private-lb}
  Let $\mc{P} = \{P_\theta\}_{\theta \in \Theta}$ be
  $\bigLone$-differentiable at $\theta_0$ with score $\score_{\theta_0}$,
  and $\E_{\theta_0}[|\score_{\theta_0}|] > 0$.  Let
  $\channeldistset_\diffp$ be the family of $(2, \diffp^2)$-R\'{e}nyi
  locally private
  sequentially interactive channels. Then for an $N = N(\functional,
  \theta_0, \mc{P}, \diffp)$ depending only on $\functional, \theta_0$, the
  family $\mc{P}$, and privacy level $\diffp$, for all $n \ge N$
  \begin{equation*}
    \localminimax_n(P_{\theta_0}, L, \mc{P}, \channeldistset_\diffp)
    \geq 
    \frac{1}{8}
    \cdot \loss\left(\frac{1}{5 \sqrt{2 n \diffp^2}}
    \cdot \information[\theta_0]^{-1} \functional'(\theta_0)\right).
  \end{equation*}
\end{proposition}

To obtain a matching upper bound we require the identifiability
assumption~\ref{assumption:identifiability}.  We make a simplifying
assumption that the loss $\loss$ is reasonably behaved, in that there exists a
numerical constant $C < \infty$ and $\beta \in \R_+$ such that $\loss(at) \le
C a^\beta \loss(t)$ for all $a \ge 1$.  Then, even when
$\channeldistset_\diffp$ is the collection of $\diffp$-locally
differentially private non-interactive channels (which,
by the discussion following Definition~\ref{definition:local-renyi},
is more limiting than channels in
Proposition~\ref{proposition:information-private-lb}), we can upper bound the
local minimax risk.
\begin{corollary}
  \label{corollary:l1-info-attained}
  Let the family $\mc{P} = \{P_\theta\}_{\theta \in \Theta}$ be $\bigLone$
  differentiable at $\theta_0$ with score $\score_{\theta_0}$ and
  $\E_{\theta_0}[|\score_{\theta_0}|] > 0$ and additionally let the above
  assumptions hold.  Let $\diffp \le 2$.  Then there exists a numerical
  constant $C < \infty$ and a $\delta_0 = \delta_0(\functional, \theta_0,
  \mc{P})$ depending only on $\functional, \theta_0$, and the family
  $\mc{P}$ such that
  \begin{equation*}
    \localminimax_n(P_{\theta_0}, \loss, \mc{P}, \channeldistset_\diffp)
    \le C \max\left\{
    (\beta/2e)^{\beta/2}
    \loss\left(\frac{\functional'(\theta_0)}{\information[\theta_0]}
    \frac{1}{\sqrt{n \diffp^2}}\right),
    \loss(\diam(\functional(\Theta))) e^{-\delta_0^2 n \diffp^2}\right\}.
  \end{equation*}
\end{corollary}
\noindent
The proof (see Appendix~\ref{sec:proof-l1-info-attained})
is a straightforward modification of that of
Claim~\ref{claim:qmd-bounded-local-minimax}.
Proposition~\ref{proposition:information-private-lb} and
Corollary~\ref{corollary:l1-info-attained} show that for
a (one-dimensional) parametric family $\mc{P} = \{P_\theta\}_{\theta \in \Theta}$,
the $\bigLone$ information describes the local modulus to within
numerical constants for small $\delta$:
for the modulus~\eqref{eqn:private-moc},
there are numerical constants $0 < c_{\textup{low}} \le c_{\textup{high}} < \infty$
such that
\begin{equation*}
  \tvmod(\delta; P_{\theta_0}, \{P_\theta\}_{\theta \in \Theta})
  \in [c_\textup{low}, c_\textup{high}]
  \cdot \frac{\delta}{\information[\theta_0]}
  ~~~ \mbox{for~all~small~} \delta > 0.
\end{equation*}
(A more general result holds; see Theorem~\ref{theorem:general-private} to
come.)  In analogy with Claim~\ref{claim:qmd-bounded-local-minimax}, where
the Fisher information $\fisher[\theta_0]$ characterizes the local minimax
squared error in non-private estimation, the $\bigLone$ information is an
alternative characterization---to within numerical constants---of the local
minimax risk in the locally private case.

As an alternative way to understand the proposition and corollary, we can
rescale the losses (in analogy with the local asymptotic
approach~\cite[Ch.~7]{VanDerVaart98}), and consider the sequence
$\loss_n(t) = \loss(\sqrt{n} \cdot t)$, where for simplicity we take $\loss(t) =
\min\{t^k, B\}$ for some $k, B < \infty$ (more generally, we could allow
$\loss$ to be bounded and nondecreasing).  Then under the conditions of
Corollary~\ref{corollary:l1-info-attained}, for $\channeldistset$ a
collection of $\diffp$-locally private channels,
\begin{equation*}
  \loss\left(
  \frac{\functional'(\theta_0)}{\information[\theta_0] \diffp}
  \right)
  \lesssim \localminimax_n(P_{\theta_0}, \loss_n, \mc{P}, \channeldistset)
  \lesssim \loss\left(
  \frac{\functional'(\theta_0)}{\information[\theta_0]\diffp}
  \right)
\end{equation*}
for all large $n$.  The analogous bounds in the non-private case are
$\loss(\fisher[\theta_0]^{-1/2} \functional'(\theta_0))$, the local
asymptotic complexity for one-dimensional
functionals~\cite[Ch.~7]{VanDerVaart98}.
Using Lemma~\ref{lemma:nice-families-l1-differentiable},
we have
\begin{equation*}
  \information[\theta_0] = \E_{\theta_0}[|\score_{\theta_0}|] \le
  \E_{\theta_0}[\score_{\theta_0}^2]^{1/2} = \fisher[\theta_0]^{1/2},
\end{equation*}
so the $\bigLone$ information is at most the Fisher information. In some cases,
as we shall see in Sec.~\ref{sec:concrete-examples}, it can be much smaller,
while in others the information measures are equal to numerical
constants.

\subsection{Examples and attainment}
\label{sec:concrete-examples}

We consider the local minimax complexity and $\bigLone$-information in four
different examples---estimation of Bernoulli and logistic the scale of an
exponential random variable, and a 1-dimensional linear regression
problem---which are particularly evocative.  In each, we derive the
$\bigLone$-information, applying
Proposition~\ref{proposition:information-private-lb} and
Corollary~\ref{corollary:l1-info-attained} to characterize the private local
minimax complexity. Throughout this section, we let $\channeldistset_\diffp$
be the collection of $\diffp$-locally differentially private channels, where
$\diffp = O(1)$ for simplicity. To keep the examples short, we do not always
provide algorithms, but we complete the picture via a private stochastic
gradient method in Section~\ref{sec:stochastic-gradient-stuff}.

\begin{example}[Example~\ref{example:modulus-bernoulli}
    continued]
  \label{example:l1-bernoulli-info}
  For $P_\theta = \bernoulli(\theta)$, the score is $\score_\theta(x) =
  \frac{x - \theta}{\theta(1 - \theta)}$, giving $\bigLone$-information
  $\information[\theta] = \half \E_\theta[|\score_\theta|] = 1$ for all
  $\theta$ and Fisher information $\fisher[\theta] =
  \E_\theta[\score_\theta^2] = \frac{1}{\theta(1 - \theta)}$. Thus
  \begin{equation*}
    \localminimax_n(P_{\theta_0}, \Lsq, \pfamily, \channeldistset_\diffp)
    \asymp \frac{1}{n \diffp^2}
  \end{equation*}
  for $\diffp = O(1)$.
  The lower bound is Proposition~\ref{proposition:information-private-lb},
  For the upper bound, consider
  the randomized-response mechanism~\cite{Warner65}
  that releases $Z_i = X_i$ with probability
  $\frac{e^\diffp}{1 + e^\diffp}$ and $Z_i = 1 - X_i$ otherwise,
  which is $\diffp$-differentially private. The
  plug-in estimate
  $\what{\theta}_n = \frac{(1+e^{\diffp})\wb{Z}_n - 1}{e^{\diffp}-1}$
  is unbiased for $\theta_0$ and has
  \begin{equation*}
    \E_0 [(\what{\theta}_n - \theta_0)^2] = 
    \Var_0(\what{\theta}_n) =
    \left(\frac{1+e^{\diffp}}{e^{\diffp} -1}\right)^2\Var(\wb{Z}_n)
    \le \frac{(1+e^{\diffp})^2}{4n(e^{\diffp}-1)^2}
    \lesssim \frac{1}{n \diffp^2}.
  \end{equation*}
  The sample mean achieves risk $\frac{\theta_0(1 - \theta_0)}{n}$, so the
  gap in efficiency between private and non-private estimation grows when
  $\theta_0(1 - \theta_0) \to 0$.  Roughly, the noise individual
  randomization introduces (the statistical
  cost of privacy) dominates the non-private (classical) statistical cost.
\end{example}

\noindent
As a brief remark, the paper~\cite{LeysiefferWa76} gives optimal asymptotics
for Bernoulli problems with randomized response channels; such precise
calculations are challenging when allowing arbitrary channels
$\channeldistset_\diffp$.

\begin{example}[Private one-dimensional logistic regression]
  \label{example:l1-logistic-info}
  A similar result to Bernoulli estimation that may be more striking holds
  for logistic regression, which is relevant for modern privacy
  applications, such as learning a classifier from (privately shared) user
  data~\cite{ErlingssonPiKo14,AbadiChGoMcMiTaZh16, ApplePrivacy17}. To see
  this, let $P_0$ be the distribution on pairs $(x, y) \in \{-1, 1\}^2$
  satisfying the logistic regression model
  \begin{equation}
    \label{eqn:logistic-model}
    P_0(y \mid x) = \frac{1}{1 + e^{-y \theta_0 x}}~~\text{and}~~P_0(x = \pm 1)
    = \half.
  \end{equation}
  Here we wish
  to construct a classifier that
  provides good confidence estimates $p(y \mid x)$ of a label $y$ given
  covariates $x$. We expect in the logistic regression
  model~\eqref{eqn:logistic-model} that large parameter values $\theta_0$
  should make estimating a classifier easier, as is the case without
  privacy. To make this concrete, we measure
  the error in estimating the conditional probability
  $p_\theta(y \mid x)$,
  \begin{equation*}
    \predloss(\theta, \theta_0)
    \defeq \E_{P_0}\left[\left|p_\theta(Y \mid X) - p_{\theta_0}(Y \mid X)\right|
      \right].
  \end{equation*}
  A calculation gives $\predloss(\theta, \theta_0)=
  |\phi(\theta) - \phi(\theta_0)|$, where $\phi(t) = 1/(1+e^t)$ is the
  logistic function. The Fisher information for the parameter $\theta$ in
  this model is $\fisher[\theta] = \phi(\theta)\phi(-\theta)$, so a change
  of variables gives $\fisher[{\phi(\theta)}] = \fisher[\theta]/
  (\phi'(\theta))^2$, and as $0 \le \predloss \le 1$,
  the non-private local minimax complexity is
  thus
  \begin{equation}
    \label{eqn:classical-local-predloss}
    \localminimax_n(P_0, \predloss, \mc{P}, \{\textup{id}\})
    \asymp \frac{1}{\sqrt{n}} \cdot \frac{1}{\sqrt{2+e^{\theta_0} + e^{-\theta_0}}}
    \asymp \frac{1}{\sqrt{n}} e^{-|\theta_0| / 2}
  \end{equation}
  by Proposition~\ref{proposition:classical-local-minimax} (and an analogous
  calculation to Claim~\ref{claim:qmd-bounded-local-minimax}).
  The delta method shows that the
  standard maximum likelihood estimator asymptotically achieves this
  risk.

  The private complexity is qualitatively different.  Noting
  that $XY \in \{\pm 1\}$ is a sufficient statistic for the logistic
  model~\eqref{eqn:logistic-model}, then applying
  Example~\ref{example:l1-bernoulli-info} via
  Proposition~\ref{proposition:information-private-lb} and
  Corollary~\ref{corollary:l1-info-attained} to the loss $\predloss$, we
  obtain numerical constants $0 < c_0 \le c_1 < \infty$
  such that for all large enough $n$,
  \begin{equation*}
    \frac{c_0}{\sqrt{n \diffp^2}}
    \le \localminimax_n(P_0, \predloss, \mc{P}, \channeldistset_\diffp)
    \le 
    \frac{c_1}{\sqrt{n \diffp^2}}.
  \end{equation*}
  By comparing this private local minimax complexity
  with the classical complexity~\eqref{eqn:classical-local-predloss},
  we see there is an exponential gap (in the parameter $|\theta_0|$)
  between the prediction risk achievable in private and
  non-private estimators---a
  non-trivial statistical price to pay for privacy.
\end{example}

\providecommand{\risk}{R}
\newcommand{\mlloss}{\ell}  
\newcommand{\noise}{\xi}

\noindent
While these examples have
parameter-independent 
$\bigLone$-information, this is
not always the case.

\begin{example}[Exponential scale,
    Example~\ref{example:exponential-modulus} continued]
  \label{example:l1-exponential-info}
  Let $p_\theta(x) = \frac{1}{\theta} \exp(-\frac{x}{\theta}) \indic{x
    \ge 0}$ be the density of an exponential distribution with scale
  $\theta$, and $\mc{P} = \{P_\theta\}_{a \le \theta \le b}$, where
  $a < b$ are any finite positive constants. The standard
  score is $\score_\theta(x) = \frac{\partial}{\partial \theta} \log
  p_\theta(x) = -\frac{1}{\theta} + \frac{x}{\theta^2}$, yielding Fisher
  information $\fisher[\theta] = \frac{1}{\theta^2}$, so that the classical
  local minimax complexity for the squared error (recall
  Claim~\ref{claim:qmd-bounded-local-minimax}) is
  $\localminimax_n(P_{\theta}, \losssq, \mc{P}, \{\mbox{id}\}) \asymp
  \frac{\theta^2}{n}$.
  In this case,
  the private local minimax complexity satisfies
  \begin{equation*}
    \localminimax_n(P_{\theta}, \losssq, \mc{P}, \channeldistset_\diffp)
    \asymp \frac{\theta^2}{n \diffp^2}.
  \end{equation*}
  To see this, note that the
  $\bigLone$-information~\eqref{eqn:l1-information} is $\information[\theta]
  = \E_\theta[|X/\theta^2 - 1/\theta|]$, so $\frac{1}{\theta}
  \lesssim \information[\theta] \lesssim \frac{1}{\theta}$.
  Proposition~\ref{proposition:information-private-lb} and
  Corollary~\ref{corollary:l1-info-attained} then give the bounds.  Thus,
  the private and non-private local minimax complexities
  differ (ignoring numerical constants) by the factor
  $1 / \diffp^2$. In distinction from Examples~\ref{example:l1-bernoulli-info}
  and~\ref{example:l1-logistic-info}, problems that are relatively easy in the
  classical setting ($\theta$ near $0$) continue to be easy.
\end{example}

\begin{example}[One-dimensional linear regression]
  \label{example:l1-regression-info}
  Consider a linear regression
  model where the data come in independent pairs $(X_i,
  Y_i)$ satisfying
  \begin{equation*}
    Y_i = \theta X_i + W_i~~\mbox{where}~~W_i \sim \normal(0, \sigma^2)
  \end{equation*}
  and the target is to estimate $\theta \in \R$. Fixing the distribution of
  $X$ and letting $\Theta \subset \R$ be a compact interval, we let
  $\pfamily = \{P_\theta\}_{\theta \in \Theta}$. We have negative
  log-likelihood $\mlloss_\theta(x, y) = \frac{1}{2 \sigma^2} (x \theta -
  y)^2$ and score $\score_\theta(x, y) = \frac{1}{\sigma^2} x(x
  \theta - y)$, so $\score_{\theta_0}(x, y) =
  \frac{x w}{\sigma^2}$ for the noise $w = y - x\theta_0$. Calculating
  the $\bigLone$ information $\information[\theta] = \sqrt{2/\pi} \E[|X|]
  / \sigma$ and applying Proposition~\ref{proposition:information-private-lb}
  and Corollary~\ref{corollary:l1-info-attained} yields
  \begin{equation}
    \label{eqn:linear-regression-local-minimax}
    \localminimax_n(P_{\theta_0}, \losssq, \mc{P}, \channeldistset_\diffp)
    \asymp \frac{\sigma^2}{n \diffp^2 \E[|X|]^2}.
  \end{equation}
  Comparing the rates~\eqref{eqn:linear-regression-local-minimax} with the
  non-private local minimax rate is instructive.
  Claim~\ref{claim:qmd-bounded-local-minimax}
  shows that
  $\localminimax_n(P_{\theta_0}, \losssq, \mc{P}, \{\mbox{id}\})
  \asymp \frac{\sigma^2}{n \E[X^2]}$.
  The local private minimax
  complexity~\eqref{eqn:linear-regression-local-minimax} depends on $X$
  through $\E[|X|]^2$, while the non-private complexity above depends
  inversely on $\E[X^2]$. Thus, as $X$ becomes more dispersed in that the
  ratio $\E[X^2] / \E[|X|]^2$ grows, the gap between private and non-private
  rates similarly grows.  Intuitively, a dispersed $X$
  requires more individual randomization to protect private information
  in $X$, increasing the statistical price of privacy.
\end{example}

\subsection{Attainment by stochastic gradient methods}
\label{sec:stochastic-gradient-stuff}

The second of our major desiderata is to (locally) uniformly achieve the
local minimax risk, and to that end we develop results on a private (noisy)
stochastic gradient method, which rely on \citet{PolyakJu92}; recall also
Duchi et al.'s (minimax-optimal) private stochastic gradient
method~\cite{DuchiJoWa18}. We
prefer (for brevity) to avoid the finest convergence conditions, instead
giving references as possible; we show how to attain the rates
in Examples~\ref{example:l1-exponential-info}
and~\ref{example:l1-regression-info}.

We wish to minimize a risk
$\risk_P(\theta) \defeq \E_P[\mlloss(\theta, X)]$, where $\mlloss : \Theta
\times \mc{X} \to \R$ is convex in its first argument. A noisy
stochastic gradient algorithm iteratively updates a parameter $\theta$ for
$i = 1, 2, \ldots$,
\begin{equation}
  \theta^{i + 1} = \theta^i - \stepsize_i (\nabla \mlloss(\theta^i, X_i)
  + \noise^i),
  \label{eqn:noisy-sgd-iteration}
\end{equation}
where $\noise^i$ is i.i.d.\ zero-mean noise, we assume $X_i
\simiid P$, and $\stepsize_i = \stepsize_0 i^{-\beta}$ are stepsizes
with $\beta \in (\half, 1)$ and $\stepsize_0 > 0$. Then under
appropriate conditions~\cite[Thm.~2 and Lem.~2]{PolyakJu92},\footnote{ The
  following suffice: (i) $\risk_P$ is $\mc{C}^2$ near $\theta\opt =
  \argmin_\theta \risk_P(\theta)$ with $\nabla^2 \risk_P(\theta\opt) \succ
  0$, (ii) there is some finite $C$ such that $\E_P[\norm{\nabla
      \mlloss(\theta, X)}^2] \le C(1 + \norm{\theta - \theta\opt}^2)$ and
  $\norm{\nabla \risk_P(\theta)} \le C(1 + \norm{\theta - \theta\opt})$ for
  all $\theta$, and (iii) $\limsup_{\theta \to \theta\opt} \E[\norm{\nabla
      \mlloss(\theta, X) - \nabla \mlloss(\theta\opt, X)}^2] = 0$.  } 
for $\theta\opt = \argmin_\theta \risk_P(\theta)$,
\begin{equation}
  \label{eqn:standard-sgd-convergence}
  \frac{1}{\sqrt{n}}
  \sum_{i = 1}^n (\theta^i - \theta\opt)
  = \frac{1}{\sqrt{n}} \sum_{i = 1}^n
  \nabla^2 \risk_P(\theta\opt)^{-1} (\nabla \mlloss(\theta\opt, X_i)
  + \noise^i) + o_P(1).
\end{equation}

The average $\wb{\theta}^n = \frac{1}{n} \sum_{i = 1}^n \theta^i$ is
asymptotically linear~\eqref{eqn:asymptotic-linearity}, so it satisfies the
regularity properties we outline in
Lemma~\ref{lemma:regular-differentiability}. The key is that $\nabla^2
\risk_P(\theta\opt)^{-1} \nabla \mlloss(\theta\opt, X_i)$ is typically the
influence function for the parameter $\theta$
(see~\cite[Proposition 1 and Lemma 8.1]{DuchiRu20}),
and we thus call this case regular:
\begin{definition}
  \label{definition:regular-risk}
  Let $\Pmodel = \{P_h\}_{h \in \R^d}$ be a sub-model of $\mc{P}$ around
  $P_0$, quadratic mean differentiable~\eqref{eqn:qmd-def} with
  score $g : \mc{X} \to \R^d$ at $P_0$. Let $\risk_h(\theta) =
  \E_{P_h}[\mlloss(\theta, X)]$ and define $\theta_h = \argmin_{\theta \in
    \Theta} \risk_h(\theta)$. The parameter $\theta_h$
  is \emph{regular} if it has
  influence function $\influencefunc(x) = \nabla^2 \risk_0(\theta_0)^{-1}
  \nabla \mlloss(\theta_0, x)$, equivalently,
  \begin{equation*}
    \theta_h = \theta_0 + \nabla^2 \risk_0(\theta_0)^{-1} \cov_0(\nabla
    \mlloss(\theta_0, X), g(X)) h + o(\norm{h}).
  \end{equation*}
\end{definition}
\noindent
By combining Lemma~\ref{lemma:regular-differentiability} with
the convergence guarantee~\eqref{eqn:standard-sgd-convergence}, we obtain
the following result.
\begin{proposition}
  \label{proposition:sgd-is-good}
  Let $\theta^i$ follow the noisy stochastic gradient
  iteration~\eqref{eqn:noisy-sgd-iteration} and satisfy the
  convergence~\eqref{eqn:standard-sgd-convergence}.  Let $\Pmodel$
  be a sub-model for which the risk $\risk_P$ is regular
  (Def.~\ref{definition:regular-risk}) and let
  \begin{equation*}
    Z \sim \normal\left(0,
    \nabla^2 \risk_{P_0}(\theta_0)^{-1}
    \left(\cov_0(\nabla \mlloss(\theta_0, X))
    + \cov(\noise) \right)
    \nabla^2 \risk_{P_0}(\theta_0)^{-1}\right).
  \end{equation*}
  Then for any bounded sequence $h_n \in \R^d$,
  $\sqrt{n}(\wb{\theta}^n - \theta_{h_n / \sqrt{n}})
  \cd Z$ under $X_i \simiid P_{h_n/\sqrt{n}}$,
  and for any bounded continuous $\loss$ and $c < \infty$,
  \begin{equation*}
    \lim_{n \to \infty} \sup_{\norm{h} \le c / \sqrt{n}}
    \E_{P_h}\left[\loss(\sqrt{n}(\wb{\theta}^n - \theta_h))\right]
    = \E[\loss(Z)].
  \end{equation*}
\end{proposition}

We complete Examples~\ref{example:l1-exponential-info} and
\ref{example:l1-regression-info} using
Proposition~\ref{proposition:sgd-is-good}:

\begin{example*}[Example~\ref{example:l1-exponential-info} continued]
  We return to the shape parameter in the exponential family. 
  As the median of $X$ is
  $\log 2 \cdot \theta_0$, estimating
  $\theta_0$ is equivalent to estimating $\median(X)$,
  or solving
  \begin{equation*}
    \minimize_{\theta} \risk_{P_0}(\theta) \defeq \E_{P_0}[\mlloss(\theta, X)]
    ~~ \mbox{for}~~ \mlloss(\theta, x) = |\theta - x|.
  \end{equation*}
  The stochastic gradient
  iteration~\eqref{eqn:noisy-sgd-iteration} is
  $\theta^{i+1} = \theta^{i} - \stepsize_i Z_i$
  where
  $Z_i = \sign(X_i - \theta^{i}) + \frac{1}{2\diffp} \cdot \laplace(1)$
  is $\diffp$-differentially private as
  and $\laplace(1)$ is the standard Laplacian distribution.
  For $P = \exponential(\theta)$,
  $\risk'(t) = P_\theta(t > X) - P_\theta(t < X)
  = 1 - 2 e^{-t/\theta}$ and
  $\risk''(t) = \frac{2}{\theta} e^{-t/\theta}$,
  at $t = \median(X) = \log 2 \cdot \theta$ we obtain
  $\risk''(\median(X)) = \frac{1}{\theta}$.
  For any symmetric quasiconvex $\loss : \R_+ \to \R_+$, define
  $\loss_n(t) = \loss(\sqrt{n} \cdot t)$,
  and let $\what{\theta}_n = \wb{\theta}^n / \log 2$.
  Applying Example~\ref{example:l1-exponential-info} and
  Proposition~\ref{proposition:sgd-is-good} yields
  \begin{equation*}
    \sup_{c < \infty} \limsup_{n \to \infty}
    \sup_{|\theta - \theta_0| \le c / \sqrt{n}}
    \frac{\E_\theta[\loss(\sqrt{n}(\what{\theta}_n - \theta))]}{
      \localminimax_n(P_{\theta_0}, \loss_n, \mc{P}, \channeldistset_\diffp)}
    \le \frac{\E[\loss(C_1 \frac{\theta_0}{\diffp} W)]}{
      \loss(C_0 \frac{\theta_0}{\diffp})}
  \end{equation*}
  where $W \sim \normal(0, 1)$ is standard normal and $C_i$ are numerical
  constants.  Whenever $\loss$ is such that $\E[\loss(C_1 \sigma W)]
  \lesssim \loss(C_0 \sigma)$---for example, $\loss(t) = \min\{t^2, B\}$---the
  private stochastic gradient method is
  local minimax rate optimal.
\end{example*}

\begin{example*}[Example~\ref{example:l1-regression-info} continued]
  We have $Y = X\theta_0 +
  \sigma W$ for $W \sim \normal(0, 1)$, where for simplicity we assume
  $\sigma$ is known. We
  transform the problem into a stochastic optimization problem, taking
  care to choose the correct objective for privacy and efficiency.
  Let $\symmetricloss : \R \to \R_+$ be any $1$-Lipschitz symmetric convex
  function with Lipschitzian gradient;
  for example, the Huber loss $\symmetricloss(t) =
  \half t^2$ for $|t| \le 1$ and $\symmetricloss(t) = |t| - \half$ satisfies
  these conditions.
  Choosing loss $\mlloss(\theta, x, y) = \frac{\sigma}{|x|}
  \symmetricloss(\frac{x \theta - y}{\sigma})$,  
  our problem is to minimize
  the risk
  $\risk_{P}(\theta) \defeq \E_P[\mlloss(\theta, X, Y)]$.
  Evidently $\mlloss$ is also $1$-Lipschitz with respect to $\theta$, with
  $\mlloss'(\theta, x, y) = \sign(x) \symmetricloss'(\sigma^{-1}(x \theta -
  y))$ and $\mlloss''(\theta, x, y) = \sigma^{-1}
  \symmetricloss''(\sigma^{-1}(x \theta - y))$. The private stochastic
  gradient iteration~\eqref{eqn:noisy-sgd-iteration} is then
  \begin{equation*}
    \theta^{i + 1} = \theta^i - \stepsize_i Z_i
    ~~ \mbox{where} ~~
    Z_i = \sign(X_i) \symmetricloss'\left(\frac{1}{\sigma}(X_i \theta^i - Y_i)
    \right)
    + \frac{1}{2 \diffp} \laplace(1).
  \end{equation*}
  Letting $Z_\infty = \symmetricloss'(W_i) + \frac{1}{2 \diffp} \laplace(1)$,
  a calculation shows that for $P_0$ corresponding to $\theta_0$,
  \begin{equation*}
    \risk_{P_0}''(\theta_0)
    = \frac{\E[|X|]}{\sigma} \E[\symmetricloss''(W)]
    ~~ \mbox{and} ~~
    \var(Z_\infty) = \E[\symmetricloss'(W)^2] + \frac{1}{2 \diffp^2}.
  \end{equation*}
  We apply Proposition~\ref{proposition:sgd-is-good} to obtain
  that along any convergent sequence $h_n \to h$,
  \begin{equation*}
    \sqrt{n}\left(\wb{\theta}^n - (\theta_0 + h_n / \sqrt{n})\right)
    \mathop{\cdlong}_{P_{\theta_0 + h_n / \sqrt{n}}}
    \normal\left(0, \frac{\sigma^2}{\E[|X|]^2 \E[\symmetricloss''(W)]^2}
    \left(\E[\symmetricloss'(W)^2] + \frac{1}{2 \diffp^2}\right)\right).
  \end{equation*}
  Notably $\E[\symmetricloss'(W)^2] \le 1$ as $\symmetricloss$ is
  $1$-Lipschitz, and whenever $\symmetricloss$ is strongly convex
  near $0$, then $\E[\symmetricloss''(W)] > 0$ is a positive
  constant. In particular, the stochastic gradient estimator
  is local minimax rate optimal, achieving asymptotic
  variance $O(1) \frac{\sigma^2}{\diffp^2 \E[|X|]^2}$ uniformly near
  $\theta_0$, which (as in Example~\ref{example:l1-regression-info}) is
  unimprovable except by numerical constants.
\end{example*}

\subsection{Asymptotic achievability in one-parameter exponential families}
\label{sec:one-param-expfams}

Proposition~\ref{proposition:information-private-lb} shows an instance-specific
lower bound of $(n \diffp^2 \information[\theta_0]^2)^{-1}$, where
$\information[\theta_0] = \E_{\theta_0}[|\score_{\theta_0}|]$ is the
$\bigLone$ information, for the estimation of a single parameter.
This section develops a novel locally private estimation scheme to
achieve the lower bound for general one-parameter exponential family 
models. Subtleties in the construction make showing that the estimator
is regular or uniform challenging, though we conjecture that it
is locally uniform.
Let $\mc{P} = \{P_\theta\}_{\theta \in \Theta}$ be a one parameter exponential 
family, so that for a
base measure $\mu$ on $\mc{X}$, each distribution $P_{\theta}$ has
density
\begin{equation*}
  p_\theta(x) \defeq
  \frac{dP_\theta}{d\mu} (x) = \exp\left(\theta T(x) - A(\theta)\right),
\end{equation*}
where $T(x)$ is the sufficient statistic and $A(\theta) = \log \int
e^{\theta T(x)} d\mu(x)$ is the log partition function.\footnote{Writing the
  family this way is no loss of generality. While typically one writes
  $p_\theta(x) = h(x) \exp(\theta^T T(x) - A(\theta))$, we can always
  include $h$ in the base measure $\mu$ and push-forward through the
  statistic $T$.}  It is well known
(cf.~\cite[Ch.~2.7]{Brown86,LehmannRo05}) that $A$ satisfies $A'(\theta) =
\E_\theta[T(X)]$ and $A''(\theta) = \var_\theta(T(X))$.  In this case, the
$\bigLone$-information~\eqref{eqn:l1-information} is the mean absolute
deviation
\begin{equation*}
  \information[\theta] = \E_\theta[|T(X) - A(\theta)|] = 
  \E_\theta[|T(X) - \E_\theta [T(X)]|].
\end{equation*}

We provide a procedure asymptotically achieving mean square error
scaling as $(n\diffp^2 \information[\theta]^2)^{-1}$, which
Proposition~\ref{proposition:information-private-lb} shows is optimal.  Our
starting point is the observation that for a one-parameter exponential
family, $\theta \mapsto P_\theta(T(X) \ge t)$ is strictly
increasing in $\theta$ for any fixed $t \in \supp \{T(X)\}$~\cite[Lemma
  3.4.2]{LehmannRo05}. A natural idea is to first estimate
$P_{\theta}(T(X) \ge t)$ and invert to estimate
$\theta$.  To that end, we develop a private two-sample procedure,
where with the first we estimate $\hat{t} \approx \E[T(X)]$, using the
second sample to approximate and invert $P_\theta(T(X) \ge \hat{t})$.
Now, define $\Psi : \R^2 \to
\R_+$ by
\begin{equation}
  \Psi(t, \theta) \defeq P_\theta (T(X) \geq t)
  = \int \indic{T(x) \geq t}
  \exp\left(\theta T(x) - A(\theta)\right) d\mu (x).
\end{equation}
The private two stage algorithm we develop splits a total sample of
size $2n$ in half, using the first half of
the sample to construct a consistent estimate $\what{T}_n$ of the value
$A'(\theta) = \E_{\theta}[T]$
(Duchi et al.'s $\diffp$-differentially private mean
estimators provide consistent estimates of $\E[T(X)]$ so long as
$\E[|T(X)|^k] < \infty$ for some $k > 1$~\cite[Corollary~1]{DuchiJoWa18}.)
In the second stage, the algorithm uses $\what{T}_n$ and
the second half of the sample in a randomized response procedure:
construct $V_i$ and private $Z_i$ as
\begin{equation*}
  V_i = \indics{T(X_i) \ge \what{T}_n},
  ~~~
  Z_i =
  \frac{e^\diffp + 1}{e^\diffp - 1}
  \cdot \left[\left\{
    \begin{array}{cl}
      V_i & \mbox{w.p.}~
      \frac{e^\diffp}{e^\diffp + 1} \\
      1 - V_i & \mbox{w.p.}~
      \frac{1}{e^\diffp + 1}
    \end{array}
    \right\}
    - \frac{1}{e^\diffp + 1}
    \right].
\end{equation*}
By inspection, this is $\diffp$-differentially-private and
$\E[Z_i \mid V_i] = V_i$. Now, define the inverse function
\begin{equation*}
  H(p, t) \defeq \inf \left\{\theta \in \R \mid P_\theta(T(X) \ge t)
  \ge p \right\}
  = \inf \left\{\theta \in \R \mid \Psi(t, \theta) \ge p \right\}.
\end{equation*}
Setting $\wb{Z}_n = \frac{1}{n} \sum_{i = 1}^n Z_i$, our final
$\diffp$-differentially private estimator is
\begin{equation}
  \label{eqn:private-expfam-estimator}
  \what{\theta}_n = H(\wb{Z}_n, \what{T}_n).
\end{equation}
We then have a convergence result showing that the
estimator~\eqref{eqn:private-expfam-estimator} has asymptotic
variance within a constant factor of the local minimax bounds. We
defer the (involved) proof to
Appendix~\ref{sec:proof-uniform-achievability-one-dim-exp}.

\newcommand{\error}{\mathcal{E}}
\begin{proposition}
  \label{proposition:uniform-achievability-one-dim-exp}
  Assume that $\var_{\theta} \left(T(X)\right) > 0$ and $\what{T}_n \cp t_0
  \defeq \E_{\theta_0}[T(X)]$.  Define $\delta_\diffp^2 = \frac{e^{\diffp}}{(e^{\diffp} - 1)^2}$.  
  Then there exist random variables $G_n = \Psi(\what{T}_n, \theta_0) \in [0, 1]$,
  $\error_{n, 1}$, and $\error_{n,2}$ such that under $P_{\theta_0}$,
  \begin{equation*}
    \sqrt{n} \left(\what{\theta}_n - \theta_0\right) 
    = 2 \information[\theta_0]^{-1} (\error_{n, 1} + \error_{n, 2}) + o_P(1)
  \end{equation*}
  where
  \begin{equation}
    \left(\error_{n, 1},
    \frac{1}{G_n (1 - G_n)} \error_{n,2}\right)
    \cd \normal\left(0, \diag(\delta_\diffp^{-2}, 1)\right).
  \end{equation}
\end{proposition}
\noindent
The complexity of the statement arises because the distribution of $T(X)$
may be discontinuous, including at $\E_{\theta_0}[T(X)]$, necessitating
the random variables $\error_{n,1}, \error_{n,2}$, and $G_n$
for the limit.


\section{Private local minimax theory for more general functionals}
\label{sec:general-influence-functions}

We broaden our investigation to consider the
local minimax approach in semi- or nonparametric problems with high or
infinite-dimensional parameters, but where the target of interest is
one-dimensional.  We first present analogues, to within numerical constants, of
classical semi-parametric information lower bounds
(Section~\ref{sec:private-info-general}). We illustrate the bounds for
estimating a functional of an exponential family parameter, where
subtleties distinguish the problem from the non-private
case. Most saliently, as we show, efficiency depends strongly on the model
assumed by the statistician: while in the non-private case, parametric and
nonparametric models yield the same efficiency bounds (as we will revisit),
the private parametric and semi-parametric cases are quite different.



\subsection{Private information, influence functions, and tangent spaces}
\label{sec:private-info-general}

Our goal here is to generalize the results in
Section~\ref{sec:examples} to provide private information lower bounds
for semi-parametric estimation problems. Similar to what we did in
Section~\ref{sec:examples}, our development builds off
Theorem~\ref{theorem:modulus-of-continuity} and
Proposition~\ref{proposition:achievable} by performing a local expansion of
the variation distance. We
parallel some of the classical development in
Section~\ref{sec:attainment-sortof},
presenting one-dimensional submodels, tangent spaces, and an
$\bigLone$-influence function, after which we derive our information
bounds.


We begin as usual with a family $\mc{P}$ of distributions, and we consider
one-dimensional sub-models $\Pmodel \subset \mc{P}$ indexed by $h \in
\R$. In analogy with quadratic mean differentiability~\eqref{eqn:qmd-def}
and our treatment in Section~\ref{sec:examples}, we say $h \mapsto P_h$ is
\emph{$\bigLone$-differentiable at $P_0$} with score $g : \mc{X} \to \R$ if
\begin{equation}
  \label{eqn:L-1-diff}
  \int |dP_h - dP_0 - h g d P_0| = o(|h|)
\end{equation}
as $h \to 0$. As in Section~\ref{sec:attainment-sortof}, we let
$h \mapsto P_h$ range over (a collection of) possible submodels
to obtain a collection of score functions $\{g\}$,
and we define the \emph{$\bigLone$-tangent space $\Lonetangentset$} to
be the closed linear span of these scores.
In contrast to the tangent space from
quadratic-mean-differentiability~\eqref{eqn:l2-diff}, which admits
Hilbert-space geometry,
$\bigLone$-differentiability gives a different duality.
Moreover, Lemma~\ref{lemma:nice-families-l1-differentiable}
states that QMD families must be $\bigLone$-differentiable, so that
the $\bigLone$-tangent space $\Lonetangentset$ always contains
the classical tangent space.
An example may be clarifying:

\begin{example}[Fully nonparametric $\bigLone$ tangents]
  \label{example:nonparametric-lone-tangent}
  In the fully nonparametric case---where $\mc{P}$ consists of all
  distributions supported on $\mc{X}$---we can identify
  $\Lonetangentset$ with mean-zero $g \in \bigLone(P_0)$. Indeed,
  for any such $g$,
  define the models $dP_h = \hinge{1 + hg} dP_0 / C_h$, where
  $C_h = \int \hinge{1 + hg} dP_0$.
  Then $|\hinge{1 + hg} - 1| \le h|g|$, and so
  by dominated convergence we have
  $1 \le C_h = 1 + o(h)$, as
  \begin{equation*}
    0 \le \frac{1}{h}(C_h - 1)
    = \int \frac{\hinge{1 + hg} - 1}{h} dP_0
    \mathop{\longrightarrow}_{h \to 0} \int g dP_0 = 0.
  \end{equation*}
  Similarly
  $\{dP_h\}_{h \in \R}$ has score $g$ at $P_0$ because
  $\lim_{h \to 0} \int \frac{1}{h} |\frac{\hinge{1 + hg} - 1}{C_h} - hg| dP_0
  = 0$ by the dominated convergence theorem as well.
  Conversely, for any $g \in \Lonetangentset$, we have
  $\int|dP_h - dP_0| \le 2$, while
  $o(h) = \int |dP_h - dP_0 - h g dP_0| \ge
  |h| \int |g| dP_0 - 2$, so that $g \in \bigLone(P_0)$. That
  $g$ is mean zero is
  immediate, as
  $\int g d P_0 = \frac{1}{h} \int (dP_0 + h g dP_0 - dP_h)
  = o(1)$ as $h \to 0$.  

  In contrast, in the fully nonparametric case for quadratic mean
  differentiability~\cite[Example~25.16]{VanDerVaart98},
  the tangent set at $P_0$ is all mean zero $g \in \bigLtwo(P_0)$, a
  smaller set of potential tangents.
\end{example}

To give a private information for estimating a function $\theta : \mc{P} \to
\R^d$, we consider submodels $\{P_h\}_{h \in \R}$ where $h \mapsto
\theta(P_h)$ is suitably smooth at $h = 0$.  We say $\theta(\cdot)$
is \emph{differentiable at $P_0$ relative to $\Lonetangentset$}
if there exists a
continuous linear mapping $\varphi_{P_0} : \Lonetangentset \to \R^d$ such that
for any $\bigLone$-differentiable submodel $\Pmodel$ with score $g$ at
$P_0$,
\begin{equation*}
  \theta(P_h) - \theta(P_0) = h \varphi(g) + o(h).
\end{equation*}
As $g \mapsto \varphi(g)$ is continuous for $g \in \Lonetangentset$,
it has a continuous extension to all of $\bigLone(P_0)$ and so
by duality there exists $\influencefunc : \mc{X} \to \R^d$, with
coordinate functions in $\bigLinf(P_0)$, such that
\begin{equation*}
  \varphi(g) = \E_{P_0}[\influencefunc(X) g(X)] =
  \int \influencefunc(x) g(x) dP_0(x).
\end{equation*}
We call this $\influencefunc$ the \emph{private influence function}.
Again, contrast with the classical approach is instructive:
there (recall Eq.~\eqref{eqn:differentiable-functional}), the
Hilbert space structure of the tangent sets
allows one to use the Riesz representation
theorem to guarantee the existence of an influence function
$\influencefunc \in \bigLtwo(P_0)$.

The main result of this section gives an information-type lower bound for
general estimation problems where we wish to estimate a functional
$\functional(\theta(P))$, where $\functional : \R^d \to \R$ is
$\mc{C}^1$. We measure error by a symmetric quasiconvex $\loss : \R \to
\R_+$, suffering loss $\loss(\what{\functional} -
\functional(\theta(P)))$ for an estimate $\what{\functional}$.  We then
obtain the following generalization of
Proposition~\ref{proposition:information-private-lb}.  (See
Appendix~\ref{sec:proof-general-private} for a proof.)

\begin{theorem}
  \label{theorem:general-private}
  Let $\Pmodel = \{P_h\}_{h \in \R} \subset \mc{P}$ be an
  $\bigLone$-differentiable submodel at $P_0$ with score $g$ and
  let $\theta : \mc{P} \to \R^d$ have $\bigLone$-influence function
  $\influencefunc$ at $P_0$. Let
  $\channeldistset_\diffp$ be the family of $(2, \diffp^2)$-locally
  R\'{e}nyi private channels (Def.~\ref{definition:local-renyi}). Then
  for an $N = N(\functional, \diffp, \theta, \Pmodel)$ independent of
  loss $\loss$, for all $n \ge N$
  \begin{equation*}
    \localminimax_n(P_0, \loss, \Pmodel, \channeldistset_\diffp)	\ge 
    \frac{1}{8} \cdot 
    \loss\left(\frac{1}{6 \sqrt{n\diffp^2}}
    \frac{\nabla\functional(\theta_0)^T \E_0[\influencefunc(X) g(X)]}{
      \E_0[|g(X)|]}\right).
  \end{equation*}
\end{theorem}
\noindent
The quantity $\information[g,0] \defeq \E_0[|g(X)|] =
\int |g| dP_0$ is the nonparametric analogue of the private
information~\eqref{eqn:l1-information} in
Proposition~\ref{proposition:information-private-lb}, as the score function
$g$ is completely parallel to the parametric case.
Additional remarks show how this result parallels and complements
classical local minimax theory.

\paragraph{Recovering the parametric case}
Theorem~\ref{theorem:general-private}
specializes to
Proposition~\ref{proposition:information-private-lb}
for one-dimensional parametric families. Let the family
be $\mc{P}_{\para} = \{P_\theta\}_{\theta \in \Theta}$ and
$L^1$
differentiable at $P_{\theta_0}$. Then the private tangent space
$\Lonetangentset$ is then the linear space spanned by the score
$\score_{\theta_0}$, and the influence function for $\theta$ is
$\influencefunc = \fisher[\theta_0]^{-1}\score_{\theta_0}$, where the Fisher
information is $\fisher[\theta_0] = \E_0[\score_{\theta_0}(X)^2]$.
Specializing Theorem~\ref{theorem:general-private} gives $\nabla
\functional(\theta_0)^T \E_0[\influencefunc(X) g(X)] / \E_0[|g(X)|] =
\functional'(\theta_0) / \E_0[|\score_{\theta_0}(X)|]$,
recovering Proposition~\ref{proposition:information-private-lb}.

\newcommand{\Pallmodel}{\mc{P}_{\textup{all},0}}

\paragraph{Dualities and classical information bounds}
As a corollary of the $\bigLone$/$\bigLinf$ duality that privacy evidently
entails in Theorem~\ref{theorem:general-private} and
Lemma~\ref{lemma:nice-families-l1-differentiable}, we have the following
lower bound.
\begin{corollary}
  \label{corollary:nonparametric-private-information}
  Let the conditions of Theorem~\ref{theorem:general-private} hold,
  and additionally let $\Pallmodel$ be a collection of
  QMD sub-models with scores $g$ that are dense in $\bigLone(P_0)$.
  Then there exists an $N$ independent of the loss $\loss$ such that
  for all $n \ge N$,
  \begin{equation*}
    \localminimax(P_0, \loss, \Pallmodel, \channeldistset_\diffp)
    \ge \frac{1}{8} \cdot \loss\left(\frac{1}{8 \sqrt{n \diffp^2}}
    \esssup_x |\nabla \functional(\theta_0)^T \influencefunc(x)| \right).
  \end{equation*}
\end{corollary}

The local minimax lower bound \emph{necessarily} depends on the (essential)
supremum of the influence function $\nabla \functional(\theta_0)^T
\influencefunc(x)$ over $x \in \mc{X}$; notably, this occurs even when the
tangent set $\tangentset$ is dense in $\bigLtwo(P_0)$.
We may compare this with classical (nonparametric) information bounds, which
rely on the Hilbert-space structure of quadratic-mean-differentiability and
are thus smaller and qualitatively different. In the classical
setting~\cite[Ch.~25.3]{VanDerVaart98} (and
Sec.~\ref{sec:attainment-sortof}), we recall that we may identify
$\tangentset$ with mean-zero functions $g \in \bigLtwo(P_0)$, and we obtain
the analogous lower bound
\begin{equation*}
  \loss\left(\frac{1}{\sqrt{n}} \cdot
  \sup_{g \in \tangentset}
  \frac{\nabla\functional(\theta_0)^T\E_0[\influencefunc(X)g(X)]}{
    \E_0[g(X)^2]^{1/2}}\right)
  = \loss\left(\frac{1}{\sqrt{n}} \E_0\left[
    \left(\nabla \functional(\theta_0)^T \influencefunc(X)\right)^2\right]^{1/2}
  \right),
\end{equation*}
where we have used that $\E_0[\influencefunc(X)] = 0$.  This information
bound is always (to numerical constants) smaller than the private
information bound in
Corollary~\ref{corollary:nonparametric-private-information}.


\subsection{Nonparametric modeling with exponential families}
\label{sec:mis-specified-expfam}

While Section~\ref{sec:examples} characterizes local minimax complexities
for several one-dimensional problems, treating one parameter exponential
families in Section~\ref{sec:one-param-expfams}, it relies on the model's
correct specification. Here, we consider estimating functionals of a
potentially mis-specified exponential family model.  To formally describe
the setting, we start with a $d$-parameter exponential family
$\{P_\theta\}_{\theta \in \Theta}$ with densities $p_\theta(x) =
\exp(\theta^Tx - A(\theta))$ with respect to some base measure $\mu$, where
for simplicity we assume that the exponential family is regular and minimal,
meaning that $\nabla^2 A(\theta) = \cov_\theta(X) \succ 0$ for all $\theta
\in \dom A$, and the log partition function $A(\theta)$ is analytic on the
interior of its domain~\cite[Thm.~2.7.1]{LehmannRo05}.  We record a few
standard facts on the associated convex analysis (for more, see the
books~\cite{Brown86, WainwrightJo08, HiriartUrrutyLe93ab}). Recall the
conjugate $A^*(x) \defeq \sup_{\theta} \{\theta^T x - A(\theta)\}$.
Then~\cite[cf.][Ch.~X]{HiriartUrrutyLe93ab}
\begin{equation}
  \label{eqn:conjugate-gradient-inverse}
  \nabla A^*(x) = \theta_x ~~\mbox{for~the~unique~} \theta_x ~
  \mbox{such that}~~ \E_{\theta_x}[X] = x.
\end{equation}
In addition, $\nabla A^*$ is continuously differentiable, one-to-one, and 
\begin{equation*}
  \dom A^* \supset \range(\nabla A(\cdot))
  = \{\E_\theta[X] \mid \theta \in \dom A\}.
\end{equation*}
Moreover, by the inverse function theorem, we also have that on the
interior of $\dom A^*$,
\begin{equation}
  \label{eqn:inverse-hessian-conjugate}
  \nabla^2 A^*(x)
  = (\nabla^2 A(\theta_x))^{-1}
  = \cov_{\theta_x}(X)^{-1}
  ~~ \mbox{for~the~unique~} \theta_x ~ \mbox{s.t.}~
  \E_{\theta_x}[X] = x.
\end{equation}
The uniqueness follows because $\nabla A^*$ is one-to-one, as
the exponential family is minimal and $\nabla^2 A(\theta)
\succ 0$.
For a distribution
$P$ with mean $\E_P[X]$, so long as the mean belongs to the
range of $\nabla A(\theta) = \E_\theta[X]$
as $\theta$ varies, the minimizer of the log loss
$\mlloss_\theta(x) = -\log p_\theta(x)$ is
\begin{equation*}
  \theta(P) \defeq \argmin_\theta
  \E_P[\mlloss_\theta(X)]
  = \nabla A^*(\E_P[X]).
\end{equation*}

We consider estimation of smooth functionals
$\functional : \R^d \to \R$ of the parameters $\theta$,
measuring the loss of an estimated value
$\what{\functional}$ by
\begin{equation*}
  \loss(\what{\functional} - \functional(\theta(P))),
\end{equation*}
where $\loss : \R \to \R_+$ as usual is quasi-convex and
symmetric. In the sequel, we show local lower bounds
on estimation, develop a (near) optimal regular estimator,
and contrast our results and the possibilities of adaptation
in private and non-private cases with somewhat striking differences.

\subsubsection{Private estimation rates}

We begin with a local minimax lower bound that
almost immediately follows Theorem~\ref{theorem:general-private}.

\begin{corollary}
  \label{corollary:mis-specified-expfam}
  Let $P_0$ be such that $\E_{P_0}[X] \in \interior
  (\range (\nabla A))$, and let $\Pallmodel$ be
  a collection of sub-models with scores $g$ dense in
  $\bigLone(P_0)$ at $P_0$.
  Let $\channeldistset_\diffp$ denote the collection of all $(2,
  \diffp^2)$-locally R\'{e}nyi private sequentially interactive channels.
  Then there exists $N = N(\Pallmodel, \functional)$ independent
  of the loss $\loss$ such that $n \ge N$ implies
  \begin{equation*}
    \localminimax_n(P_0, \loss, \Pallmodel, \channeldistset_\diffp)
    \ge \frac{1}{8} \cdot 
    \loss\left(\frac{1}{5\sqrt{2 n \diffp^2}}
    \cdot \esssup_{x}\left|\nabla \functional(\theta_0)^T
    \nabla^2 A(\theta_0)^{-1}
    (\E_{P_0}[X] - x) \right|\right).
  \end{equation*}
\end{corollary}
\begin{proof}
  The exponential family influence function is $\influencefunc(x) = \nabla^2
  A(\theta_0)^{-1} (x - \E_0[X])$~\cite[Ch.~25.3]{VanDerVaart98}.  Take $g$
  with $\frac{\nabla \functional(\theta_0)^T \E_0[\influencefunc(X)
      g(X)]}{\E_0[|g(X)|]} \ge \frac{3}{4} \esssup_x |\nabla
  \functional(\theta_0)^T \nabla^2 A(\theta_0)^{-1} (x - \E_0[X])|$ in
   Theorem~\ref{theorem:general-private}.
\end{proof}

Before we turn to private estimation, we compare
Corollary~\ref{corollary:mis-specified-expfam} to the non-private
case. The maximum likelihood estimator takes the sample mean
$\what{\mu}_n = \frac{1}{n} \sum_{i = 1}^n X_i$ and sets
$\what{\theta}_n = \nabla A^*(\what{\mu}_n)$. Letting $\theta_0 = \nabla
A^*(\E_{P_0}[X])$, Taylor expansion arguments and the
delta-method~\cite[Chs.~3--5]{VanDerVaart98} yield
\begin{equation*}
  \sqrt{n} (\what{\theta}_n - \theta_0)
  \cd \normal\left(0, \nabla^2 A(\theta_0)^{-1}
  \cov_0(X) \nabla^2 A(\theta_0)^{-1} \right)
\end{equation*}
and
\begin{equation*}
  \sqrt{n} (\functional(\what{\theta}_n)
  - \functional(\theta_0))
  \cd \normal\left(0, 
  \nabla \functional(\theta_0)^T \nabla^2 A(\theta_0)^{-1}
  \cov_0(X) \nabla^2 A(\theta_0)^{-1}
  \nabla \functional(\theta_0)\right),
\end{equation*}
and these estimators are regular (and hence locally uniform).
The lower bound in Corollary~\ref{corollary:mis-specified-expfam} is always
larger than this classical limit.
In this sense, the private lower bounds exhibit
both the importance of local geometry---via $\nabla^2 A(\theta_0)^{-1}
\nabla \functional(\theta_0)$---and the challenge of privacy
in addressing ``extraneous'' noise that must be privatized.
We will discuss this more in
Section~\ref{sec:adaptation-of-models}.

\subsubsection{An optimal one-step procedure}
\label{sec:one-step-mis-expfam}

An optimal procedure for functionals of (possibly) mis-specified exponential
family models is similar to classical one-step estimation
procedures~\cite[e.g.][Ch.~5.7]{VanDerVaart98}. 
To motivate the approach, let us assume
we have a ``good enough'' estimate $\wt{\mu}_n$ of $\mu_0 \defeq
\E_P[X]$. Then if $\wt{\theta}_n = \nabla A^*(\wt{\mu}_n)$,
we have
\begin{align}
  \functional(\theta_0)
  & = \functional(\wt{\theta}_n)
  + \nabla \functional(\wt{\theta}_n)^T (\theta_0 - \wt{\theta}_n)
  + O(\norms{\theta_0 - \wt{\theta}_n}^2) \nonumber \\
  & = \functional(\wt{\theta}_n)
  + \nabla \functional(\wt{\theta}_n)^T
  (\nabla A^*(\mu_0) - \nabla A^*(\wt{\mu}_n))
  + O(\norms{\mu_0 - \wt{\mu}_n}^2) \nonumber \\
  & = \functional(\wt{\theta}_n)
  + \nabla \functional(\wt{\theta}_n)^T
  \nabla^2 A(\wt{\theta}_n)^{-1}(\mu_0 - \wt{\mu}_n)
  + O(\norms{\mu_0 - \wt{\mu}_n}^2),
  \nonumber 
\end{align}
where each equality freely uses the duality
relationships~\eqref{eqn:conjugate-gradient-inverse}
and~\eqref{eqn:inverse-hessian-conjugate}.
In this case, if $\wt{\mu}_n - \mu_0 = o_P(n^{-1/4})$ and
we have an estimator $T_n$ satisfying
\begin{equation*}
  \sqrt{n}
  \left(T_n - 
  \nabla \functional(\wt{\theta}_n)^T \nabla^2 A(\wt{\theta}_n)^{-1}
  \mu_0\right)
  \cd \normal(0, \sigma^2),
\end{equation*}
then the estimator
\begin{equation}
  \what{\functional}_n
  \defeq \functional(\wt{\theta}_n)
  + T_n - \nabla \functional(\wt{\theta}_n)^T
  \nabla^2 A(\wt{\theta}_n)^{-1} \wt{\mu}_n
  \label{eqn:private-expfam-mis-estimator}
\end{equation}
satisfies $\sqrt{n} (\what{\functional}_n - \functional(\theta_0)) \cd
\normal\left(0, \sigma^2\right)$ by Slutsky's theorems.

We now exhibit such an estimator.  To avoid some of the difficulties
associated with estimation from unbounded data~\cite{DuchiJoWa18}, we assume
the domain $\xdomain \subset \R^d$ is the norm ball $\{x \in \R^d
\mid \norm{x} \le 1\}$.  For dual norm $\dnorm{z} = \sup_{\norm{x} \le 1}
x^T z$, the essential supremum in
Corollary~\ref{corollary:mis-specified-expfam}
thus has bounds
\begin{equation}
  \label{eqn:ess-sup-dual-equivalent}
  \esssup_x \left|\nabla \functional(\theta_0)^T
  \nabla^2 A(\theta_0)^{-1} \left(\E_{P_0}[X] - x\right)\right|
  \in \left[\half, 2 \right] \cdot
  \dnorm{\nabla^2 A(\theta_0)^{-1} \nabla \functional(\theta_0)}.
\end{equation}
Let us split the sample of size $n$ into two
sets of size $n_1 = \ceil{n^{2/3}}$ and $n_2 = n - n_1$.  For the first set,
let $Z_i$ be any $\diffp$-locally differentially private estimate of $X_i$
satisfying $\E[Z_i \mid X_i] = X_i$ and $\E[\norm{Z_i}^2] < \infty$, so that
the $Z_i$ are i.i.d.; for example, $X_i + W_i$ for a random vector of
appropriately large Laplace noise suffices~\cite{DworkMcNiSm06,DuchiJoWa18}.
Define $\wt{\mu}_n = \frac{1}{n_1} \sum_{i = 1}^{n_1} Z_i$, in which case
$\wt{\mu}_n - \mu_0 = O_P(n^{-1/3})$, and let $\wt{\theta}_n = \nabla
A^*(\wt{\mu}_n)$.  Now, for $i = n_1 + 1, \ldots, n$, define the
$\diffp$-differentially private
quantity
\begin{equation*}
  Z_i \defeq
  \nabla \functional(\wt{\theta}_n)^T
  \nabla^2 A(\wt{\theta}_n)^{-1} X_i
  + \frac{
    \dnorms{\nabla^2 A(\wt{\theta}_n)^{-1} \nabla \functional(\wt{\theta}_n)}
  }{\diffp}
  W_i
  ~~ \mbox{where} ~~
  W_i \simiid \laplace(1).
\end{equation*}
Letting $\wb{X}_{n_2} = \frac{1}{n_2} \sum_{i = n_1 + 1}^{n} X_i$
and similarly for $\wb{W}_{n_2}$ and $\wb{Z}_{n_2}$, we find that
\begin{align*}
  \lefteqn{\sqrt{n}
    \left(\wb{Z}_{n_2}
    - \nabla \functional(\wt{\theta}_n)^T \nabla^2 A(\wt{\theta}_n)^{-1}\mu_0
    \right)} \\
  & = \sqrt{n}
  \left[\nabla\functional(\wt{\theta}_n)^T
    \nabla^2 A(\wt{\theta}_n)^{-1}
    \left(\wb{X}_{n_2} - \mu_0\right)
    + \frac{ \dnorms{\nabla^2 A(\wt{\theta}_n)^{-1} \nabla \functional(\wt{\theta}_n)}}{\diffp}
    	 \wb{W}_{n_2}\right]
  \cd \normal(0, \sigma^2(P,\functional,\diffp))
\end{align*}
by Slutsky's theorem, where for
$\theta_0 = \nabla A^*(\E_P[X])$ we define
\begin{equation}
  \label{eqn:variance-private-expfam-mis-estimator}
  \sigma^2(P,\functional, \diffp)
  \defeq 
  \nabla \functional(\theta_0)^T \nabla^2 A(\theta_0)^{-1}
  \cov_P(X) \nabla^2 A(\theta_0)^{-1} \nabla \functional(\theta_0)
  + \frac{2}{\diffp^2}
  \dnorm{\nabla^2 A(\theta_0)^{-1} \functional(\theta_0)}^2.
\end{equation}
Moreover, the difference above is asymptotically
linear~\eqref{eqn:asymptotic-linearity},
so by continuity we have
\begin{equation*}
  \sqrt{n} (\what{\functional}_n - \functional(\theta_0))
  = \frac{1}{\sqrt{n}}
  \sum_{i = 1}^n \nabla \functional(\theta_0)^T \nabla^2 A(\theta_0)^{-1}
  (X_i - \mu_0)
  + \frac{\dnorm{\nabla^2 A(\theta_0)^{-1} \nabla \functional(\theta_0)}}{
    \diffp}
  \frac{1}{\sqrt{n}} \sum_{i = 1}^n W_i
  + o_{P_0}(1).
\end{equation*}

Summarizing, we can apply Lemma~\ref{lemma:regular-differentiability},
because the smoothness of $A(\cdot)$ means that the parameter $\theta_0$ is
regular in that it has influence function $\influencefunc(x) = \nabla^2
A(\theta_0)^{-1} (x - \mu_0)$ (recall also
Definition~\ref{definition:regular-risk}). Recalling the
equivalence~\eqref{eqn:ess-sup-dual-equivalent} between the dual norm
measures and essential supremum, we have thus shown that the two-step
estimator~\eqref{eqn:private-expfam-mis-estimator} is
locally minimax rate optimal.
\begin{proposition}
  \label{proposition:private-expfam-mis-estimator}
  Let $\what{\functional}_n$ be the
  estimator~\eqref{eqn:private-expfam-mis-estimator},
  $\{P_h\}$ be quadratic mean differentiable
  at $P_0$, $\theta_h = \argmax_\theta \E_{P_h}[\log p_\theta(X)]$,
  and $\sigma^2(P_0, \functional, \diffp)$ be as
  in~\eqref{eqn:variance-private-expfam-mis-estimator}.
  Let $Z \sim \normal(0, \sigma^2(P_0, \functional, \diffp)$
  and $h_n$ be a bounded sequence. Then
  $\sqrt{n} (\what{\functional}_n - \functional(\theta_{h_n/\sqrt{n}}))
  \cd Z$ under $X_i \simiid P_{h_n/\sqrt{n}}$, and
  for any bounded continuous $\loss$ and $c < \infty$,
  \begin{equation*}
    \lim_{n \to \infty} \sup_{\norm{h} \le c / \sqrt{n}}
    \E_{P_h}
    \left[\loss(\sqrt{n}(\what{\functional}_n -
      \functional(\theta_{h})))\right]
    = \E[\loss(Z)].
  \end{equation*}
\end{proposition}

\subsubsection{An extension to functionals of GLM parameters}
\label{sec:optimal-glm-experiments}

In our experiments, we will investigate the behavior of locally private
estimators for generalized linear models on a
variable $Y$ conditioned on $X$, where the model has the form
\begin{equation}
  \label{eqn:glm}
  p_\theta(y \mid x)
  = \exp\left(T(x, y)^T \theta - A(\theta \mid x)\right),
\end{equation}
where $A(\theta \mid x) = \int e^{T(x, y)^T \theta} d\basemeasure(y)$ for
some base measure $\basemeasure$ and $T : \mc{X} \times \mc{Y} \to \R^d$ is
the sufficient statistic.
We assume the distribution $\Px$ on $X$ is known. This assumption is
strong, but may (approximately) hold in practice; in biological
applications, for example, we may have covariate data
and wish to estimate the conditional distribution of $Y \mid X$ for a new
outcome $Y$~\cite[e.g.][]{CandesFaJaLv18}.
For a distribution $P$ on the pair $(X, Y)$,
let $\Px$ denote the marginal over $X$, which we assume is fixed and known,
$\Pyx$ be the conditional distribution over $Y$ given $X$,
and $P = \Pyx \Px$ for shorthand.
Define the population risk using the log loss
$\mlloss_\theta(y \mid x)
= -\log p_\theta(y \mid x)$, by
\begin{equation*}
  \risk_P(\theta) = \E_P[\mlloss_\theta(Y \mid X)]
  = \E_P[-T(X, Y)]^T \theta
  + \E_{\Px}[A(\theta \mid X)]
  = -\E_P[T(X, Y)]^T \theta + A_{\Px}(\theta),
\end{equation*}
where we use the shorthand $A_{\Px}(\theta) \defeq \E_{\Px}[A(\theta \mid
  X)]$. Let $\PfamY$ be a collection of conditional distributions
of $Y \mid X$, and for $\Pyx \in \PfamY$, we analogize the
general exponential family case to define
\begin{equation*}
  \theta(\Pyx)
  \defeq \argmin_\theta \risk_{\Pyx \Px}(\theta)
  = \nabla A_{\Px}^*(\E_{\Pyx \Px}[T(X, Y)]).
\end{equation*}
Considering again the loss
$\loss(\what{\functional} - \functional(\theta(\Pyx))$
for a smooth functional $\functional$,
Corollary~\ref{corollary:mis-specified-expfam} implies
\begin{corollary}
  \label{corollary:glm-lower-bound}
  Let $\PfamY$ be a collection of conditional distributions on $Y \mid X$,
  $P_0 \in \PfamY$, and $\channeldistset_\diffp$ be the collection of $(2,
  \diffp^2)$-R\'{e}nyi-private channels (Def.~\ref{definition:local-renyi}).
  Then for numerical constants $c_0, c_1 > 0$ there
  exists $N = N(\PfamY, \functional)$ independent of the loss $\loss$
  such that $n \ge N$ implies
  \begin{equation*}
    \localminimax_n(P_0, \loss, \PfamY, \channeldistset_\diffp)
    \ge c_0
    \sup_{\Pyx \in \PfamY}
    \loss\left(c_1 \frac{\nabla \functional(\theta_0)^T
      \nabla^2 A_{\Px}(\theta_0)^{-1}
      (\E_{P_0 \Px}[T(X, Y)] - \E_{\Pyx \Px}[T(X, Y)])}{
      \sqrt{n \diffp^2}}\right).
  \end{equation*}
\end{corollary}
\noindent
If the set $\PfamY$ and distribution $\Px$ are such that
$\{\E_{\Pyx \Px}[T] \mid \Pyx \in \PfamY\} \supset \{t \in \R^d
\mid \norm{t} \le r\}$,
then we have the simplified lower bound
\begin{equation*}
  \localminimax_n(P_0, L, \PfamY, \channeldistset_\diffp)
  \ge c_0
  \loss\Big(c_1 \frac{r \dnorm{\nabla^2 A_{\Px}(\theta_0)^{-1} \nabla
      \functional(\theta_0)}}{
    \sqrt{n \diffp^2}}\Big).
\end{equation*}

An optimal estimator parallels
Section~\ref{sec:one-step-mis-expfam}.  Split
a non-private sample
$\{(X_i, Y_i)\}_{i=1}^n$ into samples of size $n_1 =
\ceil{n^{2/3}}$ and $n_2 = n - n_1$.
For $i = 1, \ldots, n_1$,
let $Z_i$ be any $\diffp$-locally differentially private
estimate of $T(X_i, Y_i)$ with $\E[Z_i \mid X_i, Y_i] = T(X_i, Y_i)$ and
$\E[\norm{Z_i}^2] < \infty$, and define
$\wt{\mu}_n = \wb{Z}_{n_1} = \frac{1}{n_1} \sum_{i = 1}^{n_1} Z_i$
and
$\wt{\theta}_n = \nabla A^*_{\Px}(\wt{\mu}_n)
= \argmin_\theta \{-\wt{\mu}_n^T \theta + A_{\Px}(\theta)\}$.
Then, for $i = n_1 + 1, \ldots, n$, let
\begin{equation*}
  Z_i = \nabla \functional(\wt{\theta}_n)^T
  \nabla^2 A_{\Px}(\wt{\theta}_n)^{-1} T(X_i, Y_i)
  + \frac{r \dnorms{\nabla^2 A_{\Px}(\wt{\theta}_n)^{-1} \nabla
  \functional(\wt{\theta}_n)}}{\diffp}
  W_i
  ~~ \mbox{where} ~~
  W_i \simiid \laplace(1),
\end{equation*}
The $Z_i$ are evidently $\diffp$-differentially
private, and
we then define the private estimator
\begin{equation}
  \label{eqn:private-glm-estimator}
  \what{\functional}_n
  \defeq \wb{Z}_{n_2} + \nabla\functional(\wt{\theta}_n)^T \left(\wt{\theta}_n
  - \nabla^2 A_{\Px}(\wt{\theta}_n)^{-1} \wt{\mu}_n\right).
\end{equation}
An identical analysis to that we use to prove
Proposition~\ref{proposition:private-expfam-mis-estimator} then gives the
following corollary, which shows a locally uniform optimal rate of
convergence.  (We use the shorthand $\norm{x}_C^2 = x^T C x$.)
\begin{corollary}
  \label{corollary:local-glm-asymptotics}
  Let $\what{\functional}_n$ be the
  estimator~\eqref{eqn:private-glm-estimator} and
  $\theta_0 = \nabla A_{\Px}^*(\E_P[T(X, Y)])
  = \argmin_\theta \risk_P(\theta)$. Then
  \begin{equation*}
    \sqrt{n} (\what{\functional}_n - \functional(\theta_0))
    \cd \normal\left(0,
    \norm{\nabla^2 A_{\Px}(\theta_0)^{-1} v}_{\cov(T(X, Y))}^2
    + \frac{2}{\diffp^2}
    \dnorm{\nabla^2 A_{\Px}(\theta_0)^{-1} v}^2\right),
  \end{equation*}
  and convergence is locally uniform over QMD submodels.
\end{corollary}

\newcommand{\sub}{{\rm sub}}
\newcommand{\full}{{\rm full}}
\newcommand{\Gset}{{\mathcal G}}

\subsection{Model adaptation in locally private exponential family estimation}
\label{sec:adaptation-of-models}

\newcommand{\Pnonpara}{\mc{P}_{\nonpara}}
\newcommand{\Ppara}{\mc{P}_{\para}}
\newcommand{\losstrunc}{L_{\wedge B}}

We conclude this section by highlighting a phenomenon that distinguishes
locally private estimation from non-private estimation, focusing especially
on exponential families as in Section~\ref{sec:mis-specified-expfam}.
We recall~\citet{Stein56a}, who roughly asks the following:
given a parameter $\theta$ of interest and a (potentially) infinite
dimensional nuisance $G$, can we estimate $\theta$ asymptotically as well
regardless of whether we know $G$?  Here, we consider this in the context of
$G$ being the full distribution $P_0$, and we delineate cases---which depend
on the channel set $\channeldistset$ being either the identity (non-private)
or a private collection---when for a sub-family $\Pmodel \subset \mc{P}$
containing $P_0$, we have
\begin{equation}
  \label{eqn:problem-of-adaptation}
  \localminimax_n(P_0, \loss, \mc{P}, \channeldistset)
  \asymp
  \localminimax_n (P_0, \loss, \Pmodel, \channeldistset). 
\end{equation}
For exponential families,
the non-private local minimax risk is (up to constants)
independent of whether the containing family $\mc{P}$ of distributions is
parametric or non-parametric, while the private local minimax risk is larger
in the non-parametric than parametric settings, necessitating the
construction of distinct private estimators with different optimality
properties that depend on the overall model the statistician is willing to
assume.

For simplicity we study one-dimensional potentially misspecified models with
densities $p_\theta(x) = \exp(\theta x - A(\theta))$ and base measure $\mu$.
We consider nonparametric and parametric families, making an assumption (for
convenience) that the first has uniformly bounded (arbitrary) fourth moment:
\begin{equation*}
  \Pnonpara \defeq \left\{P : \E_P[X] \in \range(\nabla A),
  \E_P[|X|^4] \le M < \infty\right\}
  ~~ \mbox{and} ~~
  \Ppara \defeq \{P_\theta\}_{\theta \in \Theta}.
\end{equation*}
To avoid issues of infinite loss, we use the truncated squared error
$\losstrunc(\theta - \theta(P))
= (\theta - \theta(P))^2 \wedge B$, where $0 < B <
\infty$ is otherwise arbitrary.

To compare the private and non-private cases, we evaluate their local
minimax risks. In the non-private case, the model class is immaterial, as
the efficient influence and score functions for exponential families are
identical in both parametric and nonparametric
cases~\cite[Ch.~25.3]{VanDerVaart98}, so we have the
equivalence~\eqref{eqn:problem-of-adaptation} when $\channeldistset =
\{\textup{id}\}$.  We prove the following characterization in
Section~\ref{sec:proof-non-private-mis-specified}.
\begin{claim}
  \label{claim:non-private-mis-specified}
  Let $P_0 = P_{\theta_0}$ belong to the exponential family above.
  Then for large enough $n$,
  \begin{equation*}
    \localminimax_n(P_0, \losstrunc, \Pnonpara, \{\textup{id}\})
    \asymp \localminimax_n(P_0, \losstrunc, \Ppara, \{\textup{id}\})
    \asymp \frac{1}{n \var_0(X)}.
  \end{equation*}
\end{claim}
\noindent
The
risks, by comparison, have different behavior, as
the discussion below shows.
\begin{claim}
  \label{claim:private-mis-specified}
  Let $P_0 = P_{\theta_0}$ belong to the exponential family above.
  Let $\channeldistset_\diffp$ be the collection
  of $(2, \diffp^2)$-R\'{e}nyi differentially private
  channels and $\Pallmodel \subset \Pnonpara$ be a collection of
  sub-models with scores $g$ dense in $\bigLone(P_0)$ at $P_0$.
  Then there exist numerical constants
  $0 < c_0 \le c_1 < \infty$ such that for large enough $n$,
  \begin{subequations}
    \begin{equation}
      \label{eqn:private-nonparametric-specified}
      \localminimax_n(P_0, \losstrunc, \Pallmodel, \channeldistset_\diffp)
      \ge c_0 \cdot\frac{1}{n \diffp^2}
      \esssup_x \frac{(\E_{P_0}[X] - x)^2}{\var_0(X)^2}
    \end{equation}
    and
    \begin{equation}
      \label{eqn:private-parametric-specified}
      \localminimax_n(P_0, \losstrunc, \Ppara, \channeldistset_\diffp)
      \in [c_0, c_1] \cdot \frac{1}{n \diffp^2}
      \cdot \frac{1}{\E_0[|X - \E_0[X]|]^2}.
    \end{equation}
  \end{subequations}
  Additionally,
  $1 / \E_0[|X - \E_0[X]|]^2 \le \esssup_x (\E_0[X] - x)^2 / \var_0(X)^2$.
\end{claim}

An alternative way to view (and prove) the right-hand (variance) quantities
in the claims is via influence and score functions. The efficient influence
function in exponential families~\cite[Ch.~25.3]{VanDerVaart98} is
$\influencefunc(x) = (x - \E_0[X]) / \var_0(X) = (x - A'(\theta_0)) /
A''(\theta_0)$, with the second equality following because $P_0 =
P_{\theta_0}$ by assumption.  The asymptotic
variance~\cite[Ex.~25.16]{VanDerVaart98} in the non-private case becomes
\begin{equation*}
  \sup_{g \in \bigLtwo(P_0)}
  \frac{\E_0[\influencefunc(X) g(X)]^2}{\E_0[g(X)^2]}
  = \normbig{\influencefunc(X)}^2_{\bigLtwo(P_0)}
  = \frac{1}{\var_0(X)},
\end{equation*}
attaining the supremum at the parametric
score $\score_{\theta_0}(x) = x - \E_0[X]$.
In the private case,
we have
\begin{equation*}
  \sup_{g}
  \frac{\E_0[\influencefunc(X) g(X)]^2}{\E_0[|g(X)|]^2}
  = \esssup_x \frac{(\E_0[X] - x)^2}{\var_0(X)^2}
  ~~~ \mbox{while} ~~~ 
  \frac{\E_0[\influencefunc(X) \score_{\theta_0}(X)]^2}{
    \E_0[|\score_\theta(X)|]^2}
  = \frac{1}{\E_0[|X - \E_0[X]|]^2}
\end{equation*}
as in inequalities~\eqref{eqn:private-nonparametric-specified}
and~\eqref{eqn:private-parametric-specified}, respectively.  (Applying
Corollary~\ref{corollary:mis-specified-expfam} thus demonstrates the lower
bound~\eqref{eqn:private-nonparametric-specified}, and the preceding display
also gives the final result in Claim~\ref{claim:private-mis-specified}.  For
the bounds~\eqref{eqn:private-parametric-specified}, use
Proposition~\ref{proposition:information-private-lb} and
Corollary~\ref{corollary:l1-info-attained} with score $\score_\theta(x) = x
- A'(\theta)$.)  This contrast shows how the worst score $g$ in the
nonparametric case depends strongly on whether we have privacy or not; in
the latter, it is simply $g = \score_{\theta_0}$, while in the former, the
structure is quite different.






\section{Experiments on a flow cytometry dataset}
\label{sec:experiments}

We perform experiments investigating the behavior of
our proposed locally optimal estimators, comparing their performance both to
non-private estimators and to minimax optimal estimators developed by
\citet{DuchiJoWa18} for locally private estimation.
We consider the
generalized linear model~\eqref{eqn:glm} and
estimating the linear
functional $\functional(\theta) = v^T \theta$.
As motivation, consider the
problem of testing whether a covariate $X_j$ is relevant to a binary outcome
$Y \in \{-1, 1\}$. In this case, the logistic GLM model~\eqref{eqn:glm}
is $p_\theta(y \mid x) = \exp(yx^T \theta) / (1 + \exp(y x^T
\theta))$, and using the standard basis vectors $v = e_j$,
estimating $v^T
\theta$ corresponds to testing $\theta_j \lessgtr 0$ while
controlling for the other covariates.

\newcommand{\thetamle}[1]{\theta_{\rm ml}^{({#1})}}
\newcommand{\thetasgm}[1]{\what{\theta}_{\rm sg}^{({#1})}}
\newcommand{\thetainit}[1]{\what{\theta}_{\rm init}^{({#1})}}
\newcommand{\thetaonestep}[1]{\what{\theta}_{\rm os}^{(#1)}}

We investigate the performance of the locally private one-step
estimator~\eqref{eqn:private-glm-estimator} on a flow-cytometry dataset for
predicting protein expression~\cite[Ch.~17]{HastieTiFr09}, comparing against
(global) minimax optimal stochastic gradient estimators~\cite{DuchiJoWa18}.
The flow-cytometry dataset contains expression level measurements of
$d = 11$ proteins on $n = 7466$ cells, and the goal is to understand the
network structure linking the proteins: how does protein $j$'s expression level
depend on the remaining proteins. As the raw data is heavy-tailed
and skewed, we perform an inverse tangent transformation $x_{ij} \mapsto
\tan^{-1} (x_{ij})$.  Letting $X \in \R^{n \times d}$ be the data
matrix, to compare the methods and to guarantee a ground truth in
our experiments, we treat $X$ as the \emph{full population}, so each
experiment consists of sampling rows of $X$ with replacement.

Let $x \in \R^d$ denote a row of $X$. For $i \in [d]$, we
wish to predict  $y = \sign(x_i)$
based on $x_{-i} \in \R^{d-1}$, the remaining
covariates, and we use the logistic regression model
\begin{equation*}
  \log \frac{P_\theta(\sign(x_i) = 1 \mid x_{-i})}{
    P_\theta(\sign(x_i) = -1 \mid x_{-i})}
  = \theta^T x_{-i} + \theta_{\rm bias},
\end{equation*}
so that $T(x_{-i}, y) = y [x_{-i}^T ~ 1]^T$ and $A(\theta \mid x_{-i}) =
\log(e^{\theta^T x_{-i} + \theta_{\rm bias}} + e^{-\theta^T x_{-i} -
  \theta_{\rm bias}})$, where $y = \sign(x_i)$ is the sign of the expression
level of protein $i$.
We let
$\thetamle{i} \in \R^d$ be the parameter (including the bias) maximizing
the likelihood for this logistic model of predicting $x_i$
using the full data $X$.

We perform multiple experiments, where each is as follows.  We sample $N$
rows of $X$ uniformly (with replacement) and
vary the privacy parameter in
$\diffp \in \{1, 4\}$. We perform perform two private
procedures (and one non-private procedure) on the resampled data
$X_{\text{new}} \in \R^{N \times d}$:
\begin{enumerate}[(i)]
\item The non-private maximum likelihood estimator (MLE) on the
  resampled data of size $N$.
\item The minimax optimal stochastic gradient
  procedure of \citet[Secs.~4.2.3 \& 5.2]{DuchiJoWa18}. In brief, this
  procedure begins from $\theta^0 = 0$, and at iteration $k$ draws a pair
  $(x,y)$ uniformly at random, then uses a carefully designed
  $\diffp$-locally private version $Z^k$ of $T = T(x,y)$ with the property
  that $\E[Z \mid x,y] = T(x,y)$ and $\sup_k \E[\norms{Z^k}^2] < \infty$,
  updating
  \begin{equation*}
    \theta^{k + 1} = \theta^k - \eta_k \left(\nabla A_{\Px}(\theta^k) - Z^k\right),
  \end{equation*}
  where $\eta_k > 0$ is a stepsize sequence. (We use optimal the $\ell_\infty$
  sampling mechanism~\cite[Sec.~4.2.3]{DuchiJoWa18} to construct $Z_i$.)
  We use stepsizes $\eta_k = 1 / (20 \sqrt{k})$, which gave optimal
  performance over many choices of stepsize and power $k^{-\beta}$.
  We perform $N$ steps of this stochastic gradient method,
  yielding estimator $\thetasgm{i}$ for prediction of protein $i$ from
  the others.
\item The one-step corrected
  estimator~\eqref{eqn:private-glm-estimator}.  To construct the initial
  $\wt{\theta}_n$, we use \citeauthor{DuchiJoWa18}'s $\ell_\infty$
  sampling mechanism to construct the approximation $\wt{\mu}_{n} =
  \frac{1}{n_1} \sum_{i=1}^n Z_i$ and let $\thetainit{i} = \wt{\theta}_n
  = \nabla A^*(\Px)(\wt{\mu}_n)$. For
  coordinates $i = 1, \ldots, d$, we set
  $\functional(\theta) = v^T \theta$ for $v = e_1, \ldots, e_d$
  as in~\eqref{eqn:private-glm-estimator}.
\end{enumerate}
\noindent
We perform each of these three-part tests $T = 100$ times,
where within each test, each method uses an identical sample
(the samples are of course independent across tests).

We summarize our results in Figure~\ref{fig:violin-results} and
Table~\ref{table:who-wins}. Figure~\ref{fig:violin-results}
plots the errors across all coordinates of $\thetamle{i}$, $i = 1,
\ldots, d$, and all $T = 100$ tests of the three procedures,
with top whisker at the 99th
percentile error for each.
We vary sample sizes $N \in \{ 2n, 8n, 40n\}$ and
privacy level $\diffp \in \{1, 4\}$; results remain consistent for other
sample sizes. As
the sample size (or $\diffp$) grows, the one-step estimator converges
more quickly than the minimax stochastic gradient procedure,
though for the smaller sample size the private SGD method exhibits
better performance.

\begin{figure}[t]
  \begin{center}
    \begin{tabular}{cc}
      \includegraphics[width=.48\columnwidth]{%
        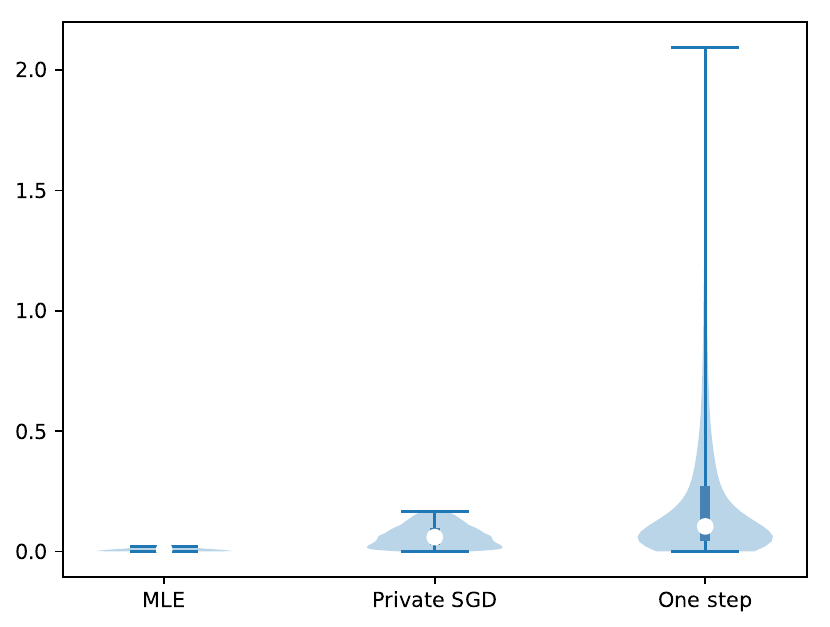}
      &
      \includegraphics[width=.48\columnwidth]{%
        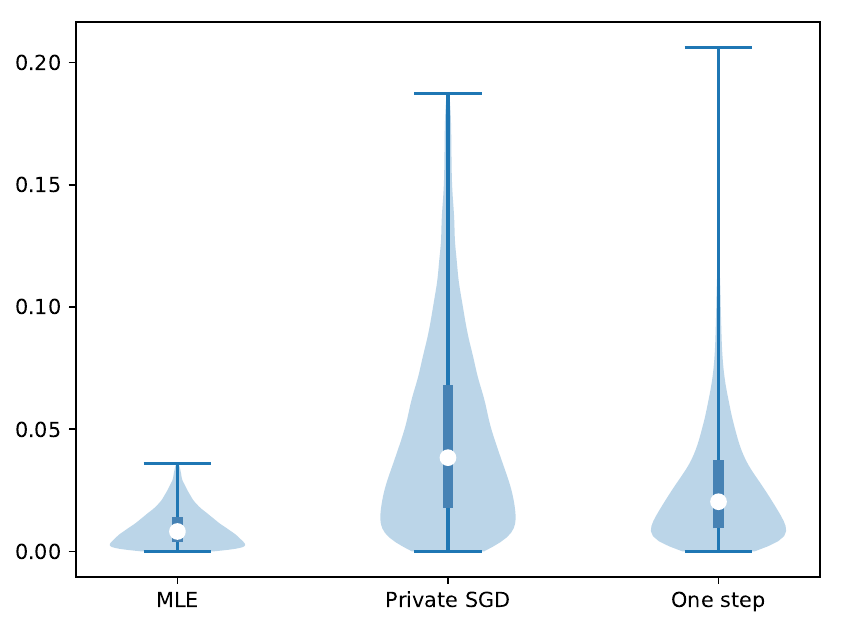} \\
      (a) $N = 2n$, $\diffp = 1$ & (b) $N = 2n$, $\diffp = 4$ \\
      \includegraphics[width=.48\columnwidth]{%
        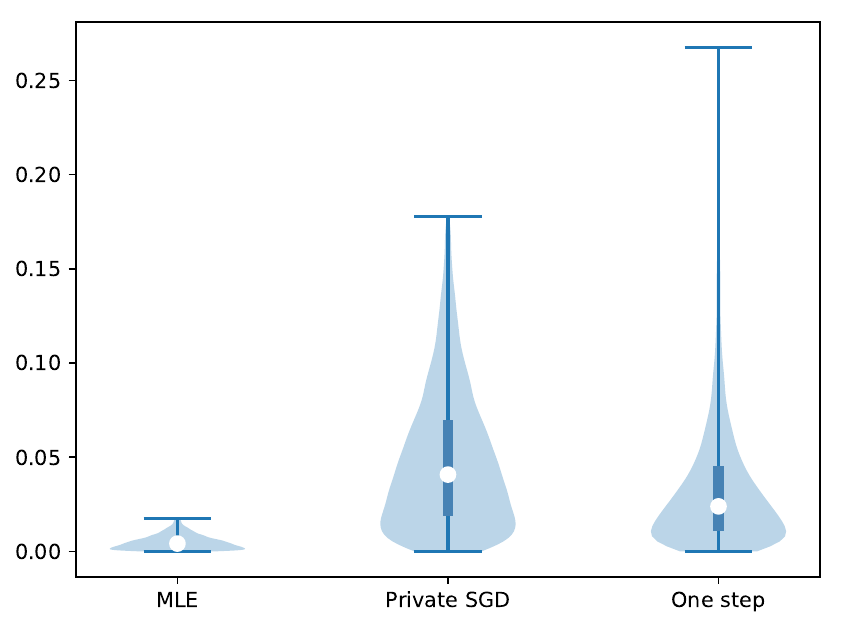}
      &
      \includegraphics[width=.48\columnwidth]{%
        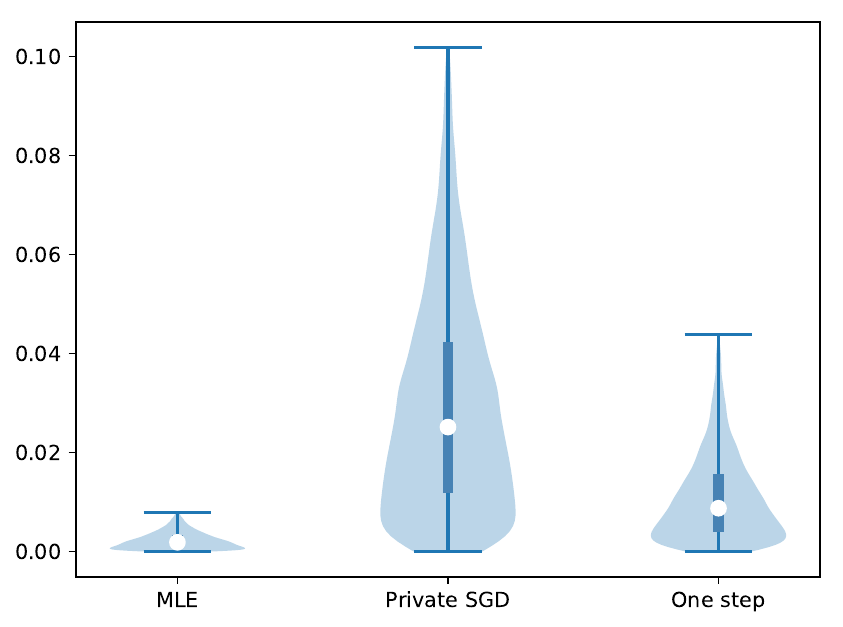} \\
      (c) $N = 8n$, $\diffp = 1$ & (d)
      $N = 40n$, $\diffp = 1$
    \end{tabular}
    \caption{\label{fig:violin-results} Errors
      $|\what{\functional}_N - v^T \thetamle{i}|$
      across all experiments, for $v = e_1, \ldots, e_d$
      and $i = 1, \ldots, d$, in the
      logistic regression model, with medians and interquartile ranges marked.}
  \end{center}
  \vspace{-.5cm}
\end{figure}


In Table~\ref{table:who-wins}, we compare the estimators $\thetainit{i}$,
$\thetasgm{i}$, and $\thetaonestep{i}$ of the true parameter $\thetamle{i}$
more directly. For each, we count the number of experiments (of $T$) and
indices $j = 1, \ldots, d$ for which
\begin{equation*}
  \left|[\thetaonestep{i}]_j - [\thetamle{i}]_j\right|
  < \left|[\thetainit{i}]_j - [\thetamle{i}]_j\right|
  ~~ \mbox{and} ~~
  \left|[\thetaonestep{i}]_j - [\thetamle{i}]_j]\right|
    < \left|[\thetasgm{i}]_j -
    [\thetamle{i}]_j\right|,
\end{equation*}
that is, the number of experiments in which the one-step estimator provides
a better estimate than its initializer or the minimax stochastic
gradient-based procedure.  Table~\ref{table:who-wins} shows these results,
displaying the proportion of experiments in which the one-step method has
higher accuracy than the other procedures.  For large sample
sizes, the asymptotic optimality of the one-step appears to
be salient, as its performance relative to the other methods improves.
Based on
additional simulations, it appears that the initializer
$\thetainit{i}$ is inaccurate for small sample sizes,
so the one-step correction has poor
Hessian estimate and performs poorly.
The full minimax procedure~\cite{DuchiJoWa18} adds more noise than
is necessary, as it privatizes the entire statistic $xy$ in each
iteration---a necessity because it iteratively builds the estimates
$\thetasgm{\cdot}$---causing an increase in sample complexity.

\begin{table}[t]
  \begin{center}
    \begin{tabular}{|c||c|c|c|c|c|c|}
      \hline
      Sample size & \multicolumn{2}{|c|}{$N = 2n$} &
      \multicolumn{2}{|c|}{$N = 8n$} &
      \multicolumn{2}{|c|}{$N = 40n$} \\
      \hline
      Privacy $\diffp$ & $\diffp = 1$ & $\diffp = 4$ & $\diffp = 1$
      & $\diffp = 4$ & $\diffp = 1$ & $\diffp = 4$  \\
      \hline
      vs.\ initializer & 0.501 & 0.82 & 0.791 & 0.848 & 0.825 & 0.852 \\
      \hline
      vs.\ minimax (stochastic gradient) & 0.321 & 0.677 & 0.659 & 0.79
      & 0.777 & 0.817 \\
      \hline
    \end{tabular}
    \caption{\label{table:who-wins}
      Frequency with which the one-step estimator outperforms
      initialization and minimax (stochastic-gradient-based) estimator
      over $T = 100$ tests, all coordinates $j$ of the parameter
      and proteins $i = 1, \ldots, d$ for the flow-cytometry data.}
  \end{center}
\end{table}

\newcommand{\centraldiffp}{\diffp_{\textup{cen}}}

The one-step correction typically outperforms alternative approaches in
large-sample regimes, and such large samples may be more effectively
achievable than is \emph{prima facie} obvious, as locally private procedures
can guarantee strong central differential privacy.
\citet{ErlingssonFeMiRaTaTh19} consider privacy amplification in the
\emph{shuffle model}, where the data $\{X_i\}$ are permuted
before the sampling $Z_i \sim \channel(\cdot \mid X_i, Z_{1:i-1})$;
other variants~\cite{BalleBeGaNi19}
randomize and then permute the $Z_i$ into $Z_{\pi(1:n)} \in
\mc{Z}^n$.  The permuted vector $Z_{\pi(1:n)}$ then
achieves $(\centraldiffp, \delta)$-differential privacy~\cite[Corollary
  5.3.1]{BalleBeGaNi19} for
\begin{equation*}
  \centraldiffp = O(1) e^\diffp \sqrt{\frac{\min\{1, \diffp^2\} \log
      \frac{1}{\delta}}{n}}.
\end{equation*}
Applying this randomize-then-shuffle approach $K \ll n$ distinct times,
whenever $\diffp = O(1)$, composition bounds for differential
privacy~\cite[Ch.~3.5.2]{DworkRo14} guarantee $(\centraldiffp, \delta)$-central
differential privacy for $\centraldiffp = O(1) \sqrt{\frac{K \diffp^2 \log
    \delta^{-1}}{n}}$.  The one-step
estimator~\eqref{eqn:private-glm-estimator} falls in this framework, and
(via a calculation) achieves $\centraldiffp \le 1$ for $N = 40n$, $\diffp =
1$.  Consequently, this type of behavior may be acceptable in natural
local privacy applications: situations (such as web-scale
data) with large sample sizes or where resampling is possible, as we may
achieve both strong privacy and reasonable performance.

\section{Proofs of main results}

\label{sec:proofs}

We collect the proofs of our main results in this section, as they are
reasonably brief and (we hope) elucidating.  The main technical tool
underpinning our lower bounds is that our definitions of privacy imply
strong contractions on the space of probability measures.  Such contractive
properties have been important in the study of information channels and
\emph{strong data processing}~\cite{CohenKeZb98,DelMoralLeMi03}
and in the mixing properites of Markov chains under
so-called \emph{strong mixing conditions}, such as the Dobrushin
condition~\cite{Dobrushin56}. Consequently, before turning to the main
proofs, we first present a few results on contractions of probability
measures, as they underly our subsequent development.

\subsection{Contractions of probability measures}
\label{sec:contraction-probability-measures}

We provide our contractions using $f$-divergences.  For a convex function $f
: \R_+ \to \R \cup \{+\infty\}$ with $f(1) = 0$, the $f$-divergence between
distributions $P$ and $Q$ is
\begin{equation*}
  \fdiv{P}{Q} \defeq \int f\left(\frac{dP}{dQ}\right) dQ,
\end{equation*}
which is non-negative and strictly positive when $P \neq Q$ and $f$ is
strictly convex at the point $1$. We typically consider $f$-divergences
parameterized by $k \in \openright{1}{\infty}$ of the form
\begin{equation*}
  f_k(t) \defeq |t - 1|^k.
\end{equation*}

Given a channel $\channel$, for $a \in \{0, 1\}$, define the marginal
distributions
\begin{equation*}
  \marginprob_a(S) \defeq \int \channel(S \mid x)
  dP_a(x).
\end{equation*}
The goal is then to provide upper bounds on the $f$-divergence
$\fdiv{\marginprob_0}{\marginprob_1}$ in terms of the channel $\channel$;
the standard data-processing inequality~\cite{CoverTh06,LieseVa06}
guarantees $\fdiv{\marginprob_0}{\marginprob_1} \le
\fdiv{P_0}{P_1}$.  Dobrushin's celebrated
ergodic coefficient $\alpha(\channel) \defeq 1 - \sup_{x, x'}
\tvnorm{\channel(\cdot \mid x) - \channel(\cdot \mid x')}$
guarantees that
for any
$f$-divergence (see~\cite{CohenKeZb98,DelMoralLeMi03}),
\begin{equation}
  \label{eqn:strong-data-processing}
  \fdiv{\marginprob_0}{\marginprob_1}
  \le \sup_{x, x'}
  \tvnorm{\channel(\cdot \mid x) - \channel(\cdot \mid x')}
  \fdiv{P_0}{P_1}.
\end{equation}
Thus, as long as the Dobrushin coefficient is strictly positive, one
obtains a strong data processing inequality.  In our case, our privacy
guarantees provide a stronger condition than the positivity of the Dobrushin
coefficient. Consequently, we are able to provide substantially stronger data
processing inequalities: we can even show that it is possible to modify the
underlying $f$-divergence.

We have the following proposition, which provides a strong data processing
inequality for all channels that are uniformly close under the
polynomial $f$-divergences with $f_k$.

\begin{proposition}
  \label{proposition:fk-contractions}
  Let $f_k(t) = |t - 1|^k$ for some $k > 1$, and let
  $P_0$ and $P_1$ be arbitrary distributions on a common
  space $\mc{X}$. Let $\channel$ be a Markov kernel
  from $\mc{X}$ to $\mc{Z}$
  satisfying
  \begin{equation*}
    \fdivf{f_k}{\channel(\cdot \mid x)}{\channel(\cdot \mid x')}
    \le \diffp^k
  \end{equation*}
  for all $x, x' \in \mc{X}$
  and $\marginprob_a(\cdot) = \int \channel(\cdot \mid x) dP_a(x)$.
  Then
  \begin{equation*}
    \fdivf{f_k}{\marginprob_0}{\marginprob_1}
    \le (2 \diffp)^k
    \tvnorm{P_0 - P_1}^k.
  \end{equation*}
\end{proposition}
\noindent
See Section~\ref{sec:proof-fk-contractions} for a proof.

Jensen's inequality implies that $2^k \tvnorm{P_0 - P_1}^k \le
\fdivf{f_k}{P_0}{P_1}$, so Proposition~\ref{proposition:fk-contractions}
provides a stronger guarantee than the classical
bound~\eqref{eqn:strong-data-processing} for the specific divergence
associated with $f_k(t) = |t - 1|^k$.  Because $\tvnorm{P_0 -
  P_1} \le 1$ for all $P_0, P_1$, it is possible that the $f_k$-divergence
is infinite, while the marginals are much closer together. It is this
transfer from power divergence to variation distance, that is, $f_k$ to
$f_1(t) = |t - 1|$, that allows us to prove the strong localized lower
bounds depending on variation distance such as
Theorem~\ref{theorem:modulus-of-continuity}.

We may
parallel the proof of~\cite[Theorem 1]{DuchiJoWa18} to obtain a
tensorization result. In this context, the most important divergence for us
is the R\'{e}nyi 2-divergence (Def.~\ref{definition:local-renyi}),
which corresponds to the case $k = 2$ (i.e.\ the $\chi^2$-divergence) in
Proposition~\ref{proposition:fk-contractions}, $f(t) = (t - 1)^2$,
and $\dchis{P}{Q} = \exp(D_2(P|\!|Q)) - 1$.
Recall the sequentially interactive
formulation~\eqref{eqn:sequential-interactive} and let
\begin{equation*}
  \channel^n(S \mid x_{1:n})
  \defeq \int_{z_{1:n} \in S}
  \prod_{i = 1}^n d\channel(z_i \mid x_i, z_{1:i-1}).
\end{equation*}
Now, let $P_a, a = 0, 1$ be product distributions on $\mc{X}$, where we say
that the distribution of $X_i$ either follows $P_{0,i}$ or $P_{1,i}$, and
define $\marginprob_a^n(\cdot) = \int \channel^n(\cdot \mid x_{1:n})
dP_a(x_{1:n})$, noting that $dP_a(x_{1:n}) = \prod_{i = 1}^n dP_{a,i}(x_i)$
as $P_a$ is a product distribution. We have the following corollary.

\begin{corollary}
  \label{corollary:tensorized-contraction-chi}
  Let $\channel$ be sequentially interactive and satisfy
  $(2, \diffp^2)$-R\'{e}nyi privacy (Def.~\ref{definition:local-renyi}).
  Then
  \begin{equation*}
    \dchi{\marginprob_0^n}{\marginprob_1^n}
    \le \prod_{i=1}^n \left(1
    + 4 \diffp^2 \tvnorm{P_{0,i} - P_{1,i}}^2\right) - 1.
  \end{equation*}
\end{corollary}
\noindent
See Section~\ref{sec:proof-tensorized-contraction-chi} for a proof.  An
immediate consequence of
Corollary~\ref{corollary:tensorized-contraction-chi} and the
fact~\cite[Lemma 2.7]{Tsybakov09} that $\dkl{P_0}{P_1} \le \log(1 +
\dchi{P_0}{P_1})$ yields
\begin{equation}
  \label{eqn:tensorized-contraction-kl}
  \dkl{\marginprob_0^n}{\marginprob_1^n}
  \le \sum_{i = 1}^n \log\left(1 + 4 \diffp^2
  \tvnorm{P_{0,i} - P_{1,i}}^2 \right)
  \le 4 \diffp^2 \sum_{i = 1}^n
  \tvnorm{P_{0,i} - P_{1,i}}^2.
\end{equation}
The tensorization~\eqref{eqn:tensorized-contraction-kl} is
the key to our results, as we see in the later sections.

\subsubsection{Proof of Proposition~\ref{proposition:fk-contractions}}
\label{sec:proof-fk-contractions}

Let $p_0$ and $p_1$ be the densities of $P_0, P_1$ with respect to some
base measure $\mu$ dominating $P_0, P_1$. Without loss of generality, we
may assume that $\mc{Z}$ is finite, as all $f$-divergences are
approximable by finite partitions~\cite{Vajda72}; we let $\margindens_a$
denote the associated p.m.f.  For $k > 1$, the function $t \mapsto t^{1 -
  k}$ is convex on $\R_+$. Thus, applying Jensen's inequality, we may
bound $\fdivf{f_k}{\marginprob_0}{\marginprob_1}$ by
\begin{align}
  \nonumber
  \fdivf{f_k}{\marginprob_0}{\marginprob_1}
  = \sum_z \frac{|\margindens_0(z) - \margindens_1(z)|^k}{
    \margindens_1(z)^{k - 1}}
  & \leq
  \sum_z \int \frac{|\margindens_0(z) - \margindens_1(z)|^k}{
    \channeldens(z \mid x_0)^{k-1}} p_1(x_0) d\mu(x_0) \\
  & =  \int \underbrace{\left(\sum_z
    \frac{|\margindens_0(z) - \margindens_1(z)|^k}{
      \channeldens(z \mid x_0)^{k-1}} \right)}_{
    \eqdef W(x_0)} p_1(x_0) d\mu(x_0).
  \label{eqn:prop-k-moment-starting-point}
\end{align}
It thus suffices to upper bound $W(x_0)$.
To do so, we rewrite $\margindens_0(z) - \margindens_1(z)$ as
\begin{equation*}
  \margindens_{0}(z) - \margindens_{1}(z)
  = \int \channeldens(z \mid x) (dP_0(x) - dP_1(x))
  = \int \left(\channeldens(z \mid x)
  - \channeldens(z \mid x_0)\right) (dP_0(x) - dP_1(x)),
\end{equation*}
where we have used that $\int (dP_0 - dP_1) = 0$.
Now define the function
\begin{equation*}
  \Delta(z \mid x, x_0)
  \defeq \frac{\channeldens(z\mid x) - 
    \channeldens(z\mid x_0)}{\channeldens(z \mid x_0)^{1-1/k}}.
\end{equation*}
By Minkowski's integral inequality, we have the upper bound
\begin{align}
  \label{eqn:key-step-Minkovski}
  \lefteqn{W(x_0)^{1/k}
    = \left(\sum_z \left|\int
    \Delta(z \mid x, x_0) (p_0(x) - p_1(x))d\mu(x)
    \right|^k\right)^{1/k}} \\
  & \leq \int  \left(\sum_z \big|\Delta(z \mid x, x_0)
  (p_0(x) - p_1(x))\big|^k \right)^{1/k}
  d\mu(x)
  = \int \left(\sum_z |\Delta(z \mid x, x_0)|^k
  \right)^\frac{1}{k} |dP_0(x) - dP_1(x)|. \nonumber
\end{align}
Now we compute the inner summation: we have
that
\begin{equation*}
  \sum_z |\Delta(z \mid x, x_0)|^k
  =
  \sum_z \left|\frac{\channeldens(z \mid x)}{\channeldens(z \mid x_0)}
  - 1 \right|^k \channeldens(z \mid x_0)
  = \fdivf{f_k}{\channel(\cdot \mid x)}{\channel(\cdot \mid x_0)}.
\end{equation*}
Substituting this into our upper bound~\eqref{eqn:key-step-Minkovski}
on $W(x_0)$, we obtain that
\begin{equation*}
  W(x_0)
  \le \sup_{x \in \mc{X}}
  \fdivf{f_k}{\channel(\cdot \mid x)}{\channel(\cdot \mid x_0)}
  2^k \tvnorm{P_0 - P_1}^k,
\end{equation*}
as $\int|dP_0 - dP_1| = 2 \tvnorm{P_0 - P_1}$.
Substitute this upper bound into
inequality~\eqref{eqn:prop-k-moment-starting-point}.

\subsubsection{Proof of Corollary~\ref{corollary:tensorized-contraction-chi}}
\label{sec:proof-tensorized-contraction-chi}

We use an inductive argument. The base case in which $n = 1$ follows
immediately by Proposition~\ref{proposition:fk-contractions}.  Now, suppose that
Corollary~\ref{corollary:tensorized-contraction-chi} holds at $n - 1$; we
will show that the claim holds for $n \in \N$. We use the shorthand
$\margindens_a(z_{1:k})$ for the density of the measure $\marginprob_a^k$,
$a \in \{0, 1\}$ and $k \in \N$, which we may assume exists w.l.o.g. Then, by
definition of the $\chi^2$-divergence, we have
\begin{equation*}
  \dchi{\marginprob_0^n}{\marginprob_1^n} + 1 = 
  \E_{\marginprob_1}\left[
    \frac{\margindens_0^2(Z_{1:n})}{\margindens_1^2(Z_{1:n})} \right] 
  = \E_{\marginprob_1} \left[
    \frac{\margindens_0^2(Z_{1:{n-1}})}{\margindens_1^2(Z_{1:{n-1}})}
    \E_{\marginprob_1}
    \left[
      \frac{\margindens_0^2(Z_n \mid Z_{1:n-1})}{
        \margindens_1^2(Z_n \mid Z_{1:n-1})}\mid Z_{1:{n-1}}\right]  \right].
\end{equation*}
Noting that the $k$th marginal distributions $\marginprob_{a,k}(\cdot \mid z_{1:k-1})
= \int \channel(\cdot \mid x, z_{1:k-1}) dP_{a,i}(x)$ for $a \in \{0, 1\}$,
we see that for any $z_{1:n-1} \in \mc{Z}^{n-1}$,
\begin{align*}
  \E_{\marginprob_1}
  \left[\frac{\margindens_0^2(Z_n \mid z_{1:n-1})}{
      \margindens_1^2(Z_n \mid z_{1:n-1})}
    \mid z_{1:{n-1}}\right] 
  & = 1 + \dchi{\marginprob_{0,n}(\cdot \mid z_{1:n-1})}{\marginprob_{1,n}(\cdot \mid z_{1:n-1})}
  \\
  & \le
  1 + 4 \diffp^2 \tvnorm{P_{0,n}(\cdot \mid z_{1:n-1})
    - P_{1,n}(\cdot \mid z_{1:n-1})}^2 \\
  & = 1 + 4 \diffp^2 \tvnorm{P_{0,n} - P_{1,n}}^2,
\end{align*}
where the inequality is Proposition~\ref{proposition:fk-contractions}
and the final equality follows because $X_n$ is independent of $Z_{1:n-1}$.
This yields the inductive step and completes the proof
once we recall the inductive hypothesis and
that $\E_{\marginprob_1}[\frac{\margindens_0^2(Z_{1:n-1})}{
    \margindens_1^2(Z_{1:n-1})}]
= \dchis{\marginprob_0^{n-1}}{\marginprob_1^{n-1}} + 1$.

\subsection{Proof of Theorem~\ref{theorem:modulus-of-continuity}}
\label{sec:proof-modulus-of-continuity}

We follow the typical reduction of estimation to testing, common in the
literature on lower bounds~\cite{AgarwalBaRaWa12,
  DuchiJoWa18,Tsybakov09,Yu97}.
For shorthand, let $\theta_v = \theta(P_v)$ for $v = 0, 1$ throughout the
proof.
Define the ``distance''
\begin{equation*}
  \lossdist(P_0, P_1)
  \defeq \inf_\theta \left\{\loss(\theta - \theta(P_0))
  + \loss(\theta - \theta(P_1)) \right\},
\end{equation*}
which satisfies $\lossdist(P_0, P_1) = 2 \loss(\frac{\theta_0 - \theta_1}{2})$
when $\loss$ is convex and (by quasi-convexity and symmetry) satisfies
$\lossdist(P_0, P_1) \ge \loss(\frac{\theta_0 - \theta_1}{2})$.
By definition of $\lossdist$, we have
the mutual exclusion that for any $\theta$,
\begin{equation}
  \label{eqn:exclusion}
  \loss(\theta - \theta_0) < \half \lossdist(P_0, P_1)
  ~~ \mbox{implies} ~~
  \loss(\theta - \theta_1) \ge \half \lossdist(P_0, P_1).
\end{equation}
Let $\marginprob_0^n$ and $\marginprob_1^n$ be the marginal probabilities
over observations $Z_{1:n}$ under $P_0$ and $P_1$ for a
channel $\channel \in \channeldistset$.  Using Markov's inequality,
we have for any estimator $\what{\theta}$
based on $Z_{1:n}$ and any $\delta \ge 0$ that
\begin{align*}
  \E_{\marginprob_0^n}\left[L(\what{\theta} - \theta_0)\right]
  +
  \E_{\marginprob_1^n}\left[L(\what{\theta} - \theta_1)\right]
  & \ge \delta \left[
    \marginprob_0^n(L(\what{\theta} - \theta_0) \ge \delta)
    + \marginprob_1^n(L(\what{\theta} - \theta_1) \ge \delta)
    \right] \\
  & = \delta \left[1 -
    \marginprob_0^n(L(\what{\theta} - \theta_0) < \delta)
    + \marginprob_1^n(L(\what{\theta} - \theta_1) \ge \delta)
    \right].
\end{align*}
Setting $\delta = \delta_{01} \defeq \half \lossdist(P_0, P_1)$
and using the
implication~\eqref{eqn:exclusion},  we obtain
\begin{align}
  \nonumber
  \E_{\marginprob_0^n}\left[L(\what{\theta} - \theta_0)\right]
  +
  \E_{\marginprob_1^n}\left[L(\what{\theta} - \theta_1)\right]
  & \ge
  \delta_{01} \left[1 -
    \marginprob_0^n(L(\what{\theta} - \theta_0) < \delta)
    + \marginprob_1^n(L(\what{\theta} - \theta_1) \ge \delta)
    \right] \\
  & \ge
  \delta_{01} \left[1 -
    \marginprob_0^n(L(\what{\theta} - \theta_1) \ge \delta)
    + \marginprob_1^n(L(\what{\theta} - \theta_1) \ge \delta)
    \right]  \nonumber \\
  & \ge \delta_{01}
  \left[1 - \tvnorm{\marginprob_0^n - \marginprob_1^n}\right],
  \label{eqn:le-cam-application}
\end{align}
where in the last step we used the definition of the
variation distance.

Now we make use of the contraction inequality of
Corollary~\ref{corollary:tensorized-contraction-chi}
and its consequence~\eqref{eqn:tensorized-contraction-kl}
for KL-divergences.
By Pinsker's inequality and the corollary, we have
\begin{equation*}
  2 \tvnorm{\marginprob_0^n - \marginprob_1^n}^2
  \le \dkl{\marginprob_0^n}{\marginprob_1^n}
  \le \log(1+\dchi{\marginprob_0^n}{\marginprob_1^n})
  \le n \log \left(1 + 4 \diffp^2
  \tvnorm{P_0 - P_1}^2\right).
\end{equation*}
Substituting this into our preceding lower
bound~\eqref{eqn:le-cam-application} and using that $\what{\theta}$ is
arbitrary and $\delta_{01} = \half \lossdist(P_0, P_1)$, we have that for
any distributions $P_0$ and $P_1$,
\begin{equation*}
  \inf_{\what{\theta}}
  \inf_{\channel \in \channeldistset}
  \max_{P \in \{P_0, P_1\}} \E_P \left[L(\what{\theta} - \theta(P))\right]
  \ge \frac{1}{4} \lossdist(P_0, P_1)
  \left[1
    - \sqrt{\frac{n}{2} \log\left(1 + 4 \diffp^2 \tvnorm{P_0 - P_1}^2\right)
    }\right].
\end{equation*}
Now, for any $\delta \ge 0$, if
$\frac{n}{2} \log(1 + 4 \diffp^2 \delta^2)
\le \frac{1}{4}$, or equivalently,
$\delta^2 \le \frac{1}{4 \diffp^2} (\exp(\frac{1}{2n}) - 1)$,
then $1 - \sqrt{\frac{n}{2} \log(1 + 4 \diffp^2 \delta^2)} \ge \half$.
Applying this to the bracketed term in the preceding display,
we obtain
\begin{align*}
  \localminimax_n(P_0, L, \mc{P}, \channeldistset)
  & \ge \frac{1}{8} \sup_{P_1 \in \mc{P}}
  \left\{\lossdist(P_0, P_1)
  \mid \tvnorm{P_0 - P_1}^2
  \le \frac{1}{4 \diffp^2} \left[e^{\frac{1}{2n}} - 1\right]\right\} \\
  & \ge \frac{1}{8}
  \sup_{P_1 \in \mc{P}}
  \left\{\lossdist(P_0, P_1) \mid \tvnorm{P_0 - P_1}^2
  \le \frac{1}{8 n \diffp^2}\right\}
\end{align*}
because $e^x - 1 \ge x$ for all $x$. When $\loss$ is
convex, this is precisely
$\frac{1}{4} \lossmodcont(\frac{1}{\sqrt{8 n \diffp^2}}; P_0, \mc{P})$,
while in the quasi-convex case, it is at least
$\frac{1}{8} \lossmodcont(\frac{1}{\sqrt{8 n \diffp^2}}; P_0, \mc{P})$.

\subsection{Proof of Proposition~\ref{proposition:super-efficiency}}
\label{sec:proof-super-efficiency}

Our starting point is a lemma extending~\cite[Thm.~1]{BrownLo96}.  In the
lemma and the remainder of this section, for measures $P_0$ and $P_1$ we
define the $2$-affinity
\begin{equation*}
  \chipone{P_0}{P_1}
  \defeq \dchi{P_0}{P_1} + 1
  = \E_{P_1}\left[\frac{dP_0^2}{dP_1^2}\right]
  = \E_{P_0}\left[\frac{dP_0}{dP_1}\right],
\end{equation*}
which measures the similarity between distributions $P_0$ and $P_1$.
With these definitions, we have the following constrained risk inequality.

\begin{lemma}[\cite{DuchiRu18}, Theorem 1]
  \label{lemma:constrained-risk}
  Let $\theta_0 = \theta(P_0)$, $\theta_1 = \theta(P_1)$,
  and define $\Delta = \Phi(\half \ltwo{\theta_0 - \theta_1})$.
  If the estimator $\what{\theta}$ satisfies
  $\risk(\what{\theta}, \theta_0, P_0) \le \delta$ for some $\delta \ge 0$, then
  \begin{equation*}
    \risk(\what{\theta}, \theta_1, P_1)
    \ge 
      \hinge{\Delta^{1/2} - (\chipones{P_1}{P_0} \cdot \delta)^{1/2}}^2.
  \end{equation*}
\end{lemma}
\noindent
The lemma shows that if an estimator has small risk under
distribution $P_0$, then its risk for a nearby
distribution $P_1$ must be nearly the distance between the associated
parameters $\theta_0$ and $\theta_1$.

With Lemma~\ref{lemma:constrained-risk} in hand, we can prove
Proposition~\ref{proposition:super-efficiency}.
For shorthand let
$\risk_a(\what{\theta}) = \risk(\what{\theta}, \theta_a, \marginprob_a^n)$
denote the risk under the marginal $\marginprob_a^n$.
By Lemma~\ref{lemma:constrained-risk},
for any distributions $P_0$ and $P_1$, we have
\begin{equation*}
  \risk_1(\what{\theta})
  \ge \hinge{\Phi \left(\half \ltwo{\theta_0 - \theta_1}\right)
    - \left(\chipone{\marginprob_1^n}{\marginprob_0^n}
    \risk(\what{\theta}, \marginprob_0^n)\right)^{1/2}}^2,
\end{equation*}
and by Corollary~\ref{corollary:tensorized-contraction-chi}
we have
\begin{equation*}
  \chipone{\marginprob_1^n}{\marginprob_0^n}
  \le \left(1 + 4 \diffp^2 \tvnorm{P_0 - P_1}^2\right)^n
  \le \exp\left(4 n \diffp^2 \tvnorm{P_0 - P_1}^2\right).
\end{equation*}
Let $\modcont(\delta; P_0) = \lossmodcont(\delta; P_0, \mc{P})$ for
shorthand. For $t \in [0, 1]$, let $\mc{P}_t$ be the collection of
distributions
\begin{equation*}
  \mc{P}_t
  \defeq \left\{P \in \mc{P}
  \mid \tvnorm{P_0 - P_1}^2
  \le t \frac{ \log \frac{1}{\eta}}{4 n \diffp^2}\right\},
\end{equation*}
so that under the conditions of the proposition,
any distribution $P_1 \in \mc{P}_t$ satisfies
\begin{equation}
  \label{eqn:risk-almost-transferred}
  \risk_1(\what{\theta})
  \ge \hinge{\Phi \left(\half \ltwo{\theta_0 - \theta_1}\right)
    - \eta^{\frac{(1 - t)}{2}}
    \modcont\left((4 n \diffp^2)^{-1/2}; P_0\right)^{1/2}}^2.
\end{equation}

As $\loss(\half (\theta_0 - \theta(P_1)))
= \Phi(\half \ltwo{\theta_0 - \theta(P_1)})$,
inequality~\eqref{eqn:risk-almost-transferred}
implies that for all $t
\in [0, 1]$, there exists $P_1 \in \mc{P}_t$ such that
\begin{equation*}
  \risk(\what{\theta}, \marginprob_1^n)
  \ge
  \hinge{\modcont\left(\frac{\sqrt{t 
        \log\frac{1}{\eta}}}{\sqrt{4 n \diffp^2}}; P_0\right)^{1/2}
    -  \eta^{\frac{(1 - t)}{2}}
    \modcont\left(\frac{1}{\sqrt{4 n \diffp^2}}; P_0\right)^{1/2}
  }^2.
\end{equation*}
Because $\delta \mapsto \modcont(\delta)$ is non-decreasing,
if $t \in [0, 1]$ we may choose $P_1 \in \mc{P}_t$
such that
\begin{equation}
  \risk(\what{\theta}, \marginprob_1^n)
  \ge
  \hinge{1  - \eta^{(1 - t)/2}}^2
  \modcont\left(\frac{\sqrt{t 
      \log\frac{1}{\eta}}}{\sqrt{4 n \diffp^2}}; P_0\right).
  \label{eqn:intermediate-eta-risk-bound}
\end{equation}

Lastly, we lower bound the modulus of continuity at $P_0$ by
a modulus at $P_1$.
We claim that under Condition~\ref{cond:reverse-triangle},
for all $\delta > 0$, if
$\tvnorm{P_0 - P_1} \le \delta$ then
\begin{equation}
  \label{eqn:transfer-modcont}
  \modcont(2 \delta; P_0)
  \ge \frac{1}{2\Creverse} \modcont(\delta; P_1).
\end{equation}
Deferring the proof of this claim, note that
by
taking $\delta^2 = t \log \frac{1}{\eta} / (16 n \diffp^2)$
in inequality~\eqref{eqn:transfer-modcont},
Eq.~\eqref{eqn:intermediate-eta-risk-bound} implies
that there exists $P_1 \in \mc{P}_t$ such that
\begin{equation*}
  \risk(\what{\theta}, \marginprob_1^n)
  \ge
  \hinge{1 - \eta^{(1 - t)/2}}^2
  \modcont\left(2 \delta; P_0\right)
  \ge
  \frac{1}{2\Creverse} \hinge{1 - \eta^{(1 - t)/2}}^2
  \modcont\left(\frac{1}{4} \frac{\sqrt{t 
      \log\frac{1}{\eta}}}{\sqrt{n \diffp^2}}; P_1\right).
\end{equation*}

Let us return to the claim~\eqref{eqn:transfer-modcont}.  
For distributions $P_0, P_1, P_2$ with
parameters $\theta_a = \theta(P_a)$,
\begin{align*}
  \loss\left(\frac{\theta_1 - \theta_2}{2}\right)
  \le \loss(\theta_0 - \theta_1) + \loss(\theta_0 - \theta_2)
  & \leq
  \Creverse \loss\left(\frac{\theta_0 - \theta_1}{2}\right)
  + \Creverse \loss\left(\frac{\theta_0 - \theta_2}{2}\right)
\end{align*}
by Condition~\ref{cond:reverse-triangle}.  Then for any $\delta \ge 0$ and
$P_1$ with $\tvnorm{P_1 - P_0} \le \delta$, we have
\begin{align*}
  \modcont(2 \delta; P_0)
  & =
  \sup_{\tvnorm{P_0 - P}
    \le 2 \delta}
  \loss\left(\frac{\theta_0 - \theta(P)}{2}\right)
  \ge \sup_{\tvnorm{P_1 - P} \le
    \delta} \loss\left(\frac{\theta_0 - \theta(P)}{2}\right) \\
  & \ge 
  \sup_{\tvnorm{P - P_1} \le \delta}
  \left\{\Creverse^{-1} \loss\left(\frac{\theta_1 - \theta(P)}{2}\right)
  - \loss\left(\frac{\theta_0 - \theta_1}{2}\right)\right\}
  \ge \Creverse^{-1} \modcont(\delta; P_1)
  - \modcont(\delta; P_0).
\end{align*}
Rearranging, we have inequality~\eqref{eqn:transfer-modcont},
as for any distribution $P_1$
such that $\tvnorm{P_0 - P_1} \le \delta$,
\begin{equation*}
  2 \modcont(2 \delta; P_0)
  \ge
  \modcont(\delta; P_0)
  + \modcont(2 \delta; P_0)
  \ge \Creverse^{-1} \modcont(\delta; P_1).
\end{equation*}


\section{Discussion}
\label{sec:discussion}

By the careful construction of locally optimal and adaptive estimators, as
well as our local minimax lower bounds, we believe results in this paper
indicate more precisely the challenges associated with locally private
estimation.  To illustrate this, let us reconsider the estimation of a
linear functional $v^T \theta$ in a classical statistical problem. 
Let $\{P_\theta\}$ be a family with Fisher information matrices
$\{\fisher[\theta]\}$ and score $\dot{\ell}_\theta : \mc{X} \to
\R^d$. Then a classical estimators $\what{\theta}_n$ of the parameter
$\theta_0$ is efficient~\cite[Sec.~8.9]{VanDerVaart98} among regular
estimators if and only if
\begin{equation*}
  \what{\theta}_n
  - \theta_0
  = \frac{1}{n} \sum_{i = 1}^n
  -\fisher[\theta_0]^{-1} \dot{\ell}_{\theta_0}(X_i)
  + o_P(1 / \sqrt{n}),
\end{equation*}
and an efficient estimator $\what{\psi}_n$ of $v^T\theta$ satisfies
$\what{\psi}_n = v^T \theta_0 - n^{-1} \sum_{i = 1}^n v^T
\fisher[\theta_0]^{-1} \dot{\ell}_{\theta_0}(X_i) + o_P(n^{-1/2})$.  In
constrast, in the private case, our rate-optimal estimators
(recall Section~\ref{sec:mis-specified-expfam}) in the nonparametric
case have the asymptotic form
\begin{equation*}
  \what{\psi}_{{\rm priv},n} = v^T \theta_0
  -v^T
  \bigg(\frac{1}{n} \sum_{i = 1}^n \fisher[\theta_0]^{-1}\dot{\ell}_{\theta_0}(X_i)
  \bigg)
  + \frac{1}{n} \sum_{i = 1}^n W_i
  + o_P(1 / \sqrt{n}),
\end{equation*}
where the random variables $W_i$ must add noise of a magnitude scaling as
$\frac{1}{\diffp} \sup_x |v^T \fisher[\theta_0]^{-1}
\dot{\ell}_{\theta_0}(x)|$, because otherwise it is possible to distinguish
examples for which $v^T \fisher[\theta_0]^{-1} \dot{\ell}_{\theta_0}(X_i)$
is large from those for which it has small magnitude.  This enforced lack of
distinguishability of ``easy'' problems (those for which the scaled score
$\fisher[\theta_0]^{-1} \dot{\ell}_{\theta_0}(X_i)$ is typically small) from
``hard'' problems (for which it is large) is a feature of local privacy
schemes, and it helps to explain the difficulty of estimation, as well as to
illustrate the more nuanced scaling of the best possible estimators with
problem parameter $\theta_0$, when $\sup_x |v^T \fisher[\theta_0]^{-1}
\score_{\theta_0}(x)|$ may be similar to $\E_0[(v^T \fisher[\theta_0]^{-1}
  \score_{\theta_0}(X))^2]^{1/2}$, the optimal non-private asymptotic
variance.

We thus believe it prudent to more carefully explore feasible definitions of
privacy, especially in local senses. Regulatory decisions and protection
against malfeasance may require less stringent notions of privacy than pure
differential privacy, but local notions of privacy---where no sensitive
non-privatized data leaves the hands of a sample participant---are
desirable. The asymptotic expansions above suggest a notion of privacy
that allows some type of \emph{relative} noise addition, to preserve
the easiness of ``easy'' problems, will help.
Perhaps large values of $\diffp$, at least for high-dimensional
problems, may still provide acceptable privacy protection, at least in
concert with centralized privacy guarantees.  We look forward
to continuing study of these fundamental limitations and acceptable
tradeoffs between data utility and protection of study participants.

\appendix


\section{Proofs of non-private minimax results}

In this appendix, we collect the (more or less standard) proofs
of the results in Section~\ref{sec:classical-local-minimax}.

\subsection{Proof of Proposition~\ref{proposition:classical-local-minimax}}
\label{sec:proof-classical-local-minimax}

The lower bound follows the typical reduction of estimation to testing
commin in the literature on lower bounds~\cite{AgarwalBaRaWa12, Tsybakov09,
  Yu97}. Fix any distribution $P_1 \in \mc{P}$, let
$\theta_v = \theta(P_v)$ for shorthand, and define $\delta = |\theta_0
- \theta_1| / 2$. Then
for any $\theta \in \R$, that
$|\theta - \theta_0| < \delta$ implies
$|\theta - \theta_1| \ge \delta$.
Thus we have
\begin{align*}
  \E_{P_0^n}\left[(\what{\theta} - \theta_0)^2\right]
  + \E_{P_1^n}\left[(\what{\theta} - \theta_1)^2\right]
  & \stackrel{(i)}{\ge}
  \delta^2 \left[P_0^n\left(|\what{\theta} - \theta_0| \ge \delta
    \right) + P_1^n\left(|\what{\theta} - \theta_1| \ge \delta \right)
    \right] \\
  & = \delta^2 \left[1 - P_0^n\left(|\what{\theta} - \theta_0| < \delta\right)
    + P_1^n\left(|\what{\theta} - \theta_1| \ge \delta \right) \right] \\
  & \stackrel{(ii)}{\ge}
  \delta^2 \left[1 - P_0^n\left(|\what{\theta} - \theta_1| \ge \delta
    \right) + P_1^n\left(|\what{\theta} - \theta_1| \ge \delta \right)
    \right],
\end{align*}
where inequality~$(i)$ is Markov's inequality, and the second
is the implication preceding the display.
By the definition of variation distance and that
$\tvnorm{P - Q} \le \sqrt{2} \dhel(P, Q)$ for any $P, Q$,
we obtain
\begin{equation}
  \label{eqn:standard-minimax-easy}
  \E_{P_0^n}\left[(\what{\theta} - \theta_0)^2\right]
  + \E_{P_1^n}\left[(\what{\theta} - \theta_1)^2\right]
  \ge \delta \left[1 - \tvnorm{P_0^n - P_1^n}\right].
  \ge \delta \left[1 - \sqrt{2} \dhel(P, Q) \right].
\end{equation}
The tensorization properties of the Hellinger distance imply that
\begin{equation*}
  \dhel^2(P_0^n, P_1^n) = \left[1 - \left(1 - \dhel^2(P_0, P_1)\right)^n\right]
  \le n \dhel^2(P_0, P_1),
\end{equation*}
and substituting this into the bound~\eqref{eqn:standard-minimax-easy}
gives that for any $P_1 \in \mc{P}$,
\begin{equation*}
  \localminimax_n(P_0, \mc{P})
  \ge
  \half \E_{P_0^n}\left[(\what{\theta} - \theta_0)^2\right]
  + \half \E_{P_1^n}\left[(\what{\theta} - \theta_1)^2\right]
  \ge \frac{1}{8} \sup_{P_1 \in \mc{P}}
  \left(\theta(P_0) - \theta(P_1)\right)^2
  \hinge{1 - \sqrt{2 n \dhel^2(P_0, P_1)}}.
\end{equation*}
Taking a supremum over all $P_1 \in \mc{P}$ satisfying
$\dhel^2(P_0, P_1) \le \frac{1}{4 n}$ then implies
\begin{equation*}
  \localminimax_n(P_0, \mc{P})
  \ge \frac{\sqrt{2} - 1}{8 \sqrt{2}}
  \helmod^2(n^{-1/2} / 2; P_0, \mc{P}).
\end{equation*}

To prove the upper bound, we exhibit an estimator.  Let $\theta_v =
\theta(P_v)$ as above, and assume w.l.o.g.\ that $P_a$ have
densities $p_a$ (take base measure $\mu = P_0 + P_1$). Define the
acceptance set
$A \defeq \{x \in \mc{X}^n \mid \prod_{i=1}^n \frac{p_0(x_i)}{p_1(x_i)}
\ge 1\}$
and 
estimator
$\what{\theta}_n = \theta_0 1_A + \theta_1 1_{A^c}$.
It is then immediate that
\begin{equation*}
  \max_{P \in \{P_0, P_1\}}
  \E_{P^n}\left[(\what{\theta}_n - \theta(P))^2\right]
  = (\theta_0 - \theta_1)^2
  \max\left\{P_0^n(A^c), P_1^n(A) \right\}
  \le (\theta_0 - \theta_1)^2
  \left[1 - \tvnorm{P_0^n - P_1^n}\right].
\end{equation*}
Using the tensorization properties of Hellinger distance and
that $\tvnorm{P - Q} \ge \dhel^2(P, Q)$ for any distributions $P$ and $Q$,
we obtain
\begin{equation*}
  \tvnorm{P_0^n - P_1^n}
  \ge \dhel^2(P_0^n, P_1^n)
  = \left[1 - \left(1 - \dhel^2(P_0, P_1)\right)^n\right]
  \ge 1 - \exp\left(-n \dhel^2(P_0, P_1)\right),
\end{equation*}
so that
\begin{equation*}
  \max_{P \in \{P_0, P_1\}}
  \E_{P^n}\left[(\what{\theta}_n - \theta(P))^2 \right]
  \le (\theta_0 - \theta_1)^2 \exp\left(-n \dhel^2(P_0, P_1)\right).
\end{equation*}
Taking a supremum over $P_1$ gives the claimed upper bound.

Finally, we turn to the bound~\eqref{eqn:local-minimax-via-hellinger}.
We have by assumption that for any $\delta \ge n^{-1/2} / 2$,
\begin{align*}
  \left\{\helmod(\delta; P_0, \mc{P}) \exp(-n \delta^2)
  \right\}
  & \le \helmod(n^{-1/2} / 2; P_0, \mc{P})
  \cdot B (4n \delta^2)^{\beta/2}   \exp(-n \delta^2) \\
  & \le \helmod(n^{-1/2} / 2; P_0, \mc{P})
  B \beta^{\beta/2} e^{-\beta / 2},
\end{align*}
where the supremum is attained at $\delta^2 = \frac{\beta}{2n}$.

\subsection{Proof of Claim~\ref{claim:qmd-bounded-local-minimax}}
\label{sec:proof-qmd-bounded-local-minimax}

We use the shorthand $\helmod(\delta;
\theta_0, \Theta) = \sup_{\theta \in \Theta} \{|v^T (\theta_0 - \theta)|
\mid \dhel(P_\theta, P_{\theta_0}) \le \delta\}$.
The lower bound is nearly immediate via
Proposition~\ref{proposition:classical-local-minimax}: by the QMD assumption
there exists $\delta > 0$ such that $\norm{h} \le \delta$ implies
$\frac{1}{9} h^T I_{\theta_0} h \le \dhel^2(P_{\theta_0 + h}, P_\theta) \le
\frac{1}{7} h^T I_{\theta_0} h$.  Thus we obtain for all $n \gtrsim
\frac{1}{\lambda_{\min}(I_{\theta_0}) \delta^2}$ that
\begin{align*}
  \helmod(n^{-1/2} / 2; \theta_0, \Theta)
  & =
  \sup_h \Big\{|h^T v| \mid
  \dhel^2(P_{\theta_0 + h}, P_{\theta_0}) \le \frac{1}{4n} \Big\} \\
  & \ge \sup_{\norm{h} \le \delta} \left\{h^T v
  \mid h^T I_{\theta_0} h \le \frac{7}{4 n} \right\}
  = \sup \left\{h^T v \mid h^T I_{\theta_0} h \le \frac{7}{4n} \right\}
  = \frac{\sqrt{7}}{2\sqrt{n}} \ltwos{I_{\theta_0}^{-1/2} v}.
\end{align*}

For the upper bound,
choose $\delta > 0$ such that $\norm{h} \le \delta$ implies that
$\dhel^2(P_{\theta_0 + h}, P_{\theta_0}) \ge \frac{1}{9} h^T I_{\theta_0}
h$, while $\norm{h} > \delta$ implies that
$\dhel^2(P_{\theta_0 + h}, P_{\theta_0}) > \gamma > 0$; such a pair of
$\delta$ and $\gamma$ exist
by Assumption~\ref{assumption:identifiability} and
quadratic mean differentiability.
There thus exists $r_0 = r_0(\delta, \theta_0)$ such that
$\dhel(P_{\theta_0 + h}, P_{\theta_0}) \le r_0$ implies
$\norm{h} \le \delta$, and so for any $r \le r_0$, we have
that $\dhel^2(P_{\theta_0 + h}, P_{\theta_0}) \le r^2$ implies
\begin{equation*}
  \frac{1}{9} h^T I_{\theta_0} h \le
  \dhel^2(P_{\theta_0 + h}, P_{\theta_0}) \le r^2.
\end{equation*}
Using this in the definition of the modulus of continuity yields
\begin{equation*}
  \helmod(r; \theta_0, \Theta)
  = \sup_h \left\{|v^T h| \mid \dhel^2(P_{\theta_0 + h}, P_{\theta_0}) \le
  r^2 \right\}
  \le \sup_h \left\{v^T h \mid h^T I_{\theta_0} h \le 9 r^2 \right\}
  = 3 r \ltwos{I_{\theta_0}^{-1/2} v}
\end{equation*}
for all $r \le r_0$.  Noting that $\helmod \le \diam(\Theta)$ regardless,
we apply Proposition~\ref{proposition:classical-local-minimax} and observe
\begin{align*}
  \sup_{r \ge 0} 
  \left\{\helmod^2(r; \theta_0, \Theta)
  \exp(-n r^2)\right\}
  & \le \max\left\{
  \sup_{0 \le r \le r_0}
  \helmod^2(r; \theta_0, \Theta)
  \exp(-n r^2),
  \sup_{r > r_0} \helmod^2(r; \theta_0, \Theta)
  \exp(-n r^2) \right\} \\
  & \le \max\left\{
  \sup_{r \ge 0}
  9 r^2 v^T I_{\theta_0}^{-1} v
  \exp\left(-n r^2 \right),
  \diam^2(\Theta) \exp(-n r_0^2)
  \right\} \\
  & = \max\left\{\frac{9}{e n} v^T I_{\theta_0}^{-1} v,
  \diam^2(\Theta) \exp(-n r_0^2)\right\}.
\end{align*}

\subsection{Proof of Lemma~\ref{lemma:regular-differentiability}}
\label{sec:proof-regular-differentiability}

The proof is essentially~\cite[Lemma 8.14]{VanDerVaart98}.
Letting $\fisher[0] = \E_{P_0}[gg^T]$,
we have under
$P_0^n \times \auxdist^n$ that
\begin{equation*}
  \log \frac{dP_{h_n/\sqrt{n}}^n \times d\auxdist^n}{
    dP_0^n \times d\auxdist^n}(X_{1:n}, \auxvar_{1:n})
  = \frac{1}{\sqrt{n}} \sum_{i = 1}^n h^T g(X_i)
  - h^T \fisher[0] h + o_{P_0}(1)
\end{equation*}
(recall~\cite[Theorem 7.2]{VanDerVaart98}).
Thus we have
\begin{equation*}
  \left[\begin{matrix}
      \sqrt{n}(\what{\theta}_n - \theta(P_0)) \\
      \log \frac{dP_{h_n/\sqrt{n}}^n \times d\auxdist^n}{
        dP_0^n \times d\auxdist^n}
    \end{matrix}
    \right]
  \mathop{\cdlong}_{P_0 \times \auxdist} \normal\left(
  \left[ \begin{matrix}
      0 \\ -\half h^T \fisher[0] h \end{matrix} \right],
  \left[\begin{matrix} \Sigma_0 + \Sigma_\aux &
      \E[\influencefunc(X) g(X)^T] h \\
      h^T \E[g(X) \influencefunc(X)^T]
      & h^T \fisher[0] h
    \end{matrix}\right]\right).
\end{equation*}
Applying the delta method and
Le Cam's third lemma~\cite[Example~6.7]{VanDerVaart98} gives
that
\begin{equation*}
  \sqrt{n}(\what{\theta}_n - \theta(P_0))
  \mathop{\cdlong}_{P_{h_n/\sqrt{n}} \times \auxdist}
  \normal\left(\E[\influencefunc(X) g(X)^T] h,
  \Sigma_0 + \Sigma_\aux\right).
\end{equation*}
The differentiability of $h \mapsto \theta(P_h)$ at $h = 0$ then
gives the first result.

The second limiting result follows by a standard compactness argument.

\subsection{Inequality~\eqref{eqn:hellinger-mean-modulus}:
  bounds on the Hellinger modulus}
\label{sec:proof-hellinger-mean-modulus}

For the lower bound on $\helmod(\delta)$, we use techniques from
semiparametric inference~\cite[e.g.][Ch.~25]{VanDerVaart98}.
Let $\theta_0 = \E_{P_0}[X]$ and define the function
$g(x) = (x - \theta_0)$.
Define the distribution $dP_t = \hinge{1 + tg} dP_0 / C_t$, where
$C_t = \int \hinge{1 + tg} dP_0$. Then we have
\begin{align*}
  1 \le C_t
  & \le \int (1 + tg) dP_0
  + t^2 \int \frac{\hinge{1 + t g} - (1 + tg)}{t^2} dP_0 \\
  & = 1 - t^2 \int \frac{1}{t^2} (1 + tg) \indic{g \le -1/t} dP_0 \\
  & \le 1 + t^2 \left[\frac{P_0(g \le -1/t)}{t^2}
    + \frac{\sqrt{\var_0(g) P_0(g \le -1/t)}}{t} \right]
  \le 1 + 2 t^2 \var_0(g)
\end{align*}
by Chebyshev's inequality. A standard
calculation~\cite[Ch.~25.3]{VanDerVaart98} via the dominated
convergence theorem---with the observation that the influence
function of the mean is $\influencefunc(x) = (x - \theta_0)$---yields
\begin{equation*}
  \E_{P_t}[X] = \theta_0 + t \var_{P_0}(X) (1 + o(1))
  ~~ \mbox{and} ~~
  \dhel^2(P_t, P_0) = \frac{1}{8} t^2 \var_{P_0}(X) (1 + o(1))
\end{equation*}
as $t \to 0$. Let $t_0$ be small enough that
$|o(1)| \le \frac{1}{7}$ for $t \le t_0$.
Then for $\delta^2 \le \var_0(X) t_0 / 7$,
\begin{equation*}
  \helmod(\delta)
  \ge \sup_{t \le t_0}
  \left\{ \frac{7}{8} t \var_{P_0}(X) \mid \frac{1}{7} t^2
  \var_0(X) \le \delta^2 \right\}
  = \frac{7 \sqrt{7}}{8} \sqrt{\var_{P_0}(X)},
\end{equation*}
while (as $o(1) \to 0$)
\begin{equation*}
  \liminf_{\delta \downarrow 0}
  \frac{\helmod(\delta)}{\sqrt{8 \var_{P_0}(X)} \delta} \ge 1.
\end{equation*}

For the upper bound on $\helmod$, we require a few more steps. Let $P_1
\in \mc{P}$ be an arbitrary distribution, where we assume that
$\dhel^2(P_0, P_1) \le \frac{1}{4}$, and use the shorthand and $\theta_1 =
\theta(P_1)$. Then
\begin{align}
  \theta_1 - \theta_0
  & = \int (x - \theta_0) (dP_1 - dP_0)
  = \int (x - \theta_0) (\sqrt{dP_1} + \sqrt{dP_0})
  (\sqrt{dP_1} - \sqrt{dP_0}) \nonumber \\
  & \le \left(\int (x - \theta_0)^2 (\sqrt{dP_1} + \sqrt{dP_0})^2\right)^{1/2}
  \sqrt{2} \dhel(P_0, P_1) \nonumber \\
  & \le \left(2 \E_0[(X - \theta_0)^2] + 2 \E_1[(X - \theta_0)^2]
  \right)^{1/2} \sqrt{2} \dhel(P_0, P_1) \nonumber \\
  & = 2 \sqrt{\var_0(X) + \var_1(X) + (\theta_0 - \theta_1)^2}
  \cdot \dhel(P_0, P_1)
  \label{eqn:bound-mean-diff-by-hellinger}
\end{align}
by the Cauchy-Schwarz inequality and definition of Hellinger distance.
Noting that
$\E[(X - \theta)^4]
\le 2^3 (\E[X^4] + \theta^4)$,
we may assume there exists some $M_4 < \infty$ such that
$M_4^4 \ge \E_P[(X - \theta_0)^4]$ for all $P \in \mc{P}$.
We now bound the variance $\var_1(X)$ in terms of $\var_0(X)$ and
$\dhel(P_0, P_1)$.
Using that
\begin{equation*}
  (x - \theta_0)^2 - (x - \theta_1)^2 =
  -2x(\theta_0 - \theta_1) + \theta_0^2 - \theta_1^2,
\end{equation*}
we obtain
\begin{align*}
  \var_0(X) - \var_1(X)
  & = \int (x - \theta_0)^2 (dP_0 - dP_1) - 2 \theta_1 (\theta_0 - \theta_1)
  + \theta_0^2 - \theta_1^2 \\
  & = \int(x - \theta_0)^2 (\sqrt{dP_0} + \sqrt{dP_1})
  (\sqrt{dP_0} - \sqrt{dP_1}) + (\theta_0 - \theta_1)^2.
\end{align*}
Again applying Cauchy-Schwarz, we observe that
\begin{align*}
  |\var_0(X) - \var_1(X)|
  & \le 2 M_4^2 \dhel(P_0, P_1) + (\theta_0 - \theta_1)^2.
\end{align*}
Substituting this bound into
inequality~\eqref{eqn:bound-mean-diff-by-hellinger}
and squaring yields
\begin{equation*}
  (\theta_1 - \theta_0)^2
  \le 4 \left(2 \var_0(X) + 2 M_4^2 \dhel(P_0, P_1)
  + 2 (\theta_0 - \theta_1)^2 \right) \dhel^2(P_0, P_1),
\end{equation*}
or
\begin{equation*}
  (\theta_1 - \theta_0)^2
  \le \frac{8 \var_0(X)}{1 - 8 \dhel^2(P_0, P_1)}
  \dhel^2(P_0, P_1)
  + \frac{8 M_4^2}{1 - 8 \dhel^2(P_0, P_1)}
  \dhel^3(P_0, P_1).
\end{equation*}
In particular, as soon as
$\dhel^2(P_0, P_1) \le \frac{1}{16}$
and $\dhel(P_0, P_1) \le (\var_0(X) / M_4^2)^{1/3}$,
\begin{equation*}
  (\theta_1 - \theta_0)^2
  \le 32 \var_0(X) \dhel^2(P_0, P_1).
\end{equation*}
Solving for the modulus~\eqref{eqn:hellinger-modcont} gives the result,
and eliminating higher order terms yields
\begin{equation*}
  \limsup_{\delta \downarrow 0} \sup \left\{
  \frac{|\theta(P_0) - \theta(P_1)|}{
    \dhel(P_0, P_1)} \mid \dhel(P_0, P_1) \le \delta \right\}
  \le \sqrt{8 \var_{P_0}(X)}.
\end{equation*}

\subsubsection{Proof of Claim~\ref{claim:non-private-mis-specified}}
\label{sec:proof-non-private-mis-specified}

A minor extension of Proposition~\ref{proposition:classical-local-minimax}
shows there exist numerical constants $0 < c_0, c_1 < \infty$ such that
\begin{equation}
  \label{eqn:min-square-classical}
  c_0 (\helmod^2(c_0 n^{-1/2}; P_0, \mc{P}) \wedge B)
  \le \localminimax_n(P_0, \losstrunc, \mc{P}, \{\textup{id}\})
  \le c_1 \sup_{r \ge 0} \left\{
  \left(\helmod^2(r; P_0, \mc{P}) \wedge B \right) e^{-n r^2}\right\}
\end{equation}
for all $n \in \N$ and any family $\mc{P}$.
For the influence function $\influencefunc(x) = (x - \E_0[X]) / \var_0(X)$
following the claim,
by appropriate renormalization, we may apply the limiting
equality~\eqref{eqn:hellinger-mean-modulus}
to  obtain
\begin{equation*}
  \helmod(\delta, P_0, \Pnonpara)
  = \frac{\sqrt{8}}{\sqrt{\var_0(X)}} \delta (1 + o(1))
  = \sqrt{8 \delta^2 \E_0[\influencefunc(X)^2]}(1 + o(1)).
\end{equation*}
The lower bound $\localminimax_n(P_0, \losstrunc, \Pnonpara,
\{\textup{id}\}) \gtrsim \frac{1}{n \var_0(X)}$ for large $n$ then follows
by inequality~\eqref{eqn:min-square-classical}.  The matching upper bound
similarly follows, as for all large enough $n$, if $r \le 1/n^{1/4}$ we have
\begin{equation*}
  \helmod^2(r; P_0, \Pnonpara)
  \le 16 r^2 \E_0[\influencefunc(X)^2],
\end{equation*}
and so
\begin{align*}
  \sup_{r \ge 0} \left\{
  \left(\helmod^2(r; P_0, \Pnonpara) \wedge B \right) e^{-n r^2} \right\}
  & \le \max\left\{
  \sup_{r \ge 0} 16 r^2 \E_0[\influencefunc(X)^2] e^{-n r^2},
  \sup_{r \ge n^{-1/4}} B e^{-n r^2} \right\} \\
  & = O(1) \max\left\{\frac{1}{n} \E_0[\influencefunc(X)^2],
  B e^{-\sqrt{n}}\right\}.
\end{align*}

A derivation \emph{mutatis mutandis} identical
to that for Claim~\ref{claim:qmd-bounded-local-minimax} gives the parametric
result, as the exponential family
has score $\score_{\theta_0}(x) = x - \E_{\theta_0}[X]$
and Fisher information $\fisher[\theta_0] = \var_0(X)$.


\section{Deferred main proofs}

\subsection{Proof of Proposition~\ref{proposition:achievable}}
\label{sec:proof-achievable}

Let $P_0$ and $P_1$ be distributions on $\mc{X}$, each with
densities $p_0,p_1$ according to some base measure $\mu$.
Let $\theta_a = \theta(P_a)$, and consider the problem of privately
collecting observations and deciding whether $\theta = \theta_0$ or $\theta
= \theta_1$. We define a randomized-response estimator
for this problem using a simple hypothesis test.
Define the acceptance set
$A \defeq \left\{x \in \mc{X} \mid p_0(x) > p_1(x) \right\}$,
so $P_0(A) - P_1(A) = \tvnorm{P_0 - P_1}$.
Now, consider the following estimator:
for each $X_i$, define
\begin{equation*}
  T_i = \indic{X_i \in A}
  ~~ \mbox{and} ~~
  Z_i \mid \{T_i = t\}
  = \begin{cases}
    1 & \mbox{with~probability~}
   (e^\diffp + 1)^{-1}\left(e^{\diffp} t + 1-t\right) \\
    0 & \mbox{with~probability~}
   (e^{-\diffp} + 1)^{-1}\left(e^{-\diffp} t + 1-t\right)
  \end{cases}
\end{equation*}
Then the channel $\channel(\cdot \mid X_i)$ for $Z_i \mid X_i$ is
$\diffp$-differentially-private, and
setting $\delta_\diffp = \frac{e^\diffp}{1 + e^\diffp} - \half$,
we have
\begin{equation*}
  \E_0[Z_i]
  = \frac{1 + \delta_\diffp}{2} P_0(A)
  + \frac{1 - \delta_\diffp}{2} P_0(A^c)
  = \frac{1- \delta_\diffp}{2} + \delta_\diffp P_0(A)
  ~~ \mbox{and} ~~
  \E_1[Z_i]
  = \frac{1- \delta_\diffp}{2} + \delta_\diffp P_1(A)
\end{equation*}
while $Z_i \in \{0, 1\}$.
Define the statistic
\begin{equation*}
  K_n \defeq \frac{1}{\delta_\diffp}
  \left(\frac{1}{n} \sum_{i = 1}^n Z_i - \frac{1-\delta_\diffp}{2}\right),
\end{equation*}
so that $\E_0[K_n] = P_0(A)$ and $\E_1[K_n] = P_1(A)$.
We define the estimator
\begin{equation*}
  \what{\theta} \defeq
  \theta_0 \indic{K_n \ge \frac{P_0(A) + P_1(A)}{2}}
  + \theta_1 \indic{K_n < \frac{P_0(A) + P_1(A)}{2}}.
\end{equation*}

We now analyze the performance of $\what{\theta}$.  By
construction of the acceptance set $A$,
\begin{equation*}
  \frac{P_0(A) + P_1(A)}{2}
  = P_0(A) + \frac{P_1(A) - P_0(A)}{2}
  = P_0(A) - \half \tvnorm{P_1 - P_0}
  = P_1(A) + \half \tvnorm{P_1 - P_0},
\end{equation*}
so by Hoeffding's inequality, we have
\begin{equation*}
  \max\left\{P_0\left(K_n \le \frac{P_0(A) + P_1(A)}{2}
  \right),
  P_1\left(K_n \ge \frac{P_0(A) + P_1(A)}{2}\right)
  \right\}
  \le \exp\left(- \frac{n \delta_\diffp^2 \tvnorm{P_0 - P_1}^2}{2}\right).
\end{equation*}
In particular, we have
\begin{equation*}
  \E_0[\loss(\what{\theta} - \theta_0)]
  + \E_1[\loss(\what{\theta} - \theta_1)]
  \le \left[\loss(\theta_1 - \theta_0)
    + \loss(\theta_0 - \theta_1)\right]
  \exp\left(-\frac{n \delta_\diffp^2 \tvnorm{P_0 - P_1}^2}{2}\right).
\end{equation*}
Using the growth condition~\ref{cond:reverse-triangle}, we obtain
\begin{align}
  \nonumber
  \E_0[\loss(\what{\theta} - \theta_0)]
  + \E_1[\loss(\what{\theta} - \theta_1)]
  & \le
  2 \Creverse \loss\left(\frac{\theta_0 - \theta_1}{2}\right)
  \exp\left(-\frac{n \delta_\diffp^2 \tvnorm{P_0 - P_1}^2}{2}
  \right) \\
  & \le 2 \Creverse \sup_{P \in \mc{P}}
  \loss \left(\frac{\theta_0 - \theta(P)}{2}\right)
  \exp\left(-\frac{n \delta_\diffp^2 \tvnorm{P_0 - P}^2}{2}
  \right) \nonumber \\
  & = 2 \Creverse \sup_{r \ge 0}
  \left\{\lossmodcont(r; P_0)
  \exp\left(-\frac{n \delta_\diffp^2 r^2}{2}\right)\right\}.
  \nonumber
\end{align}

\subsection{Proof of Corollary~\ref{corollary:l1-info-attained}}
\label{sec:proof-l1-info-attained}
Define the shorthand
\begin{equation*}
  \tvmod(\delta)
  \defeq \sup_h \left\{|\functional(\theta_0 + h) - \functional(\theta_0)|
  ~ \mbox{s.t.} ~
  \tvnorm{P_{\theta_0 + h} - P_{\theta_0}} \le \delta \right\}.
\end{equation*}
We first apply Proposition~\ref{proposition:achievable}. As in this
setting the constant $\Creverse = O(1)$ from
Condition~\ref{cond:reverse-triangle} is automatically a universal constant,  
we obtain for numerical constants $C_0, C_1 < \infty$ that
\begin{equation}
  \label{eqn:local-minimax-to-phi}
  \localminimax(P_{\theta_0}, \loss, \pfamily, \channeldistset)
  \le
  C_0 \sup_{\tau \ge 0} \loss\left(\tvmod(C_1 \tau)\right) e^{-n \tau^2 \diffp^2}.
\end{equation}

We bound $\tvmod(\delta)$ for small $\delta$.
Let $r_\functional$ and $r_0$ be remainders as in the proof of
Proposition~\ref{proposition:information-private-lb}, so that
$\tvnorm{P_{\theta_0 + h} - P_\theta} = \half \information[\theta_0]|h| +
r_0(h)$ and $\functional(\theta_0 + h) = \functional(\theta_0) +
\functional'(\theta_0) h + r_\functional(h)$, where both are $o(h)$ as
$|h| \to 0$. Let $h_0 > 0$ be such that $|r_0(h)| \le |J_{\theta_0} h| /
4$ for $|h| \le h_0$.  Choose
$\delta_0 > 0$ so that $\tvnorm{P_{\theta_0 + h} - P_{\theta_0}} \le
\delta_0$ implies $|h| \le h_0$, which is possible by
Assumption~\ref{assumption:identifiability},
and $|h| \le \frac{4 \delta}{\information[\theta_0]}$ implies
$|r_\functional(h)| \le |\functional'(\theta_0) h|$.
Then for all $\delta \le \delta_0$,
if $\tvnorm{P_{\theta_0 + h} - P_{\theta_0}} \le \delta$, we have
$|h| \le h_0$ and consequently
\begin{equation*}
  \delta \ge \tvnorm{P_{\theta_0 + h} - P_{\theta_0}}
  = \half \information[\theta_0] |h| + r_0(h)
  \ge \frac{1}{4} \information[\theta_0]|h|,
\end{equation*}
or $|h| \le 4 \delta / \information[\theta_0]$.
Thus we obtain
\begin{align*}
  \tvmod(\delta)
  & =
  \sup_h \left\{|\functional(\theta_0 + h) - \functional(\theta_0)|
  ~\mbox{s.t.}~ \tvnorm{P_{\theta_0 + h} - P_{\theta_0}} \le \delta \right\} \\
  & \le
  \sup_h \left\{|\functional'(\theta_0) h + r_\functional(h)|
  ~\mbox{s.t.}~ |h| \le \frac{4 \delta}{\information[\theta_0]}\right\}
  \le \frac{8 |\functional'(\theta_0)|}{\information[\theta_0]}
  \delta.
\end{align*}

We now return to inequality~\eqref{eqn:local-minimax-to-phi}.
Substituting the preceding bound,
for numerical constants $C_0, C_1 < \infty$ whose values
may change from line to line,
\begin{align*}
  \localminimax_n(P_{\theta_0}, \loss, \pfamily, \channeldistset)
  & \le C_0 \max\left\{
  \sup_{0 \le \tau \le \delta_0 / C_1}
  \loss\left(\tvmod\left(C_1 \tau\right)\right)
  e^{-\tau^2 n \diffp^2},
  \loss\left(\diam(\functional(\Theta))\right)
  e^{-\delta_0^2 n \diffp^2 / C_1^2}
  \right\} \\
  & \le C_0 \max\left\{
  \sup_{0 \le \tau \le \delta_0 / C_1}
  \loss\left(\frac{C_1 |\functional'(\theta_0)|}{\information[\theta_0]}
  \tau \right) e^{-\tau^2 n \diffp^2},
  \loss\left(\diam(\functional(\Theta))\right)
  e^{-\delta_0^2 n \diffp^2 / C_1^2} \right\}.
\end{align*}
Finally, we use the assumption that
$\loss(at) \le C a^\beta \loss(t)$ for all $a \ge 1$.
We have
\begin{equation*}
  \loss\left(\frac{C_1 |\functional'(\theta_0)|}{\information[\theta_0]}
  \tau \right)
  \le C' (n \diffp^2 \tau^2)^{\beta/2} \loss\left(\frac{|\functional'(\theta_0)|}{
    \information[\theta_0]} \frac{1}{\sqrt{n \diffp^2}}\right)
\end{equation*}
and using that $\sup_t t^{\beta/2} e^{-t} = (\beta/2)^{\beta/2} e^{-\beta/2}$ gives
the result.

\subsection{Proof of Proposition~\ref{proposition:information-private-lb}}
\label{sec:proof-information-private-lb}

We assume that $\functional'(\theta_0) \neq 0$; the result is otherwise
trivial. Applying Theorem~\ref{theorem:modulus-of-continuity}, we have
\begin{equation}
  \label{eqn:theorem-modulus-of-continuity-example}
  \localminimax_n(P_{\theta_0}, L, \mc{P}, \channeldistset_\diffp) \geq
  \frac{1}{8} \lossmodcont
  \left(\frac{1}{\sqrt{8 n \diffp^2}}; P_{\theta_0}, \mc{P}\right).
\end{equation}
Now, we evaluate $\modcont(\delta) \defeq \lossmodcont(\delta;
P_{\theta_0}, \mc{P})$ for small $\delta > 0$.  By assumption, there exist
remainders $r_0$ and $r_\functional$, both satisfying $|r(h) / h| \to 0$
as $h \to 0$, such that $\tvnorms{P_{\theta_0 + h} - P_{\theta_0}} = \half
\information[\theta_0]|h| + r_0(h)$ and $\functional(\theta_0 + h) -
\functional(\theta_0) = \functional'(\theta_0) h + r_\functional(h)$.
Then
\begin{align*}
  \modcont(\delta)
  & \ge \sup_h
  \left\{\loss\left(\half (\functional(\theta_0 + h) -
  \functional(\theta_0))\right)
  \mid \tvnorm{P_{\theta_0 + h} - P_{\theta_0}}
  \le \delta \right\} \\
  & = \sup_h
  \left\{\loss\left(\half (\functional'(\theta_0) h
  + r_\functional(h))\right)
  \mid \information[\theta_0] |h| + 2 r_0(h)
  \le 2 \delta \right\}.
\end{align*}
Choose $h_0 = h(\theta_0, \functional, \mc{P}) > 0$ such that
$|h| \le h_0$ implies that
$|r_0(h)| \le \information[\theta_0] |h| / 2$ and
$|r_\functional(h)| \le |\functional'(\theta_0) h| / 5$. Then evidently
\begin{align*}
  \modcont(\delta)
  & \ge \sup_{|h| \le h_0}
  \left\{\loss\left(\frac{2}{5} |\functional'(\theta_0) h|\right)
  \mid \information[\theta_0] |h|
  \le \delta \right\}
  \stackrel{(\star)}{=}
  \loss\left(\frac{2}{5} \information[\theta_0] |\functional'(\theta_0)|
  \delta \right),
\end{align*}
where equality~$(\star)$ occurs whenever $\delta \le h_0 /
\information[\theta_0]$.  Setting $\delta = \frac{1}{\sqrt{8 n
    \diffp^2}}$, letting $n$ grow, and
substituting into inequality
inequality~\eqref{eqn:theorem-modulus-of-continuity-example} gives the
proposition.

\subsection{Proof of Theorem~\ref{theorem:general-private}}
\label{sec:proof-general-private}

The proof mirrors that of
Proposition~\ref{proposition:information-private-lb}.  Let $\theta_0 =
\theta(P_0)$ be the desired parameter.  Again, we assume that $\nabla
\functional(\theta_0)^T \E_0[\influencefunc(X) g(X)] \neq 0$, as otherwise the
result is trivial. For the $\bigLone$-information $\information[g,0]
\defeq \int |g| dP_0$, there exist remainders $r_0, r_\functional$ both
satisfying $r(h) = o(h)$ and
\begin{equation*}
  \tvnorm{P_0 - P_h}
  = \half |h| \information[g,0]
  + r_0(h)
  ~~ \mbox{and} ~~
  \functional(\theta(P_h))
  = \functional(\theta_0)
  + \nabla \functional(\theta_0)^T \E_0[\influencefunc(X) g(X)] h
  + r_\functional(h)
\end{equation*}
by the differentiability assumptions.  Choose $h_0 > 0$ small enough that
$|h| \le h_0$ implies that $|r_0(h)| \le \half |h| \information[g, 0]$
and $|r_\functional(h)| \le |\nabla \functional(\theta_0)^T
\E_0[\influencefunc(X) g(X)] h| / 20$, which depends only
on $\functional$ and $\Pmodel$. Then defining $\modcont(\delta) \defeq
\lossmodcont(\delta; P_0, \mc{P})$ for shorthand, we have
\begin{align*}
  \modcont(\delta)
  & \ge \sup_{|h| \le h_0}
  \left\{\loss\left(\half (\functional(\theta_0) - \functional(P_h))\right)
  \mid \tvnorm{P_0 - P_h} \le \delta \right\} \\
  & \ge \sup_{|h| \le h_0}
  \left\{\loss\left(\frac{9}{20} \nabla \functional(\theta_0)^T
  \E_0[\influencefunc(X) g(X)] h \right)
  \mid \information[g,0]|h| \le \delta \right\}.
\end{align*}
For all $\delta \le h_0 / \information[g,0]$, then, we obtain
\begin{equation*}
  \modcont(\delta)
  \ge
  \loss\left(\frac{19 \delta}{40}
  \frac{\nabla \functional(\theta_0)^T
  \E_0[\influencefunc(X) g(X)]}{\information[g,0]}\right).
\end{equation*}
Applying Theorem~\ref{theorem:modulus-of-continuity}
and
setting $\delta = \frac{1}{\sqrt{8n\diffp^2}}$ gives the result.


\section{Technical Appendices}

\subsection{Proof of Lemma~\ref{lemma:nice-families-l1-differentiable}}
\label{sec:proof-nice-families-l1-differentiable}

By the triangle inequality, we have
\begin{align}
  &\int |p_{\theta_0 + h} - p_{\theta_0} - 
  h^T \dot{\ell}_{\theta_0} p_{\theta_0}|d\mu \nonumber \\
  &\leq \underbrace{\int \left|p_{\theta_0 + h} - p_{\theta_0} - 
    \half h^T \dot{\ell}_{\theta_0} \sqrt{p_{\theta_0}}
    ({\sqrt{p_{\theta_0 + h}} + \sqrt{p_{\theta_0}}})\right|
    d\mu}_{\defeq I_1(h; \theta_0)}
  + \underbrace{\int \left| \half h^T \dot{\ell}_{\theta_0} 
    \sqrt{p_{\theta_0}} (\sqrt{p_{\theta_0 + h}} - \sqrt{p_{\theta_0}})\right|
    d\mu}_{\defeq I_2(h; \theta_0)}.
  \nonumber 
\end{align}
We show that each of the integral terms $I_1$ and $I_2$ are both
$o(\norm{h})$ as $h \to 0$.
By algebraic 
manipulation and the Cauchy--Schwarz inequality,
\begin{align*}
  I_1(h; \theta_0) &= \int |\sqrt{p_{\theta_0 + h}} + 
  \sqrt{p_{\theta_0}}| \cdot \left|\sqrt{p_{\theta_0 + h}} - 
  \sqrt{p_{\theta_0}} - \half h^T \dot{\ell}_{\theta_0} 
  \sqrt{p_{\theta_0}}\right| d\mu \nonumber \\
  &\leq \left(\int |\sqrt{p_{\theta_0 + h}} + \sqrt{p_{\theta_0}}|^2 
  d\mu \right)^\half  \cdot \left(\int \left|\sqrt{p_{\theta_0 + h}} - 
  \sqrt{p_{\theta_0}} - \half h^T \dot{\ell}_{\theta_0} 
  \sqrt{p_{\theta_0}}\right|^2 d\mu\right)^\half
\end{align*}
Jensen's inequality gives
$\int |\sqrt{p_{\theta_0 + h}} + \sqrt{p_{\theta_0}}|^2 d\mu
\leq 2 \int (p_{\theta_0 + h} + p_{\theta_0}) d\mu = 2$.
The assumption
that $\mc{P}$ is QMD at $\theta_0$ immediately yields
$I_1(h; \theta_0) = o(\norm{h})$.
To bound $I_2$, we again apply the Cauchy--Schwarz inequality, obtaining
\begin{equation*}
  2 I_2 (h; \theta_0) \leq \left(\int |h^T \dot{\ell}_{\theta_0}
  \sqrt{p_{\theta_0}} |^2 d\mu \right)^\half 
  \cdot  \left(\int |\sqrt{p_{\theta_0 + h}} - 
  \sqrt{p_{\theta_0}}|^2 d\mu \right)^\half
\end{equation*}
Since $\mc{P}$ is QMD at $\theta_0$, we have $\int |\sqrt{p_{\theta_0 + h}}
- \sqrt{p_{\theta_0}}|^2 d\mu = \int |\half h^T \dot{\ell}_{\theta_0}
\sqrt{p_{\theta_0}} |^2 d\mu + o(\norm{h}^2) = O_{\theta_0}(\norm{h}^2)$
(see~\cite[Ch.~7.2]{VanDerVaart98}).  Thus $I_2(h; \theta_0) =
O_{\theta_0}(\norm{h}^2)$, giving the lemma.

\subsection{Proof of
  Proposition~\ref{proposition:uniform-achievability-one-dim-exp}}
\label{sec:proof-uniform-achievability-one-dim-exp}

We require one additional piece of notation before we begin the proof.
Let $W_i = Z_i - V_i$ be the error in the private version of the
quantity $V_i$, so that $\E[W_i \mid V_i] = 0$, and
\begin{equation*}
  W_i = \begin{cases}
    \frac{2}{e^\diffp - 1} V_i
    - \frac{1}{e^\diffp - 1}
    & \mbox{w.p.}~ \frac{e^\diffp}{e^\diffp + 1} \\
    \frac{-2 e^\diffp}{e^\diffp - 1} V_i
    + \frac{e^\diffp}{e^\diffp - 1}
    & \mbox{w.p.}~ \frac{1}{e^\diffp + 1}.
  \end{cases}
\end{equation*}
Recall our definitions of $V_i = \indics{T(X_i) \ge \what{T}_n}$ and $Z_i$ as
the privatized version of $V_i$. Letting $\wb{Z}_n = \frac{1}{n} \sum_{i =
  1}^n Z_i$, and similarly for $\wb{V}_n$ and $\wb{W}_n$, recall also the
definition of the random variable $G_n \defeq \Psi(\what{T}_n, \theta_0) =
P_{\theta_0}(T(X) \ge \what{T}_n)$.
By mimicking the delta method, we will show that
\begin{equation}
  \sqrt{n}(\what{\theta}_n - \theta_0)
  = 2 \information[\theta_0]^{-1}
  \cdot \sqrt{n} \left(\wb{V}_n - G_n + \wb{W}_n \right) + o_P(1).
  \label{eqn:delta-method-expfam}
\end{equation}
Deferring the proof of the expansion~\eqref{eqn:delta-method-expfam},
let us show how it implies the proposition.

First, with our definition of the $W_i$, we have
\begin{equation*}
  \var(W_i \mid V_i) =
  \E[W_i^2 \mid V_i]
  = \frac{e^\diffp }{(e^\diffp - 1)^2}
  = \delta_\diffp^{-2},
\end{equation*}
so that $\wb{W}_n = \frac{1}{n} \sum_{i = 1}^n W_i$ satisfies $\sqrt{n}
\wb{W}_n \cd \normal(0, \delta_\diffp^{-2})$ by the Lindeberg CLT.  Thus,
assuming the expansion~\eqref{eqn:delta-method-expfam}, it remains to show
the weak convergence result
\begin{equation}
  \label{eqn:asymp-Z-n}
  \frac{\sqrt{n}\left(\wb{V}_n - G_n\right)}{G_n(1-G_n)} \cd \normal(0, 1).
\end{equation}
where $G_n = \Psi(\what{T}_n, \theta_0)$.
By definition, the $\{X_i\}_{i=1}^n$ are independent of $\what{T}_n$, and
hence
\begin{equation*}
  \E [V_i \mid \what{T}_n]= \Psi(\what{T}_n, \theta_0)  = G_n~~\text{and}~~ 
  \var(V_i \mid \what{T}_n) = 
  \Psi(\what{T}_n, \theta_0)(1-\Psi(\what{T}_n, \theta_0))
  = G_n(1-G_n).
\end{equation*}
The third central moments of the $V_i$ conditional on $\what{T}_n$
have the bound
\begin{equation*}
  \E \left[ \left|V_i - \E [V_i \mid \what{T}_n]\right|^3 \mid \what{T}_n \right]
  \leq  \Psi(\what{T}_n, \theta_0)(1-\Psi(\what{T}_n, \theta_0))
  = G_n(1-G_n).
\end{equation*}
Thus, we may apply the Berry-Esseen Theorem 
\cite[Thm 11.2.7]{LehmannRo05} to obtain
\begin{equation*}
  \sup_{t\in \R} 
  \left|\P\left(\frac{\sqrt{n}\left(\wb{V}_n - G_n\right)}{G_n(1- G_n)}
  \leq t \mid \what{T}_n\right)
  - \Phi(t) \right|
  \leq U_n \defeq \frac{1}{\sqrt{n G_n(1- G_n)}} \wedge 2.
\end{equation*}
Jensens's inequality then implies
\begin{equation*}
  \sup_{t \in \R}
  \left|\P\left(\frac{\sqrt{n}\left(\wb{V}_n - G_n\right)}{G_n(1- G_n)}
  \leq t\right) - \Phi(t)\right|
  \leq 
  \E\left[\sup_{t \in \R} \left|
    \P\left(\frac{\sqrt{n}\left(\wb{V}_n - G_n\right)}{G_n(1- G_n)} \leq t \mid \what{T}_n\right)
    - \Phi(t) \right|\right]
  \leq \E [U_n]
\end{equation*}

To show the convergence~\eqref{eqn:asymp-Z-n}, it is thus sufficient to show
that $\E [U_n] \to 0$ as $n \uparrow \infty$. To that end, the
following lemma on the behavior of $\Psi(t, \theta)
= P_\theta(T(X) \ge t)$ is useful.
\begin{lemma}
  \label{lemma:overlap-nonzero-lemma}
  Let $t_0 = \E_{\theta_0}[T(X)]$ and
  assume that $\var_{\theta_0} (T(X)) > 0$. Then there exist
  $\epsilon > 0$ and $c \in (0, \half)$ such that
  if $t \in [t_0 \pm \epsilon]$ and
  $\theta \in [\theta_0 \pm \epsilon]$, then
  $\Psi(t, \theta) \in [c, 1-c]$.
\end{lemma}
\begin{proof}
  By the dominated convergence theorem
  and our assumption that
  $\var_{\theta_0}(T(X)) > 0$,
  where $t_0 = \E_{\theta_0}[T(X)]$,
  we have
  \begin{equation*}
    \liminf_{t \uparrow t_0}
    \Psi(t, \theta_0)
    = P_{\theta_0}(T(X) \ge t_0) \in (0, 1)
    ~~ \mbox{and} ~~
    \limsup_{t \downarrow t_0}
    \Psi(t, \theta_0)
    = P_{\theta_0}(T(X) > t_0) \in (0, 1).
  \end{equation*}
  The fact that
  $t \mapsto \Psi(t, \theta_0)$ is non-increasing implies
  that for some $\epsilon_1 > 0, c \in (0, \frac{1}{4})$,
  we have
  $\Psi(t, \theta_0) \in [2c, 1-2c]$ for
  $t \in [t_0 - \epsilon_1, t_0 + \epsilon_1]$.
  Fix this $\epsilon_1$ and $c$.
  By~\cite[Thm 2.7.1]{LehmannRo05}, 
  we know that any $t \in \R$,
  the function $\theta \mapsto \Psi(t, \theta)$ is continuous and
  non-decreasing. Thus for any
  $\epsilon_2 > 0$, we have
  \begin{equation*}
    \Psi(t_0 + \epsilon_1, \theta_0 - \epsilon_2)
    \le \Psi(t, \theta)
    \le \Psi(t_0 - \epsilon_1, \theta_0 + \epsilon_2)
    ~~ \mbox{for}~~
    (t, \theta) \in [t_0 \pm \epsilon_1] \times [\theta_0 \pm \epsilon_2].
  \end{equation*}
  Using the continuity of $\theta \mapsto \Psi(t, \theta)$, we may
  choose $\epsilon_2 > 0$ small enough that
  \begin{equation*}
    \Psi(t, \theta) \in [ c, 1- c]~~\mbox{for}~~
    (t, \theta) \in \{t_0 - \epsilon_1, t_0 + \epsilon_1\}
    \times \{\theta_0 - \epsilon_2, \theta_0 + \epsilon_2\}.
  \end{equation*}
  The lemma 
  follows by taking $\epsilon = \epsilon_1 \wedge \epsilon_2$. 
\end{proof}

As $\var_{\theta_0} (T(X)) > 0$ by assumption,
Lemma~\ref{lemma:overlap-nonzero-lemma} and the
fact that $\what{T}_n \cp t_0$ imply
\begin{equation}
  \label{eqn:range-of-G-n}
  G_n \defeq \Psi(\what{T}_n, \theta_0)
  = P_{\theta_0}(T(X) \ge \what{T}_n)\in [c+o_P(1), 1-c + o_P(1)].
\end{equation} 
The bounds~\eqref{eqn:range-of-G-n} imply
that $G_n(1-G_n) \geq c(1-c) + o_P(1)$, so
$U_n \cp 0$.
By construction $|U_n| \leq 2$ for all $n$,
so the bounded convergence theorem implies $\E [U_n] \rightarrow 0$,
which was what we required to show the weak convergence
result~\eqref{eqn:asymp-Z-n}. 
The joint convergence in the proposition follows because
$\wb{W}_n$ and $\wb{V}_n - G_n$ are conditionally uncorrelated.

\paragraph{The delta method expansion}
We now return to demonstrate the claim~\eqref{eqn:delta-method-expfam}.  For
$p \in [0, 1]$, recall the
definition~\eqref{eqn:private-expfam-estimator}
of the function $H$, and define
\begin{equation}
  \label{eqn:def-H-n}
  H_n(p)
  \defeq H(p, \what{T}_n)
  = \inf \left\{\theta\in \R \mid P_\theta(T(X) \ge \what{T}_n) \ge p\right\},
\end{equation}
where the value is $-\infty$ or $+\infty$ for $p$ below or above the range
of $\theta \mapsto P_\theta(T(X) \ge \what{T}_n)$, respectively.  Then
$\what{\theta}_n = H_n(\wb{Z}_n)$ by
construction~\eqref{eqn:private-expfam-estimator}.  We would like to apply
Taylor's theorem and the inverse function theorem to $\what{\theta}_n -
\theta_0 = H_n(\wb{Z}_n) - \theta_0$, but this requires a few additional
steps.

By the inverse function theorem, $p \mapsto H_n(p)$ is
$\mc{C}^\infty$ on $(\inf_\theta \Psi(\what{T}_n, \theta),
\sup_\theta \Psi(\what{T}_n, \theta))$, and letting
\begin{equation*}
  \dot{\Psi}_\theta(t, \theta) = \frac{\partial}{\partial \theta}
  \Psi(t, \theta)
  = \E_\theta[\indic{T(X) \ge t} (T(X) - A'(\theta))]
  = \frac{\partial}{\partial \theta} P_\theta(T(X) \ge t),
\end{equation*}
we
have $H_n'(p) = \dot{\Psi}_\theta(\what{T}_n, H_n(p))^{-1}$ whenever $p$ is
interior to the range of $\theta \mapsto P_\theta(T(X) \ge \what{T}_n)$.  To
show that $\wb{Z}_n$ is (typically) in this range, we require a bit of
analysis on $\dot{\Psi}_\theta$.
\begin{lemma}
  \label{lemma:continuity-expfam}
  The function $(t, \theta) \mapsto \dot{\Psi}_\theta(t, \theta)
  = \E_{\theta}[\indic{T(X) \ge t} (T(X) - A'(\theta))]$
  is continuous at $(t_0, \theta_0)$, where $t_0 =
  \E_{\theta_0}[T(X)] = A'(\theta_0)$.
\end{lemma}
\noindent
To avoid disrupting the flow, we defer the proof to
Section~\ref{sec:proof-continuity-expfam}.
Now, we have that $\dot{\Psi}_{\theta}(t_0, \theta_0) = \half
\E_{\theta_0}[|T(X) - t_0|] > 0$, so Lemma~\ref{lemma:continuity-expfam}
implies there exists $\epsilon > 0$ such that
\begin{equation}
  \label{eqn:positive-expfam-deriv}
  \inf_{|t - t_0| \le \epsilon
    ,|\theta - \theta_0| \le \epsilon}
  \dot{\Psi}_{\theta}(t, \theta) \ge c > 0
\end{equation}
for some constant $c$. Thus, we obtain that
\begin{align}
  \nonumber
  \P\left(\wb{Z}_n \not \in \range(\Psi(\what{T}_n, \cdot))\right)
  & \le \P\left(\wb{Z}_n \not \in \range(\Psi(\what{T}_n, \cdot)),
  \what{T}_n \in [t_0 \pm \epsilon] \right)
  + \P\left(\what{T}_n \not \in [t_0 \pm \epsilon]\right) \\
  & \stackrel{(i)}{\le} \P\left(\wb{Z}_n \not \in
  [\Psi(\what{T}_n, \theta_0)
    \pm c \epsilon] \right)
  + o(1)
  \to 0,
  \label{eqn:z-in-range}
\end{align}
where inequality~$(i)$ follows because
$\range(\Psi(t, \cdot)) \supset [\Psi(t, \theta_0) \pm c \epsilon]$
for all $t$ such that $|t - t_0| \le \epsilon$
by condition~\eqref{eqn:positive-expfam-deriv}, and the
final convergence because $\wb{Z}_n - \Psi(\what{T}_n, \theta_0)
\cp 0$ and $\what{T}_n \cp t_0$.

We recall that for
 fixed $t$, $\theta \mapsto \Psi(t, \theta)$ is analytic on the interior
of the natural parameter space and strictly increasing at all $\theta$ for
which $\Psi(t, \theta) \in (0, 1)$ (cf.~\cite[Thm.~2.7.1,
  Thm.~3.4.1]{LehmannRo05}). Thus,
\begin{equation*}
  H_n(\Psi(\what{T}_n, \theta)) = \theta
  ~~ \mbox{whenever} ~~
  \Psi(\what{T}_n, \theta) \in (0, 1).
\end{equation*}
As $G_n = \Psi(\what{T}_n, \theta_0)
\in [c + o_P(1), 1- c+o_P(1)]$ by definition~\eqref{eqn:range-of-G-n} of
$G_n$, we obtain
\begin{equation*}
  \P \left(H_n(\Psi(\what{T}_n, \theta_0)) \neq \theta_0\right) \to 0. 
\end{equation*}
By the differentiability of $H_n$ on the interior of its
domain (i.e.\ the range of $\Psi(\what{T}_n, \cdot)$), we use 
the convergence~\eqref{eqn:z-in-range}
and Taylor's intermediate value theorem to obtain that for
some $p_n$ between $\wb{Z}_n$ and $\Psi(\what{T}_n, \theta_0)$, we have
\begin{align}
  \label{eqn:expfam-almost-done}
  \lefteqn{\sqrt{n}(\what{\theta}_n - \theta_0)
    = \sqrt{n} (\what{\theta}_n - H_n(\Psi(\what{T}_n, \theta_0)))
    + o_P(1)} \\
  & \qquad ~ = H_n'(p_n) \sqrt{n} \left(\wb{Z}_n - \Psi(\what{T}_n, \theta_0)
  \right) + o_P(1)
  = \dot{\Psi}_\theta(\what{T}_n, H_n(p_n))^{-1}
  \sqrt{n} \left(\wb{Z}_n - \Psi(\what{T}_n, \theta_0)
  \right) + o_P(1)
  \nonumber
\end{align}
as $p_n \in \interior \dom H_n$ with high probability
by~\eqref{eqn:z-in-range}.

It remains to show that $H_n(p_n) \cp \theta_0$. When $\what{T}_n \in
[t_0 \pm \epsilon]$, the growth condition~\eqref{eqn:positive-expfam-deriv}
implies
\begin{align*}
  \Psi(\what{T}_n,
  \theta_0 + \epsilon)
  = P_{\theta_0 + \epsilon}(T(X) \ge \what{T}_n)
  & \ge P_{\theta_0}(T(X) \ge \what{T}_n)
  + c \epsilon
  = \Psi(\what{T}_n, \theta_0) + c \epsilon \\
  \Psi(\what{T}_n,
  \theta_0 - \epsilon)
  = P_{\theta_0 - \epsilon}(T(X) \ge \what{T}_n)
  & \le P_{\theta_0}(T(X) \ge \what{T}_n)
  - c \epsilon
  = \Psi(\what{T}_n, \theta_0) - c \epsilon,
\end{align*}
and thus 
\begin{equation*}
  \P(|H_n(p_n) - \theta_0| \ge \epsilon)
  \le \P(|\wb{Z}_n - \Psi(\what{T}_n, \theta_0)| \ge c \epsilon)
  + \P(|\what{T}_n - t_0| \ge \epsilon)
  \to 0.
\end{equation*}
We have the convergence $\dot{\Psi}_\theta(\what{T}_n, H_n(p_n)) \cp \half
\E_{\theta_0}[|T(X) - A'(\theta_0)|] = \half \information[\theta_0]$ by the
continuous mapping theorem, and Slutsky's theorem applied to
Eq.~\eqref{eqn:expfam-almost-done} gives the delta-method
expansion~\eqref{eqn:delta-method-expfam}.

\subsubsection{Proof of Lemma~\ref{lemma:continuity-expfam}}
\label{sec:proof-continuity-expfam}

We have
\begin{align*}
  \dot{\Psi}_\theta(t_0, \theta_0)
  - \dot{\Psi}_\theta(t, \theta)
  & = \E_{\theta_0}[\indic{T(X) \ge t_0}
    (T(X) - A'(\theta_0))]
  - \E_\theta[\indic{T(X) \ge t}
    (T(X) - A'(\theta))] \\
  & \stackrel{(i)}{=} \E_{\theta_0}\left[
    \hinge{T(X) - t_0}\right]
  - \E_\theta[\indic{T(X) \ge t}
    (T(X) - t + t - A'(\theta))] \\
  & = \E_{\theta_0}\left[\hinge{T(X) - t_0}\right]
  - \E_\theta\left[\hinge{T(X) - t}\right]
  + P_\theta(T(X) \ge t) (t - A'(\theta)) \\
  & \stackrel{(ii)}{\in} \E_{\theta_0}\left[\hinge{T(X) - t_0}\right]
  - \E_\theta\left[\hinge{T(X) - t_0}\right]
  \pm |t - t_0| \pm |t - A'(\theta)|,
\end{align*}
where step~$(i)$ follows because $t_0 = A'(\theta_0) = \E_{\theta_0}[T(X)]$,
while the inclusion~$(ii)$ is a consequence of the 1-Lipschitz continuity of
$t \mapsto \hinge{t}$.  Now we use the standard
facts that $A(\theta)$ is analytic in $\theta$
and that $\theta \mapsto \E_\theta[f(X)]$ is continuous for any $f$
(cf.~\cite[Thm.~2.7.1]{LehmannRo05}) to see that for
any $\epsilon > 0$, we can choose $\delta > 0$ such that
$|t - t_0| \le \delta$ and $|\theta - \theta_0| \le \delta$ imply
\begin{equation*}
  |t - t_0| \le \epsilon, ~~
  |t - A'(\theta)| \le \epsilon, ~~ \mbox{and} ~~
  \left|
  \E_{\theta_0}\left[\hinge{T(X) - t_0}\right]
  - \E_\theta\left[\hinge{T(X) - t_0}\right]\right| \le \epsilon.
\end{equation*}

\setlength{\bibsep}{4pt}
\bibliography{bib}
\bibliographystyle{abbrvnat}

\end{document}